\documentclass[reqno]{amsart}
\usepackage{amsthm, amssymb, amsfonts}
\usepackage{lmodern}
\usepackage{color}
\usepackage{hyperref}
\usepackage[T5]{fontenc}
\usepackage{amscd,amssymb}
\usepackage[v2,cmtip]{xy}
\usepackage{mathrsfs}
\theoremstyle{plain}
\newtheorem{thm}{Theorem}[section]

\newtheorem{con}[thm]{Conjecture}
\newtheorem{corl}[thm]{Corollary}
\theoremstyle{definition}

\theoremstyle{plain}
\newtheorem{thms}{Theorem}[subsection]
\newtheorem{props}[thms]{Proposition}
\newtheorem{lems}[thms]{Lemma}

\theoremstyle{definition}
\newtheorem{defns}[thms]{Definition}
\newtheorem{rems}[thms]{Remark}
\newtheorem{notas}[thms]{Notation}
\numberwithin{equation}{section}
\allowdisplaybreaks

\begin{document}

\title[Determination of the fifth Singer transfer]
{Determination of the fifth Singer algebraic\\ transfer in some degrees}

\author{Nguy\~\ecircumflex n Sum}
\address{Department of Mathematics and Applications, S\`ai G\`on University, 273 An D\uhorn \ohorn ng V\uhorn \ohorn ng, District 5, H\`\ocircumflex\ Ch\'i Minh city, Viet Nam}

\email{nguyensum@sgu.edu.vn}


\begin{abstract}
Let $P_k$ be the graded polynomial algebra $\mathbb F_2[x_1,x_2,\ldots ,x_k]$ over the prime field $\mathbb F_2$ with two elements and the degree of each variable $x_i$ being 1, and let $GL_k$ be the general linear group over $\mathbb F_2$ which acts on $P_k$ as the usual manner. The algebra $P_k$ is considered as a module over the mod-2 Steenrod algebra $\mathcal A$.
In 1989, Singer \cite{si1} defined the $k$-th homological algebraic transfer, which is a homomorphism
$$\varphi_k :{\rm Tor}^{\mathcal A}_{k,k+d} (\mathbb F_2,\mathbb F_2) \to  (\mathbb F_2\otimes_{\mathcal A}P_k)_d^{GL_k}$$
from the homological group of the mod-2 Steenrod algebra $\mbox{Tor}^{\mathcal A}_{k,k+d} (\mathbb F_2,\mathbb F_2)$ to the subspace $(\mathbb F_2\otimes_{\mathcal A}P_k)_d^{GL_k}$ of $\mathbb F_2{\otimes}_{\mathcal A}P_k$ consisting of all the $GL_k$-invariant classes of degree $d$. 

In this paper, by using the results of the Peterson hit problem we present the proof of the fact that the Singer algebraic transfer of rank five is an isomorphism in the internal degrees $d= 20$ and $d = 30$. Our result refutes the proof for the case of $d=20$ in Ph\'uc \cite{p24}.
\end{abstract}

\maketitle

%======================================
\section{Introduction}\label{s1} 
\setcounter{equation}{0}

Let $P_k$ be the polynomial algebra $\mathbb F_2[x_1,x_2,\ldots ,x_k]$ over the field $\mathbb F_2$ with two elements, in $k$ variables $x_1, x_2, \ldots , x_k$, each variable of degree 1. It is well-known that this algebra is the mod-2 cohomology of an elementary abelian 2-group $V_k$ of rank $k$. Hence, $P_k$ is a module over the mod-2 Steenrod algebra, $\mathcal A$.
The action of $\mathcal A$ on $P_k$ can be explicitly determined by the elementary properties of the Steenrod operations $Sq^j$ and subject to the Cartan formula
$Sq^d(gh) = \sum_{j=0}^dSq^j(g)Sq^{d-j}(h),$
for $g,\, h \in P_k$ (see Steenrod and Epstein~\cite{st}).

The \textit{Peterson hit problem} asks for a minimal generating set for $P_k$ regarded as a module over the  mod-2 Steenrod algebra. Equivalently, this problem is to find a vector space basis for $QP_k := \mathbb F_2 \otimes_{\mathcal A} P_k$ in each degree $d$. Such a basis can be represented by a list of monomials of degree $d$.  This problem is completely computed for $k \leqslant 4$, unknown in general. 

Denote by $GL_k$ the general linear group over the field $\mathbb F_2$. This group acts naturally on $P_k$ by matrix substitution. The actions of $\mathcal A$ and $GL_k$ on $P_k$ commute with each other, so there is an action of $GL_k$ on $QP_k$. 

Denote $(P_k)_d$ the vector subspace of $P_k$ consisting of all the homogeneous polynomials of degree $d$ in $P_k$ and by $(QP_k)_d$ the vector subspace of $QP_k$ consisting of all the classes represented by the elements in $(P_k)_d$. 
In \cite{si1}, Singer defined the homological algebraic transfer, which is a homomorphism
$$\varphi_k :\mbox{Tor}^{\mathcal A}_{k,k+d} (\mathbb F_2,\mathbb F_2) \longrightarrow  (QP_k)_d^{GL_k}$$
from the homology of the Steenrod algebra $\mbox{Tor}^{\mathcal A}_{k,k+d}$ to the subspace of $(QP_k)_d$ consisting of all the $GL_k$-invariant classes of degree $d$. It can be a useful tool in describing the homology groups of the Steenrod algebra. By passing to the dual we get the cohomological algebraic transfer
$$\varphi_k^*: \mathbb F_2\otimes_{GL_k}\mathcal P((P_k)^*_d) \longrightarrow \mbox{Ext}_{\mathcal A}^{k, k+d}(\mathbb F_2, \mathbb F_2),$$
where $(P_k)^*$ is the dual of $P_k$ and $\mathcal P((P_k)^*)$ is the primitive subspace consisting of all elements in $(P_k)^*$ that are annihilated by every positive degree Steenrod squares.

The algebraic transfer was studied by  Boardman~\cite{bo}, Bruner, H\`a and H\uhorn ng~\cite{br}, Ch\ohorn n and H\`a~ \cite{cha,cha1,cha2},  H\uhorn ng ~\cite{hu}, H\`a ~\cite{ha}, H\uhorn ng and Qu\`ynh~ \cite{hq}, Nam~ \cite{na2}, Minami ~\cite{mi}, Ph\'uc \cite{p231,p24}, Qu\`ynh~ \cite{qh} the present author \cite{su4} and others.

Singer proved in \cite{si1} that $\varphi_k$ is an isomorphism for $k=1,2$. Boardman also showed in \cite{bo} that $\varphi_3$ is also an isomorphism.  However, $\varphi_k$ is not a monomorphism in infinitely many degrees for any $k \geqslant 4$ (see Singer \cite{si1}, Bruner, H\`a and H\uhorn ng~\cite{br}, H\uhorn ng \cite{hu}.) Singer gave the following conjecture.

\begin{con}[Singer \cite{si1}]\label{sconj} The algebraic transfer $\varphi_k$ is an epimorphism for any $k \geqslant 0$.
\end{con}

This conjecture is true for $k\leqslant 3$. Ph\'uc stated in \cite{p231} that the conjecture is also true for $k = 4$ but the computations are incomplete and the proof is not explicit. The results in \cite{p231} is only a description of the dimension for $QP_4$ in each degree when Singer conjecture is true for $k=4$. It is not a proof for this conjecture. Recently, we have proved in \cite{suw} that this conjecture is not true for $k = 5$ and the internal degree $d = 108$. This result refused a one of Ph\'uc in \cite{p24}.

In this paper, by using some results of the Peterson hit problem we prove that the Singer algebraic transfer of rank 5 is an isomorphism in the internal degrees $d = 20$ and $d = 30$. 

Consider the case $d = 20$, we have the following.

\begin{thm}\label{thm1} 
The dimension of the $\mathbb F_2$-vector spaces $(QP_5)_{20}^{GL_5}$ is one with a basis element presented by the following polynomial: 
\begin{align*}
p &= x_1x_2x_3^{3}x_4^{3}x_5^{12} + x_1x_2x_3^{3}x_4^{5}x_5^{10} + x_1x_2x_3^{3}x_4^{6}x_5^{9} + x_1x_2x_3^{3}x_4^{12}x_5^{3}\\ 
&\quad + x_1x_2x_3^{6}x_4^{3}x_5^{9} + x_1x_2x_3^{6}x_4^{9}x_5^{3} + x_1x_2^{3}x_3x_4^{3}x_5^{12} + x_1x_2^{3}x_3x_4^{5}x_5^{10}\\ 
&\quad + x_1x_2^{3}x_3x_4^{6}x_5^{9} + x_1x_2^{3}x_3x_4^{12}x_5^{3} + x_1x_2^{3}x_3^{3}x_4^{4}x_5^{9} + x_1x_2^{3}x_3^{3}x_4^{5}x_5^{8}\\ 
&\quad + x_1x_2^{3}x_3^{5}x_4^{3}x_5^{8} + x_1x_2^{3}x_3^{5}x_4^{8}x_5^{3} + x_1x_2^{6}x_3x_4^{3}x_5^{9} + x_1x_2^{6}x_3x_4^{9}x_5^{3}\\ 
&\quad + x_1^{3}x_2x_3^{4}x_4^{3}x_5^{9} + x_1^{3}x_2x_3^{4}x_4^{9}x_5^{3} + x_1^{3}x_2x_3^{5}x_4^{3}x_5^{8} + x_1^{3}x_2x_3^{5}x_4^{8}x_5^{3}\\ 
&\quad + x_1^{3}x_2^{4}x_3x_4^{3}x_5^{9} + x_1^{3}x_2^{4}x_3x_4^{9}x_5^{3} + x_1^{3}x_2^{5}x_3x_4^{3}x_5^{8} + x_1^{3}x_2^{5}x_3x_4^{8}x_5^{3}\\ 
&\quad + x_1x_2^{3}x_3^{5}x_4^{5}x_5^{6} + x_1x_2^{3}x_3^{5}x_4^{6}x_5^{5} + x_1x_2^{3}x_3^{6}x_4^{5}x_5^{5} + x_1x_2^{6}x_3^{3}x_4^{5}x_5^{5}\\ 
&\quad + x_1^{3}x_2x_3^{5}x_4^{5}x_5^{6} + x_1^{3}x_2x_3^{5}x_4^{6}x_5^{5} + x_1^{3}x_2^{5}x_3x_4^{5}x_5^{6} + x_1^{3}x_2^{5}x_3x_4^{6}x_5^{5}\\ 
&\quad + x_1^{3}x_2^{3}x_3^{4}x_4^{5}x_5^{5} + x_1^{3}x_2^{3}x_3^{5}x_4^{4}x_5^{5} + x_1^{3}x_2^{3}x_3^{5}x_4^{5}x_5^{4} + x_1^{3}x_2^{4}x_3^{3}x_4^{5}x_5^{5}\\ 
&\quad + x_1^{3}x_2^{5}x_3^{3}x_4^{4}x_5^{5} + x_1^{3}x_2^{5}x_3^{3}x_4^{5}x_5^{4} + x_1^{3}x_2^{5}x_3^{5}x_4^{3}x_5^{4} + x_1^{3}x_2^{5}x_3^{5}x_4^{4}x_5^{3}. 
\end{align*}
\end{thm}
In \cite{p24}, Ph\'uc also stated that $\dim(QP_5)_{20}^{GL_5} = 1$ but he cannot provide any basis element. Furthermore, the proof of this result in \cite{p24} is false.

From Lin \cite{wl} and Chen \cite{che}, we have $\mbox{Ext}_{\mathcal A}^{5, 25}(\mathbb F_2, \mathbb F_2) = \langle h_0g_1 = h_2e_0 \rangle$, where where $e_{0} \in \mbox{Ext}_{\mathcal A}^{4,21}(\mathbb F_2, \mathbb F_2)$, $g_1 \in \mbox{Ext}_{\mathcal A}^{4,24}(\mathbb F_2, \mathbb F_2)$ and $h_{i}$ is the Adams element in $\mbox{Ext}_{\mathcal A}^{1,2^{i}}(\mathbb F_2, \mathbb F_2)$ for $i\geqslant 0$. Combining the results of  H\`a \cite{ha} and Singer \cite{si1}, we have $\varphi_5((h_{2}e_{0})^*) = [p]$.  Hence, Theorem \ref{thm1} implies that 
$$\varphi_5: \mbox{Tor}_{5,25}^{\mathcal A}(\mathbb F_2,\mathbb F_2) \longrightarrow (QP_5)_{20}^{GL_5}$$
is an isomorphism. So, we get the following.

\begin{corl}\label{pthm1} Singer's conjecture is true for $k=5$ and the internal degree $20$.
\end{corl}
This result is also stated in \cite[Corollary 1.7]{p24}, but it is derived from a result with a false proof.

Consider the case $d = 30$. By using the explicit results for a basis of $(QP_5)_{30}$ in our work \cite{smo}, we prove the following. 
\begin{thm}\label{thm2} 
The dimension of the $\mathbb F_2$-vector spaces $(QP_5)_{30}^{GL_5}$ is one with a basis element presented by the following polynomial which is explicitly determined as follows: 
\begin{align*}
q &= x_4^{15}x_5^{15} + x_3^{15}x_5^{15} + x_3^{15}x_4^{15} + x_2^{15}x_5^{15} + x_2^{15}x_4^{15} + x_2^{15}x_3^{15} + x_1^{15}x_5^{15} + x_1^{15}x_4^{15}\\ 
&\quad + x_1^{15}x_3^{15} + x_1^{15}x_2^{15} + x_3x_4^{14}x_5^{15} + x_3x_4^{15}x_5^{14} + x_3^{15}x_4x_5^{14} + x_2x_4^{14}x_5^{15}\\ 
&\quad + x_2x_4^{15}x_5^{14} + x_2x_3^{14}x_5^{15} + x_2x_3^{14}x_4^{15} + x_2x_3^{15}x_5^{14} + x_2x_3^{15}x_4^{14} + x_2^{15}x_4x_5^{14}\\ 
&\quad + x_2^{15}x_3x_5^{14} + x_2^{15}x_3x_4^{14} + x_1x_4^{14}x_5^{15} + x_1x_4^{15}x_5^{14} + x_1x_3^{14}x_5^{15} + x_1x_3^{14}x_4^{15}\\ 
&\quad + x_1x_3^{15}x_5^{14} + x_1x_3^{15}x_4^{14} + x_1x_2^{14}x_5^{15} + x_1x_2^{14}x_4^{15} + x_1x_2^{14}x_3^{15} + x_1x_2^{15}x_5^{14}\\ 
&\quad + x_1x_2^{15}x_4^{14} + x_1x_2^{15}x_3^{14} + x_1^{15}x_4x_5^{14} + x_1^{15}x_3x_5^{14} + x_1^{15}x_3x_4^{14} + x_1^{15}x_2x_5^{14}\\ 
&\quad + x_1^{15}x_2x_4^{14} + x_1^{15}x_2x_3^{14} + x_3^{3}x_4^{13}x_5^{14} + x_2^{3}x_4^{13}x_5^{14} + x_2^{3}x_3^{13}x_5^{14} + x_2^{3}x_3^{13}x_4^{14}\\ 
&\quad + x_1^{3}x_4^{13}x_5^{14} + x_1^{3}x_3^{13}x_5^{14} + x_1^{3}x_3^{13}x_4^{14} + x_1^{3}x_2^{13}x_5^{14} + x_1^{3}x_2^{13}x_4^{14} + x_1^{3}x_2^{13}x_3^{14}\\ 
&\quad + x_2x_3x_4^{14}x_5^{14} + x_2x_3^{14}x_4x_5^{14} + x_2^{3}x_3^{5}x_4^{10}x_5^{12} + x_1x_3x_4^{14}x_5^{14} + x_1x_3^{14}x_4x_5^{14}\\ 
&\quad + x_1x_2x_4^{14}x_5^{14} + x_1x_2x_3^{14}x_5^{14} + x_1x_2x_3^{14}x_4^{14} + x_1x_2^{14}x_4x_5^{14} + x_1x_2^{14}x_3x_5^{14}\\ 
&\quad + x_1x_2^{14}x_3x_4^{14} + x_1^{3}x_3^{5}x_4^{10}x_5^{12} + x_1^{3}x_2^{5}x_4^{10}x_5^{12} + x_1^{3}x_2^{5}x_3^{10}x_5^{12} + x_1^{3}x_2^{5}x_3^{10}x_4^{12}\\ 
&\quad + x_2x_3^{2}x_4^{13}x_5^{14} + x_2x_3^{3}x_4^{12}x_5^{14} + x_2x_3^{3}x_4^{14}x_5^{12} + x_2^{3}x_3x_4^{12}x_5^{14} + x_2^{3}x_3x_4^{14}x_5^{12}\\ 
&\quad + x_2^{3}x_3^{13}x_4^{2}x_5^{12} + x_1x_3^{2}x_4^{13}x_5^{14} + x_1x_3^{3}x_4^{12}x_5^{14} + x_1x_3^{3}x_4^{14}x_5^{12} + x_1x_2^{2}x_4^{13}x_5^{14}\\ 
&\quad + x_1x_2^{2}x_3^{13}x_5^{14} + x_1x_2^{2}x_3^{13}x_4^{14} + x_1x_2^{3}x_4^{12}x_5^{14} + x_1x_2^{3}x_4^{14}x_5^{12} + x_1x_2^{3}x_3^{12}x_5^{14}\\ 
&\quad + x_1x_2^{3}x_3^{12}x_4^{14} + x_1x_2^{3}x_3^{14}x_5^{12} + x_1x_2^{3}x_3^{14}x_4^{12} + x_1^{3}x_3x_4^{12}x_5^{14} + x_1^{3}x_3x_4^{14}x_5^{12}\\ 
&\quad + x_1^{3}x_3^{13}x_4^{2}x_5^{12} + x_1^{3}x_2x_4^{12}x_5^{14} + x_1^{3}x_2x_4^{14}x_5^{12} + x_1^{3}x_2x_3^{12}x_5^{14} + x_1^{3}x_2x_3^{12}x_4^{14}\\ 
&\quad + x_1^{3}x_2x_3^{14}x_5^{12} + x_1^{3}x_2x_3^{14}x_4^{12} + x_1^{3}x_2^{13}x_4^{2}x_5^{12} + x_1^{3}x_2^{13}x_3^{2}x_5^{12} + x_1^{3}x_2^{13}x_3^{2}x_4^{12}\\ 
&\quad + x_2x_3^{2}x_4^{12}x_5^{15} + x_2x_3^{2}x_4^{15}x_5^{12} + x_2x_3^{15}x_4^{2}x_5^{12} + x_2^{15}x_3x_4^{2}x_5^{12} + x_1x_3^{2}x_4^{12}x_5^{15}\\ 
&\quad + x_1x_3^{2}x_4^{15}x_5^{12} + x_1x_3^{15}x_4^{2}x_5^{12} + x_1x_2^{2}x_4^{12}x_5^{15} + x_1x_2^{2}x_4^{15}x_5^{12} + x_1x_2^{2}x_3^{12}x_5^{15}\\ 
&\quad + x_1x_2^{2}x_3^{12}x_4^{15} + x_1x_2^{2}x_3^{15}x_5^{12} + x_1x_2^{2}x_3^{15}x_4^{12} + x_1x_2^{15}x_4^{2}x_5^{12} + x_1x_2^{15}x_3^{2}x_5^{12}\\ 
&\quad + x_1x_2^{15}x_3^{2}x_4^{12} + x_1^{15}x_3x_4^{2}x_5^{12} + x_1^{15}x_2x_4^{2}x_5^{12} + x_1^{15}x_2x_3^{2}x_5^{12} + x_1^{15}x_2x_3^{2}x_4^{12}\\ 
&\quad + x_1x_2x_3^{2}x_4^{14}x_5^{12} + x_1x_2x_3^{6}x_4^{10}x_5^{12} + x_1x_2x_3^{14}x_4^{2}x_5^{12} + x_1x_2^{2}x_3^{3}x_4^{12}x_5^{12}\\ 
&\quad + x_1x_2^{2}x_3^{4}x_4^{8}x_5^{15} + x_1x_2^{2}x_3^{4}x_4^{9}x_5^{14} + x_1x_2^{2}x_3^{4}x_4^{15}x_5^{8} + x_1x_2^{2}x_3^{5}x_4^{8}x_5^{14}\\ 
&\quad + x_1x_2^{2}x_3^{5}x_4^{10}x_5^{12} + x_1x_2^{2}x_3^{5}x_4^{14}x_5^{8} + x_1x_2^{2}x_3^{12}x_4x_5^{14} + x_1x_2^{2}x_3^{13}x_4^{2}x_5^{12}\\ 
&\quad + x_1x_2^{2}x_3^{15}x_4^{4}x_5^{8} + x_1x_2^{3}x_3^{4}x_4^{8}x_5^{14} + x_1x_2^{3}x_3^{4}x_4^{14}x_5^{8} + x_1x_2^{3}x_3^{6}x_4^{12}x_5^{8}\\ 
&\quad + x_1x_2^{3}x_3^{12}x_4^{2}x_5^{12} + x_1x_2^{3}x_3^{14}x_4^{4}x_5^{8} + x_1x_2^{14}x_3x_4^{2}x_5^{12} + x_1x_2^{15}x_3^{2}x_4^{4}x_5^{8}\\ 
&\quad + x_1^{3}x_2x_3^{4}x_4^{8}x_5^{14} + x_1^{3}x_2x_3^{4}x_4^{14}x_5^{8} + x_1^{3}x_2x_3^{6}x_4^{12}x_5^{8} + x_1^{3}x_2x_3^{12}x_4^{2}x_5^{12}\\ 
&\quad + x_1^{3}x_2x_3^{14}x_4^{4}x_5^{8} + x_1^{3}x_2^{5}x_3^{2}x_4^{12}x_5^{8} + x_1^{3}x_2^{5}x_3^{10}x_4^{4}x_5^{8} + x_1^{3}x_2^{13}x_3^{2}x_4^{4}x_5^{8}\\ 
&\quad + x_1^{15}x_2x_3^{2}x_4^{4}x_5^{8} + x_1^{3}x_2^{5}x_3^{6}x_4^{6}x_5^{10}.
\end{align*}
\end{thm}
By using a computer program with an algorithm of the mathematics system SAGEMATH, T\'in proved in \cite[Page 1923]{Tur} that $\dim(QP_5)_{30} = 840$. This dimensional result is presented again in Ph\'uc \cite[Appendix]{pp21} but there are no any details of the computations for a basis of this space. So, it is not a solution of the hit problem because dimensions are not the object of study of the hit problem. When the basis cannot be found, the dimensional results are meaningless. In \cite{smo} we present a solution of the hit problem by explicitly computing a basis of $(QP_5)_{30}$. This result confirms the accuracy of dimensional result for $(QP_5)_{30}$ in T\'in work \cite{Tur}. 

Ph\'uc \cite{pp21} also stated that by using this dimensional result he proved that $\dim(QP_5)_{30}^{GL_5}=1$ but this result must be proven by using a basis, it cannot be proven by using a dimensional result. So, his statement is only a prediction. 

From Lin \cite{wl} and Chen \cite{che}, we have $\mbox{Ext}_{\mathcal A}^{5,35}(\mathbb F_2, \mathbb F_2) = \langle h_0^3h_4^2  \rangle$. Hence, using a result of Singer \cite{si1}, we have $\varphi_5((h_0^3h_4^2)^*) = [q]$. Hence, Theorem \ref{thm2} implies that 
$$\varphi_5: \mbox{Tor}_{5,35}^{\mathcal A}(\mathbb F_2,\mathbb F_2) \longrightarrow (QP_5)_{30}^{GL_5}$$
is an isomorphism. So, we get the following.

\begin{corl} Singer's conjecture is true for $k=5$ and the internal degree $d=30$.
\end{corl}
This result is also stated in \cite{pp21} but it is based on an unconfirmed result. 
  
\smallskip
The paper is organized as follows. In Section \ref{s2}, we recall some notions and results on the admissible monomials in $P_k$, criterion of Singer on the hit monomials and some needed notations. In Section \ref{s3}, we present a structure of the space $(QP_5)_{20}$ and prove Theorem \ref{thm1}. Theorem \ref{thm2} is proved in Section \ref{s4}. Finally, in Section \ref{s5}, we present some needed data for the proofs of the main results.

%====================================
\section{Preliminaries on the hit problem}\label{s2}
\setcounter{equation}{0}

In this the section we present some needed notions and results for studying the hit problem such as the weight vector of a monomial, admissible monomial, criterion of Singer for hit monomial from the works of Kameko~\cite{ka}, Singer \cite{si2} and our work \cite{su1,su5}.

\subsection{The admissible monomials and Singer's criterion on hit monomials}\

\medskip
\begin{defns} Let $u =x_1^{c_1}x_2^{c_2}\ldots x_k^{c_k}$ be a monomial  in $P_k$. Denote $\nu_j(w) = c_j, 1 \leqslant j \leqslant k$. We define 
\begin{align*} 
\omega(u)&=(\omega_1(u),\ldots , \omega_t(u), \ldots),\ \
\sigma(u) = (\nu_1(u),\nu_2(u),\ldots ,\nu_k(u)),
\end{align*}
where $\omega_t(u) = \sum_{1\leqslant j \leqslant k} \alpha_{t-1}(\nu_j(w)),\ t \geqslant 1.$
The sequences $\sigma(u)$ is called the exponent vector and $\omega(u)$ is called the weight vector of $u$.  
	
A a sequence of non-negative integers $\omega= (\omega_1,\ldots , \omega_s, \ldots)$ is called a weight vector if $\omega_s = 0$ for $s \gg 0$. We define $\deg \omega = \sum_{s > 0}2^{s-1}\omega_s$, the length $\ell(\omega) = \max\{t : \omega_t >0\}$. We denote $\omega= (\omega_1,\ldots , \omega_r)$ if $\omega_s = 0$ for $s > r$.
	
For weight vectors $\omega= (\omega_1,\ldots , \omega_t, \ldots)$ and $\xi= (\xi_1,\ldots , \xi_t, \ldots)$, we define the concatenation of weight vectors $\omega|\xi = (\omega_1,\ldots , \omega_r,\xi_1,\xi_2,\ldots)$ if $\ell(\omega) = r$ and $(b)|^r = (b)|(b)|\ldots|(b)$, ($r$ times of $(b)$'s), where $b,\, r$ are positive integers.
\end{defns}
We order the sets of weight vectors and exponent vectors by the left lexicographical order.  

For any weight vector, we denote
\begin{align*}
&P_k(\omega) = \langle u \in P_k: \deg u = \deg\omega,\mbox{ and } \omega(u)\leqslant \omega\rangle,\\ 
&P_k^-(\omega) = \langle u \in P_k(\omega): \omega(u) < \omega\rangle.
\end{align*} 

\begin{defns} Suppose $\omega$ is a weight vector and $f,\, f_1,\, f_2$ are polynomials of the same degree in $P_k$. We define
	
i) $f_1 \equiv f_2$ if $f_1+f_2 \in \mathcal A^+P_k$ and $f$ is called hit if $f \equiv 0$.
	
ii) $f_1 \equiv_{\omega} f_2$ if $f_1 + f_2 \in \mathcal A^+P_k+P_k^-(\omega)$. 
\end{defns}

It is easy to see that $\equiv$ and $\equiv_{\omega}$ are equivalence relations. We set 
$$QP_k(\omega)= P_k(\omega)/ ((P_k^-(\omega)+\mathcal A^+P_k\cap P_k(\omega))).$$   

\begin{props}[\cite{su3}] For a weight vector $\omega$, the vector space $QP_k(\omega)$ is an $GL_k$-module. 
\end{props}

\begin{defns} 
Let $u,\, v$ be monomials in $P_k$ with $\deg u = \deg v$, we say that $u < v$ if one of the following conditions satisfies:  

i) $\omega (u) < \omega(v)$;
	
ii) $\omega (u) = \omega(v)$ and $\sigma(u) < \sigma(v).$
\end{defns}

\begin{defns}\label{dfnad}
A monomial $w$ in $P_k$ is said to be inadmissible if there are monomials $w_1,w_2,\ldots, w_t$ such that $w_s<w$ for $s=1,2,\ldots , t$ and $w + \sum_{s=1}^tw_s \in \mathcal A^+P_k.$ 
A monomial $w$ is called admissible	if it is not inadmissible.
\end{defns} 

It is clear that, the set of all admissible monomials of degree $d$ in $P_k$ is a minimal set of generators for $\mathcal{A}$-module $P_k$ in degree $d$.

\begin{defns}\label{spi}  If $z$ is a monomial in $P_k$ such that $\nu_j(z)=2^{s_t}-1$ with $s_t$ a non-negative integer for $t=1,2, \ldots , k$, then $z$ is called a spike. If $s_1>s_2>\ldots >s_{u-1}\geqslant s_u>0$ and $s_t=0$ for $t>u$, then it is called the minimal spike.
\end{defns}

\begin{thms}[Singer~{\cite{si2}}]\label{dlsig} Suppose $y$ is a monomial of degree $d$ in $P_k$ such that $\mu(d) \leqslant k$ and $z$ is the minimal spike of degree $d$. If $\omega(y) < \omega(z)$, then $y$ is hit. 
\end{thms}

\subsection{Some other tools and notations}\

\medskip
We set $P_k^+ = \langle\{y\in P_k :  \nu_s(y)>0, \mbox{ for all } s\}\rangle$ and $P_k^0 =\langle\{y\in P_k : \nu_s(y)=0 \mbox{ for some } s\}\rangle$. The spaces $P_k^0$ and $P_k^+$ are the $\mathcal{A}$-submodules of $P_k$. We denote $QP_k^0 = P_k^0/\mathcal A^+P_k^0$ and  $QP_k^+ = P_k^+/\mathcal A^+P_k^+$. Then we have $QP_k =QP_k^0 \oplus  QP_k^+.$

Set 
$\mathcal N_k =\{(i;I) : I=(s_1,s_2,\ldots,s_t),1 \leqslant  i < s_1 <  \ldots < s_t\leqslant  k,\ 0\leqslant t <k\}.$
For $(i;I) \in \mathcal N_k$, we define the $\mathcal A$-homomorphism of algebras $p_{(i;I)}: P_k \to P_{k-1}$  by setting
\begin{equation}\label{ct23}
p_{(i;I)}(x_s) =\begin{cases} x_s, &\mbox{ if } 1 \leqslant s < i,\\
\sum_{v\in I}x_{v-1}, &\mbox{ if }  s = i,\\  
x_{s-1},&\mbox{ if } i< s \leqslant k.
\end{cases}
\end{equation}

\begin{lems}[\cite{sp}]\label{bdm} If $y$ is a monomial in $P_k$, then $p_{(i;I)}(y) \in P_{k-1}(\omega(y)).$
\end{lems}

From Lemma \ref{bdm}, we can see that if $\omega$ is a weight vector and $y \in P_k(\omega)$, then $p_{(j;J)}(y) \in P_{k-1}(\omega)$. Hence, we have the homomorphisms  
\begin{align*}
&p_{(i;I)}^{(n)} :(QP_k)_n\longrightarrow (QP_{k-1})_n,\ \ p_{(i;I)}^{(\omega)} :QP_k(\omega)\longrightarrow QP_{k-1}(\omega),
\end{align*} 
where $n = \deg\omega$. We set 
\begin{align*}
&\mbox{  } (\widetilde {SF}_k)_n = \bigcap_{(i;I)\in \mathcal N_k}  \mbox{Ker}(p_{(i;I)}^{(n)}),\ \ \widetilde {SF}_k(\omega) = \bigcap_{(i;I)\in \mathcal N_k}  \mbox{Ker}(p_{(i;I)}^{(\omega)}),\\
&\widetilde{QP}_k(\omega) = QP_k(\omega)/\widetilde {SF}_k(\omega), \mbox{ } (\widetilde{QP}_k)_n = (QP_k)_d/(\widetilde {SF}_k)_n.
\end{align*}
We note that $(\widetilde {SF}_k)_n$ and $\widetilde {SF}_k(\omega)$ are respectively the subspaces of $(QP_k^+)_n$ and $QP_k^+(\omega)$.

\begin{notas} We denote $[g]$ the class in $QP_k$ represented by $g\in P_k$. If $g \in  P_k(\omega)$, we denote $[g]_\omega$ the class in $QP_k(\omega)$ represented by $g$. 
If $R$ is a subset of $P_k$, then we denote $[R] = \{[g] : g \in R\}$. If $R \subset P_k(\omega)$, then we set $[R]_\omega = \{[g]_\omega : g \in L\}$. We denote $|R|$ the cardinal of a set $R$.
	
Denote by $B_{k}(n)$ the set of all admissible monomials of degree $n$ in $P_k$ for any nonnegative integer $n$. Set 
$B_{k}^+(n) = B_{k}(n)\cap P_k^+,\ B_{k}^0(n) = B_{k}(n)\cap P_k^0.$
If $\omega$ is a weight vector of degree $n$, we denote 
$B_k(\omega) = B_{k}(n)\cap P_k(\omega),\ B_k^+(\omega) = B_{k}^+(n)\cap P_k(\omega),\ QP_k^+(\omega) := QP_k(\omega)\cap QP_k^+.$

For any subgroup $G \subset GL_k$ and $S \subset P_k(\omega)$, we denote
$$[G(S)]_\omega = \langle [gs]_\omega : g \in G, s \in S\rangle \subset QP_k(\omega).$$ 
Obviously, $[G(S)]_\omega$ is an $G$-submodule of $QP_k(\omega)$. If $\omega$ is a minimal weight vector, then we denote $[G(S)]_\omega = [G(S)]\subset QP_k$.

If $\mathbb J$ is an index set and $\gamma_j \in \mathbb F_2$ with $j \in \mathbb J$, then we denote $\gamma_{\mathbb J} = \sum_{j\in \mathbb J}\gamma_j$.
\end{notas}

For a sequence $T= (t_1, t_2, \ldots, t_r),\, 1 \leqslant t_1 <\ldots < t_r \leqslant k$, we define a monomorphism $\theta_T: P_r \to P_k$ of algebras by setting 
\begin{equation}\label{ctbs}
\theta_T(x_i) = x_{t_i} \ \mbox{ for } \ 1 \leqslant i \leqslant r.
\end{equation} 
Obviously, $\theta_T$ is an $\mathcal A$-homomorphism. If $\omega$ is a weight vector of degree $n$, then 
\[Q\theta_T(P_r^+)(\omega) \cong  QP_r^+(\omega)\mbox{ and } (Q\theta_T(P_r^+))_n \cong (QP_r^+)_n\] 
for $1 \leqslant r \leqslant k$. Here, $Q\theta_T(P_r^+) = \theta_T(P_r^+)/\mathcal A^+\theta_T(P_r^+)$. Then we have
\begin{equation}\label{ctbs2}
B_k(\omega) = \bigcup_{\mu(n) \leqslant r\leqslant k,\atop \ell(T)=r} \theta_T(B_r^+(\omega)),\ B_k(n) = \bigcup_{\mu(n) \leqslant r\leqslant k,\atop \ell(T)=r} \theta_T(B_r^+(n)).
\end{equation}

%=========================
\section{A structure of $(QP_5)_{20}$ and a proof of Theorem \ref{thm1}}\label{s3}
\setcounter{equation}{0}

By using a computer program with an algorithm of the mathematics system SAGEMATH, T\'in \cite{Tij} proved that $\dim(QP_5)_{20} = 641$. This result is also repeated in Ph\'uc \cite{p24} but the computation for a basis of $(QP_5)_{20}$ is incomplete and the proof of this result in \cite{p24} is false.

Firstly, we recall the results in our work \cite{suz} on the admissible monomials of degree 20 in $P_5$. Recall from \cite{suz} and \cite{p24} that if $x$ is an admissible monomial of degree 20 in $P_5$, then
\begin{equation*}
\omega(x) = (4)|(2)|(1)|^2 \mbox{ or } \omega(x) = (4)|(2)|(3) \mbox{ or } \omega(x) =(4)|^{2}|(2).
\end{equation*}
From this result we get 
$$
(QP_5)_{20}	\cong QP_5((4)|(2)|(1)|^2) \bigoplus QP_5((4)|(2)|(3)) \bigoplus QP_5((4)|^{2}|(2)).
$$

In \cite{p24}, the author says that
\begin{align*}&\dim QP_5((4)|(2)|(1)|^2) = 450,\\ &\dim QP_5((4)|(2)|(3)) = 70,\\ &\dim QP_5((4)|^{2}|(2)) = 121.
\end{align*}
However, the detailed computation is only presented for the space $QP_5((4)|(2)|(3))$ but the proof of it is false. 

The set of all admissible monomials of the degree 20 in $P_5$ is explicitly presented in Subsection \ref{s50}.

Consider the spaces $(\widetilde {SF}_k)_n$ and $\widetilde {SF}_k(\omega)$ as defined in \cite[Page 5]{suw}.

\subsection{Computation of $(\widetilde{SF}_5)_{20}$}\

\smallskip
We determined $\widetilde{SF}_5{((4)|(2)|(1)|^2)}$. The following is announced \cite{suz} with no the proof.
\begin{props}\label{mdsp1} We have 
$$\widetilde{SF}_5{((4)|(2)|(1)|^2)} = \langle [g_u]: 1 \leqslant u \leqslant 10\rangle,$$
where the polynomials $g_u$ are explicitly determined in Subsection \ref{s51}.
\end{props}
\begin{proof} By a direct computation, it is not difficult to verify that $p_{(i;I)}(g_u) \equiv 0$ for all $(i;I)\in \mathcal N_5$. Hence, $[g_u] \in \widetilde{SF}_5{((4)|(2)|(1)|^2)}$ for $1 \leqslant u \leqslant 10$.
The set $\{[a_t] :1\leqslant t \leqslant 225\}$ is a basis of  $QP_5^+((4)|(2)|(1)|^2)$ (see Subsection \ref{s50}), hence if $[\theta] \in \widetilde {SF}_{5}((4)|(2)|(1)|^2)\subset QP_5^+((4)|(2)|(1)|^2)$, then we have $\theta \equiv \sum_{1\leqslant t \leqslant 225}\bar\gamma_ta_{t}$, where $\bar\gamma_t \in \mathbb F_2$. The leading monomial of $g_u$ is $a_{(u+215)},\, 1\leqslant u \leqslant 10$. Hence, we get 
\begin{equation}\label{ctd612}
\widetilde\theta := \theta + \sum_{1\leqslant u \leqslant 10}\bar\gamma_{(u+215)}g_u\equiv \sum_{1\leqslant t \leqslant 215}\gamma_ta_{t},
\end{equation}
where $\gamma_t \in \mathbb F_2$, $1\leqslant t \leqslant 215$, and $p_{(i;I)}(\widetilde\theta) \equiv 0$ for all $(i;I) \in \mathcal N_5$. For simplicity, we denote $\gamma_{\mathbb J} = \sum_{j \in \mathbb J}\gamma_j$ for any $\mathbb J \subset \{i\in \mathbb N:1\leqslant i \leqslant 215\}$.

Let $u_s,\, 1\leqslant s \leqslant 45$, be as in Subsection \ref{s50}. Apply $p_{(1;j)}$, $2 \leqslant j \leqslant 5,$ to (\ref{ctd612}) to obtain
\begin{align*}
p_{(1;2)}(\widetilde\theta) &\equiv \gamma_{10}v_{16} + \gamma_{46}v_{17} + \gamma_{47}v_{18} + \gamma_{48}v_{19} + \gamma_{49}v_{20} + \gamma_{11}v_{21} + \gamma_{50}v_{22}\\ 
&\quad + \gamma_{51}v_{23} + \gamma_{52}v_{24} + \gamma_{53}v_{25} + \gamma_{54}v_{26} + \gamma_{55}v_{27} + \gamma_{56}v_{28} + \gamma_{57}v_{29}\\ 
&\quad + \gamma_{58}v_{30} + \gamma_{59}v_{31} + \gamma_{60}v_{32} + \gamma_{12}v_{33} + \gamma_{\{97,169\}}v_{34} + \gamma_{\{98,170\}}v_{35}\\ 
&\quad + \gamma_{99}v_{36} + \gamma_{\{100,171\}}v_{37} + \gamma_{\{101,172\}}v_{38} + \gamma_{\{102,173\}}v_{39} + \gamma_{103}v_{40}\\ 
&\quad + \gamma_{104}v_{41} + \gamma_{\{105,174\}}v_{42} + \gamma_{121}v_{43} + \gamma_{122}v_{44} + \gamma_{123}v_{45} \equiv 0,\\ 
p_{(1;3)}(\widetilde\theta) &\equiv \gamma_{5}v_{4} + \gamma_{20}v_{5} + \gamma_{21}v_{6} + \gamma_{22}v_{7} + \gamma_{23}v_{8} + \gamma_{6}v_{9} + \gamma_{\{34,141\}}v_{10}\\ 
&\quad + \gamma_{\{35,142\}}v_{11} + \gamma_{\{36,143\}}v_{12} + \gamma_{\{37,144\}}v_{13} + \gamma_{\{44,155,206\}}v_{14}\\ 
&\quad + \gamma_{\{45,156,207\}}v_{15} + \gamma_{71}v_{22} + \gamma_{72}v_{23} + \gamma_{73}v_{24} + \gamma_{78}v_{25} + \gamma_{79}v_{26}\\ 
&\quad + \gamma_{80}v_{27} + \gamma_{81}v_{28} + \gamma_{\{88,164,185\}}v_{29} + \gamma_{\{89,165,186\}}v_{30} + \gamma_{\{92,185\}}v_{31}\\ 
&\quad + \gamma_{\{93,186\}}v_{32} + \gamma_{\{96,168,213\}}v_{33} + \gamma_{112}v_{38} + \gamma_{113}v_{39} + \gamma_{116}v_{40}\\ 
&\quad + \gamma_{117}v_{41} + \gamma_{\{120,192\}}v_{42} + \gamma_{9}v_{45} \equiv 0,\\
p_{(1;4)}(\widetilde\theta) &\equiv \gamma_{2}v_{1} + \gamma_{\{17,127\}}v_{2} + \gamma_{\{19,131,199\}}v_{3} + \gamma_{25}v_{5} + \gamma_{27}v_{6} + \gamma_{\{29,138,148\}}v_{7}\\ 
&\quad + \gamma_{\{31,148\}}v_{8} + \gamma_{\{33,140,205\}}v_{9} + \gamma_{39}v_{11} + \gamma_{41}v_{12} + \gamma_{\{43,154\}}v_{13}\\ 
&\quad + \gamma_{4}v_{15} + \gamma_{62}v_{17} + \gamma_{64}v_{18} + \gamma_{\{66,161,178\}}v_{19} + \gamma_{\{68,178\}}v_{20}\\ 
&\quad + \gamma_{\{70,163,212\}}v_{21} + \gamma_{75}v_{23} + \gamma_{\{77,184\}}v_{24} + \gamma_{83}v_{26} + \gamma_{85}v_{27}\\ 
&\quad + \gamma_{\{87,167,184\}}v_{28} + \gamma_{91}v_{30} + \gamma_{95}v_{32} + \gamma_{107}v_{35} + \gamma_{109}v_{36}\\ 
&\quad + \gamma_{\{111,191\}}v_{37} + \gamma_{115}v_{39} + \gamma_{119}v_{41} + \gamma_{8}v_{44} \equiv 0,\\
p_{(1;5)}(\widetilde\theta) &\equiv \gamma_{\{16,126,198\}}v_{1} + \gamma_{\{18,130\}}v_{2} + \gamma_{1}v_{3} + \gamma_{\{24,137,204\}}v_{4} + \gamma_{\{26,147\}}v_{5}\\ 
&\quad + \gamma_{\{28,139,147\}}v_{6} + \gamma_{30}v_{7} + \gamma_{32}v_{8} + \gamma_{\{38,153\}}v_{10} + \gamma_{40}v_{11} + \gamma_{42}v_{12}\\ 
&\quad + \gamma_{3}v_{14} + \gamma_{\{61,160,211\}}v_{16} + \gamma_{\{63,177\}}v_{17} + \gamma_{\{65,162,177\}}v_{18} + \gamma_{67}v_{19}\\ 
&\quad + \gamma_{69}v_{20} + \gamma_{\{74,183\}}v_{22} + \gamma_{76}v_{23} + \gamma_{\{82,166,183\}}v_{25} + \gamma_{84}v_{26}\\ 
&\quad + \gamma_{86}v_{27} + \gamma_{90}v_{29} + \gamma_{94}v_{31} + \gamma_{\{106,190\}}v_{34} + \gamma_{108}v_{35} + \gamma_{110}v_{36}\\ 
&\quad + \gamma_{114}v_{38} + \gamma_{118}v_{40} + \gamma_{7}v_{43} \equiv 0.
\end{align*}
	
By computing from above relations we have
\begin{equation}\label{c61}
\gamma_t = 0 \mbox{ for } t \in \mathbb L_1,\ \gamma_u= \gamma_u \mbox{ for } (u,v) \in \mathbb M_1,
\end{equation}
where $\mathbb L_1 = \{$1,\, 2,\, 3,\, 4,\, 5,\, 6,\, 7,\, 8,\, 9,\, 10,\, 11,\, 12,\, 20,\, 21,\, 22,\, 23,\, 25,\, 27,\, 30,\, 32,\, 39,\, 40,\, 41,\, 42,\, 46,\, 47,\, 48,\, 49,\, 50,\, 51,\, 52,\, 53,\, 54,\, 55,\, 56,\, 57,\, 58,\, 59,\, 60,\, 62,\, 64,\, 67,\, 69,\, 71,\, 72,\, 73,\, 75,\, 76,\, 78,\, 79,\, 80,\, 81,\, 83,\, 84,\, 85,\, 86,\, 90,\, 91,\, 94,\, 95,\, 99,\, 103,\, 104,\, 107,\, 108,\, 109,\, 110,\, 112,\, 113,\, 114,\, 115,\, 116,\, 117,\, 118,\, 119,\, 121,\, 122,\, 123$\}$ and $\mathbb M_1 = \{$(17,127),\, (18,130),\, (26,147),\, (31,148),\, (34,141),\, (35,142),\, (36,143),\, (37,144),\, (38,153),\, (43,154),\, (63,177),\, (68,178),\, (74,183),\, (77,184),\, (92,185),\, (93,186),\, (97,169),\, (98,170),\, (100,171),\, (101,172),\, (102,173),\, (105,174),\, (106,190),\, (111,191),\, (120,192)$\}$. 

By applying the homomorphisms $p_{(2;u)}, u = 3,4,5$ to \eqref{ctd612} and using \eqref{c61}, we obtain
\begin{align*}
p_{(2;3)}(\widetilde\theta) &\equiv \gamma_{\{34,97\}}v_{10} + \gamma_{\{35,98\}}v_{11} + \gamma_{36}v_{12} + \gamma_{\{37,100\}}v_{13} + \gamma_{\{44,92\}}v_{14}\\ 
&\quad + \gamma_{\{45,93\}}v_{15} + \gamma_{134}v_{22} + \gamma_{135}v_{23} + \gamma_{136}v_{24} + \gamma_{\{34,97\}}v_{25}\\ 
&\quad + \gamma_{\{35,98\}}v_{26} + \gamma_{36}v_{27} + \gamma_{\{37,100\}}v_{28} + \gamma_{\{101,151,164,181\}}v_{29}\\ 
&\quad + \gamma_{\{102,152,165,182\}}v_{30} + \gamma_{\{92,155\}}v_{31} + \gamma_{\{93,156\}}v_{32}\\ 
&\quad + \gamma_{\{105,120,159,168,189,195\}}v_{33} + \gamma_{202}v_{38} + \gamma_{203}v_{39} + \gamma_{206}v_{40}\\ 
&\quad + \gamma_{207}v_{41} + \gamma_{\{210,213\}}v_{42} + \gamma_{15}v_{45} \equiv 0,\\
p_{(2;4)}(\widetilde\theta) &\equiv \gamma_{\{17,97\}}v_{2} + \gamma_{\{19,68\}}v_{3} + \gamma_{\{29,101\}}v_{7} + \gamma_{31}v_{8} + \gamma_{\{33,77,102\}}v_{9}\\ 
&\quad + \gamma_{\{43,105\}}v_{13} + \gamma_{125}v_{17} + \gamma_{\{17,97\}}v_{18} + \gamma_{\{98,129,161,176\}}v_{19}\\ 
&\quad + \gamma_{\{68,131\}}v_{20} + \gamma_{\{100,111,133,163,180,194\}}v_{21} + \gamma_{\{101,138\}}v_{23}\\ 
&\quad + \gamma_{\{77,102,140\}}v_{24} + \gamma_{146}v_{26} + \gamma_{31}v_{27} + \gamma_{\{150,167,188\}}v_{28}\\ 
&\quad + \gamma_{\{43,105\}}v_{30} + \gamma_{158}v_{32} + \gamma_{197}v_{35} + \gamma_{199}v_{36} + \gamma_{\{201,212,215\}}v_{37}\\ 
&\quad + \gamma_{205}v_{39} + \gamma_{209}v_{41} + \gamma_{14}v_{44} \equiv 0,\\
p_{(2;5)}(\widetilde\theta) &\equiv \gamma_{\{16,63,98\}}v_{1} + \gamma_{\{18,100\}}v_{2} + \gamma_{\{24,74,101\}}v_{4} + \gamma_{26}v_{5} + \gamma_{\{28,102\}}v_{6}\\ 
&\quad + \gamma_{\{38,105\}}v_{10} + \gamma_{\{97,106,124,160,175,193\}}v_{16} + \gamma_{\{63,98,126\}}v_{17}\\ 
&\quad + \gamma_{\{128,162,179\}}v_{18} + \gamma_{\{18,100\}}v_{19} + \gamma_{132}v_{20} + \gamma_{\{74,101,137\}}v_{22}\\ 
&\quad + \gamma_{\{102,139\}}v_{23} + \gamma_{\{145,166,187\}}v_{25} + \gamma_{26}v_{26} + \gamma_{149}v_{27}\\ 
&\quad + \gamma_{\{38,105\}}v_{29} + \gamma_{157}v_{31} + \gamma_{\{196,211,214\}}v_{34} + \gamma_{198}v_{35} + \gamma_{200}v_{36}\\ 
&\quad + \gamma_{204}v_{38} + \gamma_{208}v_{40} + \gamma_{13}v_{43} \equiv 0.
\end{align*}

Computing from the above equalities gives
\begin{equation}\label{c62}
\gamma_t = 0 \mbox{ for } t \in \mathbb L_2,\ \gamma_u= \gamma_v \mbox{ for } (u,v) \in \mathbb M_2,
\end{equation}
where $\mathbb L_2 = \{$13,\, 14,\, 15,\, 26,\, 31,\, 36,\, 125,\, 132,\, 134,\, 135,\, 136,\, 143,\, 146,\, 147,\, 148,\, 149,\, 157,\, 158,\, 197,\, 198,\, 199,\, 200,\, 202,\, 203,\, 204,\, 205,\, 206,\, 207,\, 208,\, 209$\}$ and $\mathbb M_2 = \{$(16,126),\, (17,34),\, (17,97),\, (18,37), (18,100), (19,68), (19,131), (24,137), (28,102), (28,139), (29,101), (29,138), (33,140), (35,98), (38,43), (38,105),\, (44,92),\, (44,155),\, (45,93),\, (45,156),\, (210,213)$\}$.

Applying the homomorphisms $p_{(i;j)}$, $3 \leqslant i < j \leqslant 5$ to \eqref{ctd612} and using the relations \eqref{c61}, \eqref{c62} give
\begin{align*}
p_{(3;4)}(\widetilde\theta) &\equiv \gamma_{\{19,44\}}v_{3} + \gamma_{\{66,88\}}v_{7} + \gamma_{\{19,44\}}v_{8} + \gamma_{\{45,70,77,87,89,96\}}v_{9}\\ 
&\quad + \gamma_{\{111,120\}}v_{13} + \gamma_{\{29,35,129,151\}}v_{19} + \gamma_{\{19,44\}}v_{20}\\ 
&\quad + \gamma_{\{18,33,38,45,133,150,152,159\}}v_{21} + \gamma_{\{161,164\}}v_{23} + \gamma_{\{163,165,167,168\}}v_{24}\\ 
&\quad + \gamma_{\{176,181\}}v_{26} + \gamma_{\{19,44\}}v_{27} + \gamma_{\{45,77,180,182,188,189\}}v_{28} + \gamma_{\{111,120\}}v_{30}\\ 
&\quad + \gamma_{\{194,195\}}v_{32} + \gamma_{\{201,210\}}v_{37} + \gamma_{\{210,212\}}v_{39} + \gamma_{215}v_{41} \equiv 0,\\
p_{(3;5)}(\widetilde\theta) &\equiv \gamma_{\{16,35,45\}}v_{1} + \gamma_{\{44,61,74,82,88,96\}}v_{4} + \gamma_{\{45,63\}}v_{5} + \gamma_{\{65,89\}}v_{6}\\ 
&\quad + \gamma_{\{106,120\}}v_{10} + \gamma_{\{17,24,38,44,124,145,151,159\}}v_{16} + \gamma_{\{16,35,45\}}v_{17}\\ 
&\quad + \gamma_{\{28,128,152\}}v_{18} + \gamma_{\{160,164,166,168\}}v_{22} + \gamma_{\{162,165\}}v_{23}\\ 
&\quad + \gamma_{\{44,74,175,181,187,189\}}v_{25} + \gamma_{\{45,63\}}v_{26} + \gamma_{\{179,182\}}v_{27} + \gamma_{\{106,120\}}v_{29}\\ 
&\quad + \gamma_{\{193,195\}}v_{31} + \gamma_{\{196,210\}}v_{34} + \gamma_{\{210,211\}}v_{38} + \gamma_{214}v_{40} \equiv 0,\\
p_{(4;5)}(\widetilde\theta) &\equiv \gamma_{\{24,28,29,33\}}v_{1} + \gamma_{\{19,61,63,65,66,70\}}v_{4} + \gamma_{\{74,77\}}v_{5} + \gamma_{\{82,87\}}v_{6}\\ 
&\quad + \gamma_{\{106,111\}}v_{10} + \gamma_{\{16,17,18,19,124,128,129,133\}}v_{16} + \gamma_{\{24,28,29,33\}}v_{17}\\ 
&\quad + \gamma_{\{145,150\}}v_{18} + \gamma_{\{160,161,162,163\}}v_{22} + \gamma_{\{166,167\}}v_{23}\\ 
&\quad + \gamma_{\{19,63,175,176,179,180\}}v_{25} + \gamma_{\{74,77\}}v_{26} + \gamma_{\{187,188\}}v_{27} + \gamma_{\{106,111\}}v_{29}\\ 
&\quad + \gamma_{\{193,194\}}v_{31} + \gamma_{\{196,201\}}v_{34} + \gamma_{\{211,212\}}v_{38} + \gamma_{\{214,215\}}v_{40} \equiv 0.
\end{align*}

From these equalities it implies
\begin{equation}\label{c63}
\gamma_{214} = \gamma_{214} = 0 \mbox{ and }  \gamma_u= \gamma_v \mbox{ for } (u,v) \in \mathbb M_3,
\end{equation}
where $\mathbb M_3 = \{$(19,44),\, (45,63), (65,89), (66,88), (74,77), (82,87), (106,111), (106,120), (145,150), (161,164), (162,165), (166,167), (176,181), (179,182), (187,188), (193,194), (193,195), (196,201), (196,210), (196,211), (196,212)$\}$.

Apply the homomorphisms $p_{(1;(2,3))}$, $p_{(1;(2,4))}$ and $p_{(1;(3,4))}$ to \eqref{ctd612} and using \eqref{c61}, \eqref{c62}, \eqref{c63} to get
\begin{align*}
p_{(1;(2,3))}(\widetilde\theta) &\equiv \gamma_{17}v_{22} + \gamma_{35}v_{23} + \gamma_{18}v_{24} + \gamma_{\{19,66,151,161\}}v_{29} + \gamma_{\{45,65,152,162\}}v_{30}\\ 
&\quad + \gamma_{19}v_{31} + \gamma_{45}v_{32} + \gamma_{\{96,159,168\}}v_{33} + \gamma_{176}v_{38} + \gamma_{179}v_{39} + \gamma_{19}v_{40}\\ 
&\quad + \gamma_{45}v_{41} + \gamma_{189}v_{42} + \gamma_{193}v_{45} \equiv 0,\\
p_{(1;(2,4))}(\widetilde\theta) &\equiv \gamma_{17}v_{17} + \gamma_{\{19,66,129,161\}}v_{19} + \gamma_{19}v_{20} + \gamma_{\{70,133,163\}}v_{21}\\ 
&\quad + \gamma_{\{28,33\}}v_{24} + \gamma_{29}v_{26} + \gamma_{\{28,74,82,145,166\}}v_{28} + \gamma_{38}v_{32} + \gamma_{176}v_{35}\\ 
&\quad + \gamma_{19}v_{36} + \gamma_{180}v_{37} + \gamma_{74}v_{39} + \gamma_{187}v_{41} + \gamma_{193}v_{44} \equiv 0,\\
p_{(1;(3,4))}(\widetilde\theta) &\equiv \gamma_{17}v_{5} + \gamma_{129}v_{7} + \gamma_{19}v_{8} + \gamma_{\{133,196\}}v_{9} + \gamma_{\{19,151\}}v_{11} + \gamma_{19}v_{12}\\ 
&\quad + \gamma_{\{145,152\}}v_{13} + \gamma_{\{159,196\}}v_{15} + \gamma_{\{19,66,161\}}v_{19} + \gamma_{\{70,163,196\}}v_{21}\\ 
&\quad + \gamma_{\{19,161\}}v_{23} + \gamma_{\{163,179,180,196\}}v_{24} + \gamma_{\{19,161\}}v_{26}\\ 
&\quad + \gamma_{\{74,82,162,166,179,180\}}v_{28} + \gamma_{\{19,66,161\}}v_{29}\\ 
&\quad + \gamma_{\{45,65,162,166,187,189\}}v_{30} + \gamma_{\{168,187,189,196\}}v_{32}\\ 
&\quad + \gamma_{\{96,168,196\}}v_{33} + \gamma_{106}v_{39} + \gamma_{106}v_{41} \equiv 0.  
\end{align*}
By a direct computation using the above equalities we obtain
\begin{equation}\label{c64}
\gamma_j = 0 \mbox{ for } j \in \mathbb L_4,\  \gamma_j= \gamma_{124} \mbox{ for } j \in \mathbb L_5,
\end{equation}
where $\mathbb L_4 = \{$16,\, 17,\, 18,\, 19,\, 24,\, 28,\, 29,\, 33,\, 34,\, 35,\, 37,\, 38,\, 43,\, 44,\, 45,\, 61,\, 63,\, 65,\, 66,\, 68,\, 70,\, 74,\, 77,\, 82,\, 87,\, 88,\, 89,\, 92,\, 93,\, 96,\, 97,\, 98,\, 100,\, 101,\, 102,\, 105,\, 106,\, 111,\, 120,\, 126,\, 127,\, 128,\, 129,\, 130,\, 131,\, 137,\, 138,\, 139,\, 140,\, 141,\, 142,\, 144,\, 145,\, 150,\, 151,\, 152,\, 153,\, 154,\, 155,\, 156,\, 161,\, 162,\, 164,\, 165,\, 166,\, 167,\, 169,\, 170,\, 171,\, 172,\, 173,\, 174,\, 175,\, 176,\, 177,\, 178,\, 179,\, 180,\, 181,\, 182,\, 183,\, 184,\, 185,\, 186,\, 187,\, 188,\, 189,\, 190,\, 191,\, 192,\, 193,\, 194,\, 195$\}$ and $\mathbb L_5 = \{$133,\, 159,\, 160,\, 163,\, 168,\, 196,\, 201,\, 210,\, 211,\, 212,\, 213$\}$. 

Then, $\widetilde \theta \equiv \gamma_{124}\sum_{j\in \mathbb J_5}w_j$. By a simple computation we get
$$p_{(2;(3,4))}(\widetilde \theta) \equiv \gamma_{124}(v_{24} + v_{32} + v_{39} + v_{41}) \equiv 0.$$
This equality implies that $\gamma_{124} =0$ and $\gamma_t = 0$ for $1 \leqslant t \leqslant 215$. Hence, we get $\theta \equiv \sum_{1\leqslant u \leqslant 10}\bar\gamma_{(u+215)}g_u$. The proposition is proved.
\end{proof}

\begin{props}\label{mdp2} Set $\omega = (4)|(2)|(3)$, we have 
$$\widetilde{SF}_5({\omega}) = \langle [g_{11}]_\omega\rangle,$$
where the polynomials $g_{11}$ is explicitly determined in Subsection \ref{s51}.
\end{props}
\begin{proof} The proposition is proved by an argument analogous to the one in the proof of Proposition \ref{mdsp1}.
By a direct computation, we can easily verify that $p_{(i;I)}(g_{11}) \equiv_\omega 0$ for all $(i;I)\in \mathcal N_5$. Hence, $[g_{11}] \in \widetilde{SF}_5{(\omega)}$.
The set $\{[b_t]_\omega :1\leqslant t \leqslant 50\}$ is a basis of  $QP_5^+(\omega)$ (see Subsection \ref{s50}), hence if $\theta^* \in \widetilde {SF}_{5}(\omega)$, then we have $\theta^* \equiv_\omega \sum_{1\leqslant t \leqslant 50}\gamma_t^*b_{t}$, where $\gamma_t^* \in \mathbb F_2$. The leading monomial of $g_{11}$ is $b_{50}$. Hence, we get 
\begin{equation}\label{dtd612}
\overline{\theta} := \theta^* + \gamma_{50}^*g_{11}\equiv \sum_{1\leqslant t \leqslant 49}\gamma_tb_{t},
\end{equation}
where $\gamma_j \in \mathbb F_2$, $1\leqslant j \leqslant 49$, and $p_{i;I}(\widetilde\theta) \equiv 0$ for all $(i;I) \in \mathcal N_5$. 

Let $v_s,\, 46\leqslant s \leqslant 49$, be as in Subsection \ref{s50}. Applying $p_{(1;u)}$, $2 \leqslant u \leqslant 5,$ to (\ref{dtd612}) and using Theorem \ref{dlsig} give
\begin{align*}
p_{(1;2)}(\overline{\theta}) &\equiv_\omega \gamma_{1}v_{46} + \gamma_{2}v_{47} + \gamma_{3}v_{48} + \gamma_{\{34,46\}}v_{49} \equiv_\omega 0,\\ 
p_{(1;3)}(\overline{\theta}) &\equiv_\omega \gamma_{6}v_{46} + \gamma_{7}v_{47} + \gamma_{\{33,43\}}v_{48} + \gamma_{4}v_{49} \equiv_\omega 0,\\
p_{(1;4)}(\overline{\theta}) &\equiv_\omega \gamma_{8}v_{46} + \gamma_{\{32,44,47\}}v_{47} + \gamma_{10}v_{48} + \gamma_{12}v_{49} \equiv_\omega 0,\\
p_{(1;5)}(\overline{\theta}) &\equiv_\omega \gamma_{\{31,45,48,49\}}v_{46} + \gamma_{9}v_{47} + \gamma_{11}v_{48} + \gamma_{13}v_{49} \equiv_\omega 0,\\
p_{(2;3)}(\overline{\theta}) &\equiv_\omega \gamma_{14}v_{46} + \gamma_{15}v_{47} + \gamma_{\{37,41,43,46\}}v_{48} + \gamma_{5}v_{49} \equiv_\omega 0,\\
p_{(2;4)}(\overline{\theta}) &\equiv_\omega \gamma_{16}v_{46} + \gamma_{\{36,44\}}v_{47} + \gamma_{18}v_{48} + \gamma_{26}v_{49} \equiv_\omega 0,\\
p_{(2;5)}(\overline{\theta}) &\equiv_\omega \gamma_{\{35,42,45\}}v_{46} + \gamma_{17}v_{47} + \gamma_{19}v_{48} + \gamma_{27}v_{49} \equiv_\omega 0,\\
p_{(3;4)}(\overline{\theta}) &\equiv_\omega \gamma_{20}v_{46} + \gamma_{\{39,41,47\}}v_{47} + \gamma_{23}v_{48} + \gamma_{28}v_{49} \equiv_\omega 0,\\
p_{(3;5)}(\overline{\theta}) &\equiv_\omega \gamma_{\{38,41,42,48\}}v_{46} + \gamma_{21}v_{47} + \gamma_{24}v_{48} + \gamma_{29}v_{49} \equiv_\omega 0,\\
p_{(4;5)}(\overline{\theta}) &\equiv_\omega \gamma_{\{40,42,49\}}v_{46} + \gamma_{22}v_{47} + \gamma_{25}v_{48} + \gamma_{30}v_{49} \equiv_\omega 0.
\end{align*}
	
From the above relations we get
\begin{equation}\label{d61}
\gamma_j = 0,\ 1\leqslant t \leqslant 30,\ \gamma_{33}= \gamma_{43},\ \gamma_{34}= \gamma_{46},\ \gamma_{36}= \gamma_{44}.
\end{equation}

By applying the homomorphisms $p_{(1;(u,v))},\, 1 < u < v < 5$, to \eqref{dtd612} and using \eqref{d61}, we get
\begin{align*}
p_{(1;(2,3))}(\overline{\theta}) &\equiv_\omega \gamma_{37}v_{48} + \gamma_{41}v_{49} \equiv_\omega 0,\\
p_{(1;(2,4))}(\overline{\theta}) &\equiv_\omega \gamma_{\{32,34,47\}}v_{47} + \gamma_{47}v_{49}  \equiv_\omega 0,\\
p_{(1;(3,4))}(\overline{\theta}) &\equiv_\omega \gamma_{\{32,33,36,39,41,47\}}v_{47} + \gamma_{36}v_{48} \equiv_\omega 0,\\
\end{align*}
By a direct computation using these equalities we get
\begin{equation}\label{d62}
\gamma_j = 0 \mbox{ for } j \in \{32,\, 33,\, 34,\, 36,\, 37,\, 39,\, 41,\, 43,\, 44,\, 46,\, 47\},
\end{equation}

Apply the homomorphisms $p_{(2;(3,5))}$, $p_{(2;(4,5))}$ to \eqref{dtd612} using \eqref{d61}, \eqref{d62} to obtain
\begin{align*}
p_{(2;(3,5))}(\overline{\theta}) &\equiv_\omega \gamma_{\{35,38,42,45\}}v_{46} + \gamma_{\{45,48\}}v_{48} \equiv_\omega 0,\\
p_{(2;(4,5))}(\overline{\theta}) &\equiv_\omega \gamma_{\{35,40,42,45\}}v_{46} + \gamma_{\{45,49\}}v_{47} \equiv_\omega 0.
\end{align*}
From the above relations we obtain
\begin{equation}\label{d63}
 \gamma_{45} = \gamma_{48} =\gamma_{49}.
\end{equation}

Apply the homomorphisms $p_{(1;(2,5))}$, $p_{(1;(3,5))}$, $p_{(1;(4,5))}$ to \eqref{dtd612} using \eqref{d61}, \eqref{d62}, \eqref{d63} to get
\begin{align*}
p_{(1;(2,5))}(\overline{\theta}) &\equiv_\omega \gamma_{35}v_{46} + \gamma_{42}v_{49} \equiv_\omega 0,\\
p_{(1;(3,5))}(\overline{\theta}) &\equiv_\omega \gamma_{38}v_{46} + \gamma_{42}v_{48} \equiv_\omega 0,\\
p_{(1;(4,5))}(\overline{\theta}) &\equiv_\omega \gamma_{40}v_{46} + \gamma_{42}v_{47} \equiv_\omega 0,\\
\end{align*}
These equalities imply that $\gamma_t = 0$ for $1 \leqslant t \leqslant 49$. Hence, we get $\theta^* \equiv \gamma_{50}^*g_{11}$. The proposition is proved.
\end{proof}
\begin{rems}\
	
i) Consider the relation (3.13) at the end of Pages 15 of Ph\'uc \cite{p24}: 
$$\eta := \sum_{501 \leqslant j \leqslant 550}\mbox{adm}_j \equiv_\omega 0$$
with $\omega = (4,2,3)$. Here, we use the notations as defined in \cite{p24}. From the proof of Proposition \ref{mdp2}, we see that $p_{(i;I)}(\eta) \equiv_\omega 0$ for all $(i;I)\in \mathcal N_5$ if and only if $\gamma_{j} = \gamma_{531}$ for $531 \leqslant j \leqslant 550$, $j \ne 535, 540, 541, 542$ and $\gamma_{j} = 0$ for $j \in \{501,502,\ldots,530, 535, 540, 541, 542\}$.
However, in \cite{p24}, the author says that from $p_{(u;v)}(\eta) \equiv_\omega 0$, $1\leqslant u < v \leqslant 5$, $p_{(1;(2,4))}(\eta) \equiv_\omega 0$ and $p_{(1;(2,5))}(\eta) \equiv_\omega 0$ he obtains $\gamma_{j} = 0$ for all $501\leqslant  j \leqslant 550$. Thus, the computations for $p_{(i;I)}(\eta)$ on Page 16 of \cite{p24} are false.

ii) It is not difficult to show that $[g_{11}]_\omega$ is an $GL_5$-invariant, hence we obtain $\dim QP_5(\omega)^{GL_5} \geqslant 1$. However, in \cite{p24}, the author stated that $\dim QP_5(\omega)^{GL_5} = 0$. Thus, the computations for $QP_5(\omega)^{GL_5}$ on Pages 16, 17 of \cite{p24} are false. 
\end{rems}

By similar computations as presented on the proof of Proposition \ref{mdsp1} we get the following.
\begin{props} We have $\widetilde{SF}_5((4)|^2|(2)) = 0.$
\end{props}

Recall that $V_k \cong \langle x_1, x_2,\ldots, x_k\rangle \subset P_k$. For $1 \leqslant j \leqslant k$. Define the $\mathbb F_2$-linear map $\rho_j:V_k \to V_k$, by $\rho_k(x_1) = x_1+x_2$,  $\rho_k(x_t) = x_t$ for $t > 1$ and $\rho_j(x_j) = x_{j+1}, \rho_j(x_{j+1}) = x_j$, $\rho_j(x_t) = x_t$ for $t \ne j, j+1,\ 1 \leqslant j < k$.  The group $GL_k\cong GL(V_k)$ is generated by $\rho_j,\ 1\leqslant j \leqslant k$, and the subgroup $\Sigma_k$ is generated by $\rho_j,\ 1 \leqslant j < k$. The linear map $\rho_j$ induces an $\mathcal A$-homomorphism of algebras $\rho_j: P_k \to P_k$. Hence, a class $[g]_\omega \in QP_k(\omega)$ is an element of $QP_k^{GL_k}$ if and only if $\rho_j(g) \equiv_\omega g$ for $1 \leqslant j\leqslant k$.  It is an element of $QP_k(\omega)^{\Sigma_k}$ if and only if $\rho_j(g) \equiv_\omega g$ for $1 \leqslant j < k$.

\subsection{Computation of $QP_5((4)|(2)|(1)|^2)^{GL_5}$}\

\smallskip
In this subsection we prove the following.

\begin{props}\label{mdt1} We have 
$QP_5((4)|(2)|(1)|^2)^{GL_5} = 0.$
\end{props}
To compute $QP_5((4)|(2)|(1)|^2)^{GL_5}$ we need the following.
\begin{lems}\label{bdt1} We have 
$$QP_5((4)|(2)|(1)|^2)^{\Sigma_5} = \langle [h_s]: 1 \leqslant s \leqslant 7\rangle, $$
where the polynomials $h_s$ are determined as in Subsection \ref{s52}.
\end{lems}

\begin{proof}
We have $B_5((4)|(2)|(1)|^2)= B_5^0((4)|(2)|(1)|^2)\cup\{a_t:1\leqslant t \leqslant 225\}$, where the monomials $a_t$ are determined as in Subsection \ref{s50}. By using this set we see that there is a direct summand decompositions of the $\Sigma_5$-modules
\begin{align*} 
QP_5((4)|(2)|(1)|^2) [\Sigma_5(v_1)]\bigoplus&[\Sigma_5(v_2)]\bigoplus[\Sigma_5(v_5,v_6,v_7)]\\
&\bigoplus[\Sigma_5(v_{23})]\bigoplus[\Sigma_5(a_1)]\bigoplus \mathcal M,
\end{align*}
where $v_s$, $a_t$ are defined as in Subsection \ref{s50} and $\mathcal M = \langle [a_t]: 16 \leqslant t \leqslant 225\rangle$. By a direct computation, we have 
\begin{align*}&[\Sigma_5(v_1)]^{\Sigma_5} = \langle [h_1]\rangle,\ \ [\Sigma_5(v_2)]^{\Sigma_5} = \langle [h_2]\rangle,\\
&[\Sigma_5(v_5,v_6,v_7)]^{\Sigma_5} = \langle [h_3]\rangle,\ \ [\Sigma_5(w_{23})]^{\Sigma_5} = \langle [h_4]\rangle,\\
&[\Sigma_5(a_1)]^{\Sigma_5} = 0,\ \ \mathcal M^{\Sigma_5} = \langle [h_5],[h_6],[h_7]\rangle.
\end{align*}
To illustrate the proofs, we present a proof for the relation $[\Sigma_5(a_1)]^{\Sigma_5} = 0$. The proofs of other cases are carried out by a similar computation.

Note that $[\Sigma_5(a_1)] = \langle [a_t]: 1 \leqslant t \leqslant 15\rangle$. Suppose $h \in P_5((4)|(2)|(1)|^2)$ and $[h] \in \mathcal [\Sigma_5(a_1)]^{\Sigma_5}$. Then we have
$h \equiv \sum_{t=1}^{15}\gamma_ta_t$. Cnssider the homomorphisms $\rho_j: P_k \to P_k$ as defined in \cite[Page 46]{suw} with $1 \leqslant j \leqslant 5$. A direct computation shows
\begin{align*}
\rho_1(h)+h &\equiv \gamma_{7}a_{1} + \gamma_{8}a_{2} + \gamma_{7}a_{3} + \gamma_{8}a_{4} + \gamma_{9}a_{5} + \gamma_{9}a_{6} + \gamma_{\{10,13\}}a_{10}\\ 
&\quad + \gamma_{\{11,14\}}a_{11} + \gamma_{\{12,15\}}a_{12} + \gamma_{\{10,13\}}a_{13} + \gamma_{\{11,14\}}a_{14}\\ 
&\quad + \gamma_{\{12,15\}}a_{15} \equiv 0,\\ 
\rho_2(h)+h &\equiv \gamma_{\{3,7\}}a_{3} + \gamma_{\{4,8\}}a_{4} + \gamma_{\{5,10\}}a_{5} + \gamma_{\{6,11\}}a_{6} + \gamma_{\{3,7\}}a_{7}\\ 
&\quad + \gamma_{\{4,8\}}a_{8} + \gamma_{\{9,12\}}a_{9} + \gamma_{\{5,10\}}a_{10} + \gamma_{\{6,11\}}a_{11} + \gamma_{\{9,12\}}a_{12}\\ 
&\quad + \gamma_{15}a_{13} + \gamma_{15}a_{14} \equiv 0,\\
\rho_3(h)+h &\equiv \gamma_{\{1,3\}}a_{1} + \gamma_{\{2,5\}}a_{2} + \gamma_{\{1,3\}}a_{3} + \gamma_{\{4,6\}}a_{4} + \gamma_{\{2,5\}}a_{5}\\ 
&\quad + \gamma_{\{4,6\}}a_{6} + \gamma_{\{8,9\}}a_{8} + \gamma_{\{8,9\}}a_{9} + \gamma_{\{11,12\}}a_{11} + \gamma_{\{11,12\}}a_{12}\\ 
&\quad + \gamma_{\{14,15\}}a_{14} + \gamma_{\{14,15\}}a_{15} \equiv 0,\\
\rho_4(h)+h &\equiv \gamma_{\{1,2\}}a_{1} + \gamma_{\{1,2\}}a_{2} + \gamma_{\{3,4\}}a_{3} + \gamma_{\{3,4\}}a_{4} + \gamma_{\{5,6\}}a_{5}\\ 
&\quad + \gamma_{\{5,6\}}a_{6} + \gamma_{\{7,8\}}a_{7} + \gamma_{\{7,8\}}a_{8} + \gamma_{\{10,11\}}a_{10} + \gamma_{\{10,11\}}a_{11}\\ 
&\quad + \gamma_{\{13,14\}}a_{13} + \gamma_{\{13,14\}}a_{14} \equiv 0.
\end{align*}
By computing from the above relations, we get $\gamma_t = 0$ for $1 \leqslant t \leqslant 15$. Hence, $h \equiv 0$ and $[\Sigma_5(a_1)]^{\Sigma_5} = 0$.
\end{proof}
\begin{proof}[Proof of Proposition \ref{mdt1}] Suppose $f \in QP_5((4)|(2)|(1)|^2)$ such that the class $[f] \in P_5((4)|(2)|(1)|^2)^{GL_5}$. Then, $[f] \in P_5((4)|(2)|(1)|^2)^{\Sigma_5}$. By Lemma \ref{bdt1}, we have 
$$ f \equiv \sum_{1 \leqslant s \leqslant 7}\gamma_sh_s,$$
where $\gamma_s \in \mathbb F_2$. By computing $\rho_5(f)+f$ in terms of the admissible monomials, we get
\begin{align*}
\rho_5(f)+f &\equiv \gamma_1a_1 + \gamma_2a_{125}+ \gamma_3a_{16} + \gamma_4a_{26} + \gamma_5a_{121}\\
&\quad + (\gamma_2 + \gamma_6)a_{61} + \gamma_7a_{63} + \mbox{ other terms } \equiv 0.
\end{align*}
This equality implies $\gamma_s = 0$ for $1 \leqslant s \leqslant 7$. Hence, $[f] = 0$. The proposition is proved.
\end{proof}

\subsection{Computation of $QP_5((4)|(2)|(3))^{GL_5}$}\

\smallskip
In this subsection we prove the following.

\begin{props}\label{mdt2} Set $\overline{\omega} = (4)|(2)|(3)$. Then, we have 
$$QP_5(\overline{\omega})^{GL_5} = \langle [h_{10}]_{\overline{\omega}}\rangle.$$
\end{props}
Note that $h_{10} = g_{11}$. We need the following for the proof of Proposition \ref{mdt2}.

\begin{lems}\label{bdt2} We have 
$$QP_5((4)|(2)|(1)|^2)^{\Sigma_5} = \langle [h_8]_{\overline{\omega}}, [h_9]_{\overline{\omega}}, [h_{10}]_{\overline{\omega}}\rangle. $$
\end{lems}

\begin{proof}
We have $B_5(\overline{\omega})= B_5^0(\overline{\omega})\cup\{b_t:1\leqslant t \leqslant 50\}$, where the monomials $b_t$ are determined as in Subsection \ref{s50}. By using this set we see that there is a direct summand decompositions of the $\Sigma_5$-modules
\begin{align*} 
QP_5(\overline{\omega}) = [\Sigma_5(v_{46})]_{\overline{\omega}}\bigoplus&[\Sigma_5(b_1)]_{\overline{\omega}}\bigoplus[\Sigma_5(b_{31})]_{\overline{\omega}},
\end{align*}

where $v_{s}$, $b_t$ are defined as in Subsection \ref{s50}. By a direct computation, we have  
\begin{align*}&[\Sigma_5(v_{46})]^{\Sigma_5} = \langle [h_8]_{\overline{\omega}}\rangle,\ \ [\Sigma_5(b_1)]^{\Sigma_5} = \langle [h_9]_{\overline{\omega}}\rangle, \ \ [\Sigma_5(b_{31})]^{\Sigma_5} = \langle [h_{10}]_{\overline{\omega}}\rangle.
\end{align*}
We present a detailed proof for the case $[\Sigma_5(b_{31})]^{\Sigma_5} = \langle [h_{10}]_{\overline{\omega}}\rangle$. 
	
Note that $[\Sigma_5(b_{31})] = \langle [b_t]: 31 \leqslant t \leqslant 50\rangle$. Suppose $h \in P_5(\overline{\omega})$ and $[h]_{\overline{\omega}} \in \mathcal [\Sigma_5(b_{31})]_{\overline{\omega}}^{\Sigma_5}$. Then we have
$h \equiv_{\overline{\omega}} \sum_{t=31}^{50}\gamma_tb_t$. By a direct computation we have 
\begin{align*}
\rho_1(h)+h &\equiv_{\overline{\omega}} \gamma_{\{31,35\}}b_{31} + \gamma_{\{32,36\}}b_{32} + \gamma_{\{33,37,46\}}b_{33} + \gamma_{\{34,46\}}b_{34}\\ &\quad + \gamma_{\{31,35\}}b_{35} + \gamma_{\{32,36\}}b_{36} + \gamma_{\{33,34,37\}}b_{37} + \gamma_{41}b_{43}\\ &\quad + \gamma_{\{47,50\}}b_{44} + \gamma_{\{42,48,49\}}b_{45} + \gamma_{\{34,46\}}b_{46} +  \equiv_{\overline{\omega}} 0,\\ 
\rho_2(h)+h &\equiv_{\overline{\omega}} \gamma_{\{33,34\}}b_{33} + \gamma_{\{33,34\}}b_{34} + \gamma_{\{35,38,41\}}b_{35} + \gamma_{\{36,39,41\}}b_{36}\\ &\quad + \gamma_{\{37,41\}}b_{37} + \gamma_{\{35,37,38\}}b_{38} + \gamma_{\{36,37,39\}}b_{39} + \gamma_{\{37,41\}}b_{41}\\ &\quad + \gamma_{\{43,46\}}b_{43} + \gamma_{\{44,47\}}b_{44} + \gamma_{\{45,48\}}b_{45} + \gamma_{\{43,46\}}b_{46}\\ &\quad + \gamma_{\{44,47\}}b_{47} + \gamma_{\{45,48\}}b_{48} \equiv_{\overline{\omega}} 0,\\
\rho_3(h)+h &\equiv_{\overline{\omega}} \gamma_{\{32,33\}}b_{32} + \gamma_{\{32,33\}}b_{33} + \gamma_{\{36,37,46\}}b_{36} + \gamma_{\{36,37,46\}}b_{37}\\ &\quad + \gamma_{\{38,40,47\}}b_{38} + \gamma_{\{39,50\}}b_{39} + \gamma_{\{38,39,40,41\}}b_{40} + \gamma_{\{41,47,50\}}b_{41}\\ &\quad + \gamma_{\{43,44,47,50\}}b_{43} + \gamma_{\{41,43,44\}}b_{44} + \gamma_{\{39,41,47\}}b_{47} + \gamma_{\{48,49\}}b_{48}\\ &\quad + \gamma_{\{48,49\}}b_{49} + \gamma_{\{39,50\}}b_{50} \equiv_{\overline{\omega}} 0,\\
\rho_4(h)+h &\equiv_{\overline{\omega}} \gamma_{\{31,32\}}b_{31} + \gamma_{\{31,32\}}b_{32} + \gamma_{\{35,36\}}b_{35} + \gamma_{\{35,36\}}b_{36}\\ &\quad + \gamma_{\{38,39\}}b_{38} + \gamma_{\{38,39\}}b_{39} + \gamma_{\{40,42\}}b_{40} + \gamma_{\{40,42\}}b_{42}\\ &\quad + \gamma_{\{42,44,45\}}b_{44} + \gamma_{\{40,44,45\}}b_{45} + \gamma_{\{42,47,48\}}b_{47} + \gamma_{\{40,47,48\}}b_{48}\\ &\quad + \gamma_{\{49,50\}}b_{49} + \gamma_{\{49,50\}}b_{50} \equiv_{\overline{\omega}} 0.
\end{align*}
Computing from the above relations gives $\gamma_t = 0$ for $t\in \{37,\, 40,\, 41,\, 42\}$ and $\gamma_t = \gamma_{31}$ for $t \ne 37,\, 40,\, 41,\, 42$. Hence, $h \equiv_{\overline{\omega}} h_{10}$. The lemma is proved.
\end{proof}
\begin{proof}[Proof of Proposition \ref{mdt2}] Suppose $g \in P_5(\overline{\omega})$ and $[g]_{\overline{\omega}} \in QP_5(\overline{\omega})^{GL_5}$. Then, $[g]_{\overline{\omega}} \in QP_5(\overline{\omega})^{\Sigma_5}$. By Lemma \ref{bdt2}, we have 
$$ g \equiv \gamma_{8}h_8 + \gamma_{9}h_9 + \gamma_{10}h_{10},$$
where $\gamma_s \in \mathbb F_2$. By computing $\rho_5(g)+g$ in terms of the admissible monomials, we get
\begin{align*}
\rho_5(g)+g &\equiv_{\overline{\omega}}  (\gamma_8+\gamma_9)(x_2^3x_3^5x_4^5x_5^7 + x_2^3x_3^5x_4^7x_5^5 + x_2^3x_3^7x_4^5x_5^5)\\ 
&\quad \gamma_8\big(x_2^{7}x_3^{3}x_4^{5}x_5^{5} + x_1^{3}x_2^{7}x_4^{5}x_5^{5} + x_1^{3}x_2^{7}x_3^{5}x_5^{5} + x_1^{3}x_2^{7}x_3^{5}x_4^{5}\\ &\quad + b_{31} + b_{32} + b_{33} + b_{34} + b_{35} + b_{36} + b_{46}\big)\\
&\quad +\gamma_9\big(b_{4} + b_{6} + b_{7} + b_{8} + b_{9} + b_{10} + b_{11} + b_{12} + b_{13}\\ &\quad + b_{23} + b_{24} + b_{25} + b_{43} + b_{44} + b_{45}\big)
 \equiv_{\overline{\omega}} 0.
\end{align*}
This equality implies $\gamma_8 = \gamma_9 = 0$. Hence, $[g]_{\overline{\omega}} = [h_{10}]_{\overline{\omega}}$. The proposition is proved.
\end{proof}

\subsection{Computation of $QP_5((4)|^{2}|(2))^{GL_5}$}\

\smallskip
In this subsection we prove the following.

\begin{props}\label{mdt3} We have 
$QP_5((4)|^2|(2))^{GL_5} = \langle 0\rangle.$
\end{props}

\begin{proof} This proposition is proved by the same computations as in the proof of Proposition \ref{mdt2}. We only present here a brief proof.
	
\smallskip
We have $B_5((4)|^{2}|(2))= B_5^0((4)|^{2}|(2))\cup\{c_t:1\leqslant t \leqslant 91\}$, where the monomials $c_t$ are determined as in Subsection \ref{s50}. Based on this set we have a direct summand decomposition of the $\Sigma_5$-modules
$$QP_5(\omega^*)= [\Sigma_5(w_{50})]_{\omega^*}\bigoplus[\Sigma_5(c_1)]_{\omega^*}\bigoplus[\Sigma_5(c_{31})]_{\omega^*}\bigoplus[\Sigma_5(c_{86})]_{\omega^*},$$
where $\omega^* := (4)|^{2}|(2)$. By a direct computation we have 
\begin{align*}
&[\Sigma_5(w_{50}]_{\omega^*}^{\Sigma_5} = \langle[h_{11}]_{\omega^*}\rangle,\ \ [\Sigma_5(c_1)]_{\omega^*}^{\Sigma_5} = \langle[h_{12}]_{\omega^*}\rangle,\\ &[\Sigma_5(c_{31})]_{\omega^*}^{\Sigma_5} = 0,\ \
[\Sigma_5(c_{86})]_{\omega^*}^{\Sigma_5} = \langle[h_{13}]_{\omega^*}\rangle,
\end{align*}
where $h_{11}, h_{12}, h_{13}$ are defined in Subsection \ref{s52}. 

Suppose $f \in P_5(\omega^*)$ and $[f]_{\omega^*} \in QP_5(\omega^*)^{GL_5}$. Then, $[f]_{\omega^*} \in QP_5(\omega^*)^{\Sigma_5}$. Hence, we have 
$$f \equiv_{\omega^*} \gamma_{11}h_{11} + \gamma_{12}h_{12} + \gamma_{13}h_{13},$$
where $\gamma_{11}, \gamma_{12}, \gamma_{13} \in \mathbb F_2$. By computing $\rho_5(f)+f$ in terms of the admissible monomials, we get
\begin{align*}
\rho_5(f)+f &\equiv_{\omega^*} \gamma_{11}w_{52} + \gamma_{12}c_4 + \gamma_{13}c_{40} + \mbox{ other terms } \equiv_{\omega^*} 0.
\end{align*}
This equality implies $\gamma_{11} = \gamma_{12} = \gamma_{13} = 0$. Hence, $[f]_{\omega^*} = 0$ and the proposition is proved.
\end{proof}

\begin{rems} In \cite{p24}, Ph\'uc stated on Page 16 that $\dim\mathbb V^{GL_5} = 1$ with $\mathbb V = QP_5(4,2,1,1) \bigoplus QP_5(4,4,2)$. However, by combining Propositions \ref{mdt1} and \ref{mdt3}, we have $\mathbb V^{GL_5} = 0$. Hence, Ph\'uc's result is false.  
\end{rems}

\subsection{Proof of Theorem \ref{thm1}}\

\smallskip
We note that by a direct computation we can easily check that $\rho_j(p) + p \equiv 0$ for $1 \leqslant j \leqslant 5$. Hence, $[p]\in (QP_5)_{20}^{GL_5}$.

Suppose $f \in P_5$ and $[f] \in (QP_5)_{20}^{GL_5}$. By Proposition \ref{mdt3}, $[f]_{\omega^*} = 0$, hence $[f]_{\overline{\omega}}\in QP_5(\overline{\omega})^{GL_5}$. By using Proposition \ref{mdt2}, we get 
$$f \equiv \gamma h_{10} + \sum_{x\in B((4)|(2)|(1)^2)}\gamma_x.x,$$
where $\gamma,\, \gamma_x \in \mathbb F_2$ for all $x\in B((4)|(2)|(1)^2)$. By computing from the relation $\rho_j(f)+f \equiv 0$, we see that $\rho_j(f)+f \equiv 0$ for $1 \leqslant j < 5$ if and only if 
$$ f \equiv \gamma(h_{10}+h_0) + \sum_{1 \leqslant t \leqslant 7}\gamma_th_t,$$
where $\gamma_t \in \mathbb F_2$ and
\begin{align*}
h_0 &= x_1x_2x_3^{3}x_4^{3}x_5^{12} + x_1x_2x_3^{3}x_4^{5}x_5^{10} + x_1x_2x_3^{3}x_4^{6}x_5^{9} + x_1x_2x_3^{3}x_4^{12}x_5^{3}\\ 
&\quad + x_1x_2x_3^{6}x_4^{3}x_5^{9} + x_1x_2x_3^{6}x_4^{9}x_5^{3} + x_1x_2^{3}x_3x_4^{3}x_5^{12} + x_1x_2^{3}x_3x_4^{5}x_5^{10}\\ 
&\quad  + x_1x_2^{3}x_3x_4^{6}x_5^{9} + x_1x_2^{3}x_3x_4^{12}x_5^{3}+ x_1x_2^{3}x_3^{3}x_4^{4}x_5^{9} + x_1x_2^{3}x_3^{3}x_4^{5}x_5^{8}\\ 
&\quad  + x_1x_2^{3}x_3^{5}x_4^{3}x_5^{8} + x_1x_2^{3}x_3^{5}x_4^{8}x_5^{3} + x_1x_2^{6}x_3x_4^{3}x_5^{9} + x_1x_2^{6}x_3x_4^{9}x_5^{3}\\ 
&\quad  + x_1^{3}x_2x_3^{4}x_4^{3}x_5^{9} + x_1^{3}x_2x_3^{4}x_4^{9}x_5^{3} + x_1^{3}x_2x_3^{5}x_4^{3}x_5^{8} + x_1^{3}x_2x_3^{5}x_4^{8}x_5^{3}\\ 
&\quad + x_1^{3}x_2^{4}x_3x_4^{3}x_5^{9} + x_1^{3}x_2^{4}x_3x_4^{9}x_5^{3} + x_1^{3}x_2^{5}x_3x_4^{3}x_5^{8} + x_1^{3}x_2^{5}x_3x_4^{8}x_5^{3}.
\end{align*}
Note that $h_0+h_{10} = p$. Since $[p]\in (QP_5)_{20}^{GL_5}$, we get $$[\sum_{1 \leqslant t \leqslant 7}\gamma_th_t] \in QP_5((4)|(2)|(1)^2)^{GL_5}.$$ By using Proposition \ref{mdt1}, we obtain $\gamma_t = 0$ for $1 \leqslant t \leqslant 7$. The theorem is proved.

%=========================

\section{Proof of Theorem \ref{thm2}}\label{s4}
\setcounter{equation}{0}

Recall from \cite{suz} that
$$
(QP_5)_{30}	\cong QP_5((2)|^4) \bigoplus QP_5(2,4,3,1) \bigoplus QP_5(4,3,3,1).
$$ 
The set of all admissible monomials of the degree 30 in $P_5$ is explicitly presented in \cite{smo}.

The proof of this theorem is carried out by the same argument as presented in the proof of Theorem \ref{thm1}. So, we only present here some main results of the proof. 

%============================
\subsection{Computation of $QP_5((2)|^4)^{GL_5}$ and $QP_5(2,4,3,1)^{GL_5}$}

\begin{props}\label{mdt31} We have 
$QP_5((2)|^4)^{GL_5} = 0.$
\end{props}

To compute $QP_5((2)|^4)^{GL_5}$ we need the following.

\begin{lems}\label{bdt31} We have 
$$QP_5((2)|^4)^{\Sigma_5} = \langle [p_s]: 1 \leqslant s \leqslant 9\rangle, $$
where the polynomials $p_s$ are determined as in Subsection \ref{s53}.
\end{lems}

\begin{proof}[Proof of Proposition \ref{mdt31}] Suppose $f \in QP_5((2)|^4)$ and $[f] \in P_5((2)|^4)^{GL_5}$. Then, $[f] \in P_5((2)|^4)^{\Sigma_5}$. By Lemma \ref{bdt31}, we have 
$$ f \equiv \sum_{1 \leqslant s \leqslant 9}\gamma_sp_s,$$
where $\gamma_s \in \mathbb F_2$ for $1 \leqslant s \leqslant 9$. By computing $\rho_5(f)+f$ in terms of the admissible monomials, we get
\begin{align*}
\rho_5(f)+f &\equiv (\gamma_1 + \gamma_2)x_2^{15}x_3^{15} + (\gamma_2+\gamma_6)x_2x_3^{15}x_4^{14} + \gamma_5x_2^3x_3x_4^{12}x_5^{14} \\ &\quad + \gamma_6[x_1^3x_2x_3^4x_4^{10}x_5^{12}] +
\gamma_7x_1x_2^3x_3^2x_4^{12}x_5^{12} +
\gamma_8x_1x_2^3x_3^{12}x_4^{2}x_5^{12}\\ &\quad +
(\gamma_2+\gamma_4)x_2^{15}x_4x_5^{14} + 
(\gamma_2+\gamma_3)x_1x_2^{15}x_3^{14}\\ &\quad  + 
(\gamma_5+\gamma_9)[x_2x_3^3x_4^{12}x_5^{14}] + \mbox{other terms} \equiv 0.
\end{align*}
This equality implies $\gamma_s = 0$ for $1 \leqslant s \leqslant 9$. Hence, $[f] = 0$. The proposition is proved.
\end{proof}

Since $QP_5(2,4,3,1)$ is an $GL_5$-module of dimension 1, we get the following.  
\begin{props}\label{mdt32} We have 
$$QP_5(2,4,3,1)^{GL_5} = \langle [x_1^3x_2^5x_3^6x_4^6x_5^{10}]_{\omega}\rangle, \mbox{ with } \omega = (2,4,3,1).$$
\end{props}
\subsection{Computation of $QP_5((4)|(3)|^2|(1))^{GL_5}$}\

\medskip
By a direct computation based on the basis of the space $QP_5((4)|(3)|^2|(1))$ as presented in \cite{smo} one gets the following result. 

\begin{lems}\label{bdt3} We have
$$QP_5((4)|(3)|^2|(1)) = \langle [q_s]_{\omega^*} : 1 \leqslant s \leqslant 7\rangle, $$
where the polynomials $q_s,\, 1 \leqslant s \leqslant 7$, are determined as in Subsection \ref{s54} and $\omega^* = (4,3,3,1)$.
\end{lems}
\begin{proof} We have a direct summand decomposition of the $\Sigma_5$-modules
\begin{align*}
&QP_5((4)|(3)|^2|(1)) \cong \Sigma_5(x_1x_2^7x_3^7x_4^{15})\bigoplus \Sigma_5(x_1^3x_2^5x_3^7x_4^{15}) \bigoplus \Sigma_5(x_1^3x_2^7x_3^7x_4^{13})\\
&\quad \bigoplus \Sigma_5(x_1x_2x_3^6x_4^7x_5^{15}, x_1x_2^2x_3^5x_4^7x_5^{15}) \bigoplus \Sigma_5(x_1x_2^3x_3^5x_4^6x_5^{15})\bigoplus \mathcal M,
\end{align*}
where $\mathcal M$ is the subspace of $QP_5((4)|(3)|^2|(1))$ spanned by 330 admissible monomials of $B_5^+(4,3,3,1)\setminus \mathcal C$ as listed in \cite[Theorem 3.3.7]{smo}. By a direct computation one gets
\begin{align*}
&\Sigma_5(x_1x_2^7x_3^7x_4^{15})^{\Sigma_5} = \langle[q_1]_{\omega^*}\rangle\\
&\Sigma_5(x_1^3x_2^5x_3^7x_4^{15})^{\Sigma_5} = \langle[q_2]_{\omega^*}\rangle\\
&\Sigma_5(x_1^3x_2^7x_3^7x_4^{13})^{\Sigma_5} = 0\\
&\Sigma_5(x_1x_2x_3^6x_4^7x_5^{15}, x_1x_2^2x_3^5x_4^7x_5^{15})^{\Sigma_5} = \langle[q_3]_{\omega^*}\rangle\\
&\Sigma_5(x_1x_2^3x_3^5x_4^6x_5^{15})^{\Sigma_5} = \langle[q_4]_{\omega^*}\rangle\\
&\mathcal M^{\Sigma_5} = \langle [q_5]_{\omega^*}, [q_6]_{\omega^*}, [q_7]_{\omega^*}\rangle
\end{align*}

\medskip
We present the detailed proof for the relation $\Sigma_5(x_1x_2^3x_3^5x_4^6x_5^{15})^{\Sigma_5} = \langle[q_4]_{\omega^*}\rangle$. We have  $\Sigma_5(x_1x_2^3x_3^5x_4^6x_5^{15}) = \langle [d_t]_{\omega^*}: 1 \leqslant t \leqslant 60\rangle$, where the monomials $d_t,\, 11 \leqslant t \leqslant 60$, are determined as follows:

\medskip
\centerline{\begin{tabular}{llll}
$1.\ x_1x_2^{3}x_3^{5}x_4^{6}x_5^{15}$& $2.\ x_1x_2^{3}x_3^{5}x_4^{15}x_5^{6}$& $3.\ x_1x_2^{3}x_3^{6}x_4^{5}x_5^{15}$& $4.\ x_1x_2^{3}x_3^{6}x_4^{15}x_5^{5} $\cr  $5.\ x_1x_2^{3}x_3^{15}x_4^{5}x_5^{6}$& $6.\ x_1x_2^{3}x_3^{15}x_4^{6}x_5^{5}$& $7.\ x_1x_2^{6}x_3^{3}x_4^{5}x_5^{15}$& $8.\ x_1x_2^{6}x_3^{3}x_4^{15}x_5^{5} $\cr  $9.\ x_1x_2^{6}x_3^{15}x_4^{3}x_5^{5}$& $10.\ x_1x_2^{15}x_3^{3}x_4^{5}x_5^{6}$& $11.\ x_1x_2^{15}x_3^{3}x_4^{6}x_5^{5}$& $12.\ x_1x_2^{15}x_3^{6}x_4^{3}x_5^{5} $\cr  $13.\ x_1^{3}x_2x_3^{5}x_4^{6}x_5^{15}$& $14.\ x_1^{3}x_2x_3^{5}x_4^{15}x_5^{6}$& $15.\ x_1^{3}x_2x_3^{6}x_4^{5}x_5^{15}$& $16.\ x_1^{3}x_2x_3^{6}x_4^{15}x_5^{5} $\cr  $17.\ x_1^{3}x_2x_3^{15}x_4^{5}x_5^{6}$& $18.\ x_1^{3}x_2x_3^{15}x_4^{6}x_5^{5}$& $19.\ x_1^{3}x_2^{3}x_3^{4}x_4^{5}x_5^{15}$& $20.\ x_1^{3}x_2^{3}x_3^{4}x_4^{15}x_5^{5} $\cr  $21.\ x_1^{3}x_2^{3}x_3^{5}x_4^{4}x_5^{15}$& $22.\ x_1^{3}x_2^{3}x_3^{5}x_4^{15}x_5^{4}$& $23.\ x_1^{3}x_2^{3}x_3^{15}x_4^{4}x_5^{5}$& $24.\ x_1^{3}x_2^{3}x_3^{15}x_4^{5}x_5^{4} $\cr  $25.\ x_1^{3}x_2^{4}x_3^{3}x_4^{5}x_5^{15}$& $26.\ x_1^{3}x_2^{4}x_3^{3}x_4^{15}x_5^{5}$& $27.\ x_1^{3}x_2^{4}x_3^{15}x_4^{3}x_5^{5}$& $28.\ x_1^{3}x_2^{5}x_3x_4^{6}x_5^{15} $\cr  $29.\ x_1^{3}x_2^{5}x_3x_4^{15}x_5^{6}$& $30.\ x_1^{3}x_2^{5}x_3^{2}x_4^{5}x_5^{15}$& $31.\ x_1^{3}x_2^{5}x_3^{2}x_4^{15}x_5^{5}$& $32.\ x_1^{3}x_2^{5}x_3^{3}x_4^{4}x_5^{15} $\cr  $33.\ x_1^{3}x_2^{5}x_3^{3}x_4^{15}x_5^{4}$& $34.\ x_1^{3}x_2^{5}x_3^{6}x_4x_5^{15}$& $35.\ x_1^{3}x_2^{5}x_3^{6}x_4^{15}x_5$& $36.\ x_1^{3}x_2^{5}x_3^{15}x_4x_5^{6} $\cr  $37.\ x_1^{3}x_2^{5}x_3^{15}x_4^{2}x_5^{5}$& $38.\ x_1^{3}x_2^{5}x_3^{15}x_4^{3}x_5^{4}$& $39.\ x_1^{3}x_2^{5}x_3^{15}x_4^{6}x_5$& $40.\ x_1^{3}x_2^{15}x_3x_4^{5}x_5^{6} $\cr  $41.\ x_1^{3}x_2^{15}x_3x_4^{6}x_5^{5}$& $42.\ x_1^{3}x_2^{15}x_3^{3}x_4^{4}x_5^{5}$& $43.\ x_1^{3}x_2^{15}x_3^{3}x_4^{5}x_5^{4}$& $44.\ x_1^{3}x_2^{15}x_3^{4}x_4^{3}x_5^{5} $\cr  $45.\ x_1^{3}x_2^{15}x_3^{5}x_4x_5^{6}$& $46.\ x_1^{3}x_2^{15}x_3^{5}x_4^{2}x_5^{5}$& $47.\ x_1^{3}x_2^{15}x_3^{5}x_4^{3}x_5^{4}$& $48.\ x_1^{3}x_2^{15}x_3^{5}x_4^{6}x_5 $\cr  $49.\ x_1^{15}x_2x_3^{3}x_4^{5}x_5^{6}$& $50.\ x_1^{15}x_2x_3^{3}x_4^{6}x_5^{5}$& $51.\ x_1^{15}x_2x_3^{6}x_4^{3}x_5^{5}$& $52.\ x_1^{15}x_2^{3}x_3x_4^{5}x_5^{6} $\cr  $53.\ x_1^{15}x_2^{3}x_3x_4^{6}x_5^{5}$& $54.\ x_1^{15}x_2^{3}x_3^{3}x_4^{4}x_5^{5}$& $55.\ x_1^{15}x_2^{3}x_3^{3}x_4^{5}x_5^{4}$& $56.\ x_1^{15}x_2^{3}x_3^{4}x_4^{3}x_5^{5} $\cr  $57.\ x_1^{15}x_2^{3}x_3^{5}x_4x_5^{6}$& $58.\ x_1^{15}x_2^{3}x_3^{5}x_4^{2}x_5^{5}$& $59.\ x_1^{15}x_2^{3}x_3^{5}x_4^{3}x_5^{4}$& $60.\ x_1^{15}x_2^{3}x_3^{5}x_4^{6}x_5. $\cr  
\end{tabular}} 

Suppose $f \in P_5(\omega^*)$ and $[f]_{\omega^*} \in  \Sigma_5(x_1x_2^3x_3^5x_4^6x_5^{15})^{\Sigma_5}$. Then we have
$f \equiv_{\omega^*} \sum_{t=1}^{60}d_t$ with $\gamma_t \in \mathbb F_2$ and $\rho_j(f)+f \equiv_{\omega^*} 0$, $1 \leqslant j \leqslant 4$. 

By a direct computation we get
\begin{align*}
\rho_1(f) + f &\equiv_{\omega^*} \gamma_{\{1,13\}}d_{1} + \gamma_{\{2,14\}}d_{2} + \gamma_{\{3,15,25\}}d_{3} + \gamma_{\{4,16,26\}}d_{4} + \gamma_{\{5,17\}}d_{5}\\ 
&\quad + \gamma_{\{6,18,27\}}d_{6} + \gamma_{\{7,25\}}d_{7} + \gamma_{\{8,26\}}d_{8} + \gamma_{\{9,27\}}d_{9} + \gamma_{\{10,49\}}d_{10}\\ 
&\quad + \gamma_{\{11,50\}}d_{11} + \gamma_{\{12,51\}}d_{12} + \gamma_{\{1,13\}}d_{13} + \gamma_{\{2,14\}}d_{14} + \gamma_{\{3,7,15\}}d_{15}\\ 
&\quad + \gamma_{\{4,8,16\}}d_{16} + \gamma_{\{5,17\}}d_{17} + \gamma_{\{6,9,18\}}d_{18} + \gamma_{30}d_{19} + \gamma_{31}d_{20}\\ 
&\quad + \gamma_{32}d_{21} + \gamma_{33}d_{22} + \gamma_{37}d_{23} + \gamma_{38}d_{24} + \gamma_{\{7,25\}}d_{25} + \gamma_{\{8,26\}}d_{26}\\ 
&\quad + \gamma_{\{9,27\}}d_{27} + \gamma_{\{40,52\}}d_{40} + \gamma_{\{41,53\}}d_{41} + \gamma_{\{42,54\}}d_{42} + \gamma_{\{43,55\}}d_{43}\\ 
&\quad + \gamma_{\{44,56\}}d_{44} + \gamma_{\{45,57\}}d_{45} + \gamma_{\{46,58\}}d_{46} + \gamma_{\{47,59\}}d_{47} + \gamma_{\{48,60\}}d_{48}\\ 
&\quad + \gamma_{\{10,49\}}d_{49} + \gamma_{\{11,50\}}d_{50} + \gamma_{\{12,51\}}d_{51} + \gamma_{\{40,52\}}d_{52} + \gamma_{\{41,53\}}d_{53}\\ 
&\quad + \gamma_{\{42,54\}}d_{54} + \gamma_{\{43,55\}}d_{55} + \gamma_{\{44,56\}}d_{56} + \gamma_{\{45,57\}}d_{57} + \gamma_{\{46,58\}}d_{58}\\ 
&\quad + \gamma_{\{47,59\}}d_{59} + \gamma_{\{48,60\}}d_{60}  \equiv_{\omega^*} 0,\\
\rho_2(f) + f &\equiv_{\omega^*} \gamma_{\{3,7\}}d_{3} + \gamma_{\{4,8\}}d_{4} + \gamma_{\{5,10\}}d_{5} + \gamma_{\{6,11\}}d_{6} + \gamma_{\{3,7\}}d_{7} + \gamma_{\{4,8\}}d_{8}\\ 
&\quad + \gamma_{\{9,12\}}d_{9} + \gamma_{\{5,10\}}d_{10} + \gamma_{\{6,11\}}d_{11} + \gamma_{\{9,12\}}d_{12} + \gamma_{\{13,28,30\}}d_{13}\\ 
&\quad + \gamma_{\{14,29,31\}}d_{14} + \gamma_{\{15,30\}}d_{15} + \gamma_{\{16,31\}}d_{16} + \gamma_{\{17,40\}}d_{17}\\ 
&\quad + \gamma_{\{18,41\}}d_{18} + \gamma_{\{19,25\}}d_{19} + \gamma_{\{20,26\}}d_{20} + \gamma_{\{21,32,34\}}d_{21}\\ 
&\quad + \gamma_{\{22,33,35\}}d_{22} + \gamma_{\{23,42\}}d_{23} + \gamma_{\{24,43\}}d_{24} + \gamma_{\{19,25\}}d_{25}\\ 
&\quad + \gamma_{\{20,26\}}d_{26} + \gamma_{\{27,44\}}d_{27} + \gamma_{\{13,15,28\}}d_{28} + \gamma_{\{14,16,29\}}d_{29}\\ 
&\quad + \gamma_{\{15,30\}}d_{30} + \gamma_{\{16,31\}}d_{31} + \gamma_{\{21,32,34\}}d_{32} + \gamma_{\{22,33,35\}}d_{33}\\ 
&\quad + \gamma_{\{36,45\}}d_{36} + \gamma_{\{37,46\}}d_{37} + \gamma_{\{38,47\}}d_{38} + \gamma_{\{39,48\}}d_{39} + \gamma_{\{17,40\}}d_{40}\\ 
&\quad + \gamma_{\{18,41\}}d_{41} + \gamma_{\{23,42\}}d_{42} + \gamma_{\{24,43\}}d_{43} + \gamma_{\{27,44\}}d_{44} + \gamma_{\{36,45\}}d_{45}\\ 
&\quad + \gamma_{\{37,46\}}d_{46} + \gamma_{\{38,47\}}d_{47} + \gamma_{\{39,48\}}d_{48} + \gamma_{\{49,52\}}d_{49}\\ 
&\quad + \gamma_{\{50,53,56\}}d_{50} + \gamma_{\{51,56\}}d_{51} + \gamma_{\{49,52\}}d_{52} + \gamma_{\{50,51,53\}}d_{53}\\ 
&\quad + \gamma_{58}d_{54} + \gamma_{59}d_{55} + \gamma_{\{51,56\}}d_{56} \equiv_{\omega^*} 0,\\
\rho_3(f) + f &\equiv_{\omega^*} \gamma_{\{1,3\}}d_{1} + \gamma_{\{2,5\}}d_{2} + \gamma_{\{1,3\}}d_{3} + \gamma_{\{4,6\}}d_{4} + \gamma_{\{2,5\}}d_{5} + \gamma_{\{4,6\}}d_{6}\\ 
&\quad + \gamma_{\{8,9\}}d_{8} + \gamma_{\{8,9\}}d_{9} + \gamma_{\{11,12\}}d_{11} + \gamma_{\{11,12\}}d_{12} + \gamma_{\{13,15,25\}}d_{13}\\ 
&\quad + \gamma_{\{14,17\}}d_{14} + \gamma_{\{13,15,25\}}d_{15} + \gamma_{\{16,18\}}d_{16} + \gamma_{\{14,17\}}d_{17} + \gamma_{\{16,18\}}d_{18}\\ 
&\quad + \gamma_{\{19,21,32\}}d_{19} + \gamma_{\{20,23\}}d_{20} + \gamma_{\{19,21,30\}}d_{21} + \gamma_{\{22,24\}}d_{22}\\ 
&\quad + \gamma_{\{20,23\}}d_{23} + \gamma_{\{22,24\}}d_{24} + \gamma_{\{26,27\}}d_{26} + \gamma_{\{26,27\}}d_{27} + \gamma_{\{28,34\}}d_{28}\\ 
&\quad + \gamma_{\{29,36\}}d_{29} + \gamma_{\{30,32\}}d_{30} + \gamma_{\{31,37\}}d_{31} + \gamma_{\{30,32\}}d_{32} + \gamma_{\{33,38\}}d_{33}\\ 
&\quad + \gamma_{\{28,34\}}d_{34} + \gamma_{\{35,39\}}d_{35} + \gamma_{\{29,36\}}d_{36} + \gamma_{\{31,37\}}d_{37} + \gamma_{\{33,38\}}d_{38}\\ 
&\quad + \gamma_{\{35,39\}}d_{39} + \gamma_{\{40,45,46\}}d_{40} + \gamma_{\{41,46\}}d_{41} + \gamma_{\{42,44\}}d_{42}\\ 
&\quad + \gamma_{\{43,47,48\}}d_{43} + \gamma_{\{42,44\}}d_{44} + \gamma_{\{40,41,45\}}d_{45} + \gamma_{\{41,46\}}d_{46}\\ 
&\quad + \gamma_{\{43,47,48\}}d_{47} + \gamma_{\{50,51\}}d_{50} + \gamma_{\{50,51\}}d_{51} + \gamma_{\{52,57,58\}}d_{52}\\ 
&\quad + \gamma_{\{53,58\}}d_{53} + \gamma_{\{54,56\}}d_{54} + \gamma_{\{55,59,60\}}d_{55} + \gamma_{\{54,56\}}d_{56}\\ 
&\quad + \gamma_{\{52,53,57\}}d_{57} + \gamma_{\{53,58\}}d_{58} + \gamma_{\{55,59,60\}}d_{59} \equiv_{\omega^*} 0,\\
\rho_4(f) + f &\equiv_{\omega^*} \gamma_{\{1,2\}}d_{1} + \gamma_{\{1,2\}}d_{2} + \gamma_{\{3,4\}}d_{3} + \gamma_{\{3,4\}}d_{4} + \gamma_{\{5,6\}}d_{5} + \gamma_{\{5,6\}}d_{6}\\ 
&\quad + \gamma_{\{7,8\}}d_{7} + \gamma_{\{7,8\}}d_{8} + \gamma_{\{10,11\}}d_{10} + \gamma_{\{10,11\}}d_{11} + \gamma_{\{13,14\}}d_{13}\\ 
&\quad + \gamma_{\{13,14\}}d_{14} + \gamma_{\{15,16\}}d_{15} + \gamma_{\{15,16\}}d_{16} + \gamma_{\{17,18,27\}}d_{17}\\ 
&\quad + \gamma_{\{17,18,27\}}d_{18} + \gamma_{\{19,20\}}d_{19} + \gamma_{\{19,20\}}d_{20} + \gamma_{\{21,22\}}d_{21}\\ 
&\quad + \gamma_{\{21,22\}}d_{22} + \gamma_{\{23,24,38\}}d_{23} + \gamma_{\{23,24,37\}}d_{24} + \gamma_{\{25,26\}}d_{25}\\ 
&\quad + \gamma_{\{25,26\}}d_{26} + \gamma_{\{28,29\}}d_{28} + \gamma_{\{28,29\}}d_{29} + \gamma_{\{30,31\}}d_{30} + \gamma_{\{30,31\}}d_{31}\\ 
&\quad + \gamma_{\{32,33\}}d_{32} + \gamma_{\{32,33\}}d_{33} + \gamma_{\{34,35\}}d_{34} + \gamma_{\{34,35\}}d_{35} + \gamma_{\{36,39\}}d_{36}\\ 
&\quad + \gamma_{\{37,38\}}d_{37} + \gamma_{\{37,38\}}d_{38} + \gamma_{\{36,39\}}d_{39} + \gamma_{\{40,41,44\}}d_{40}\\ 
&\quad + \gamma_{\{40,41,44\}}d_{41} + \gamma_{\{42,43,47\}}d_{42} + \gamma_{\{42,43,46\}}d_{43} + \gamma_{\{45,48\}}d_{45}\\ 
&\quad + \gamma_{\{46,47\}}d_{46} + \gamma_{\{46,47\}}d_{47} + \gamma_{\{45,48\}}d_{48} + \gamma_{\{49,50\}}d_{49} + \gamma_{\{49,50\}}d_{50}\\ 
&\quad + \gamma_{\{52,53,56\}}d_{52} + \gamma_{\{52,53,56\}}d_{53} + \gamma_{\{54,55,59\}}d_{54} + \gamma_{\{54,55,58\}}d_{55}\\ 
&\quad + \gamma_{\{57,60\}}d_{57} + \gamma_{\{58,59\}}d_{58} + \gamma_{\{58,59\}}d_{59} + \gamma_{\{57,60\}}d_{60} \equiv_{\omega^*} 0.
\end{align*}
From the above equalities we get $\gamma_t = 0$ for $t\in 
\mathbb J = \{$15,\, 16,\, 18,\, 30,\, 31,\, 32,\, 33,\, 37,\, 38,\, 41,\, 46,\, 47,\, 53,\, 58,\, 59$\}$ and $\gamma_t = \gamma_1$ for $t \notin \mathbb J$. Hence $f \equiv_{\omega^*} q_4$. The lemma ia proved.
\end{proof}
\begin{props}\label{mdt33} We have $QP_5((4)|(3)|^2|(1))^{GL_5} = 0.$
\end{props}
\begin{proof} Let $h \in P_5(\omega^*)$ such that $[h]_{\omega^*} \in  QP_5((4)|(3)|^2|(1))^{GL_5}$. Then we have $[h]_{\omega^*} \in  QP_5((4)|(3)|^2|(1))^{\Sigma_5}$. By Lemma \ref{bdt3} we have
$h \equiv_{\omega^*} \sum_{u=1}^{7}q_u$ with $\gamma_u \in \mathbb F_2$. By a direct computation we get
\begin{align*}
\rho_5(h) + h &\equiv_{\omega^*} \gamma_{1}x_2x_3^{7}x_4^{7}x_5^{15} + \gamma_{2}x_2^{3}x_3^{5}x_4^{7}x_5^{15} + \gamma_{3}x_1x_2^{15}x_3x_4^{6}x_5^{7}\\  &\quad + \gamma_{4}x_1x_2^{15}x_3^{3}x_4^{5}x_5^{6} + \gamma_{5}x_1^{3}x_2^{7}x_3x_4^{7}x_5^{12} + \gamma_{6}x_1^{3}x_2^{7}x_3x_4^{5}x_5^{14}\\  &\quad + \gamma_{7}x_1^{3}x_2^{7}x_3^{3}x_4^{5}x_5^{12} + \mbox{ other terms } \equiv_{\omega^*} 0.
\end{align*}
This equality implies $\gamma_u = 0$ for $1 \leqslant u \leqslant 7$. The proposition is proved.
\end{proof}

\subsection{Proof of Theorem \ref{thm2}}\

\medskip
It is not difficult to check that $\rho_t(q) + q \equiv 0$ for $1 \leqslant t \leqslant 5$. Hence, $[q]\in (QP_5)_{30}^{GL_5}$.

Suppose $g \in P_5$ and $[g] \in (QP_5)_{30}^{GL_5}$. By Proposition \ref{mdt33}, $[g]_{(4,3,3,1)} = 0$, hence $[g]_{(2,4,3,1)}\in QP_5(2,4,3,1)^{GL_5}$. By using Proposition \ref{mdt32}, we get 
$$g \equiv \gamma x_1^3x_2^5x_3^6x_4^6x_5^{10} + \sum_{b\in B((2)^4)}\gamma_b.b,$$
where $\gamma,\, \gamma_b \in \mathbb F_2$ for all $b\in B((2)^4)$. By a direct computation, we see that $\rho_t(g)+g \equiv 0$ for $1 \leqslant t < 5$ if and only if 
$$ f \equiv \gamma(x_1^3x_2^5x_3^6x_4^6x_5^{10}+p_0) + \sum_{1 \leqslant s \leqslant 9}\gamma_sp_s,$$
where $\gamma_s \in \mathbb F_2$ for $1 \leqslant s \leqslant 9,$ and
\begin{align*}
p_0 &= x_1x_2^{2}x_3x_4^{14}x_5^{12} + x_1x_2^{2}x_3^{3}x_4^{12}x_5^{12} + x_1x_2^{2}x_3^{12}x_4x_5^{14} + x_1x_2^{2}x_3^{13}x_4^{2}x_5^{12}\\ &\quad + x_1x_2^{3}x_3^{2}x_4^{12}x_5^{12} + x_1x_2^{3}x_3^{4}x_4^{10}x_5^{12} + x_1x_2^{3}x_3^{6}x_4^{8}x_5^{12} + x_1x_2^{3}x_3^{12}x_4^{2}x_5^{12}\\ &\quad + x_1x_2^{14}x_3x_4^{2}x_5^{12} + x_1^{3}x_2x_3^{4}x_4^{10}x_5^{12} + x_1^{3}x_2x_3^{6}x_4^{8}x_5^{12} + x_1^{3}x_2x_3^{12}x_4^{2}x_5^{12}\\ &\quad + x_1^{3}x_2^{5}x_3^{2}x_4^{8}x_5^{12}.
\end{align*}
Now, computing from the relation $\rho_5(g) + g \equiv 0$, we obtain $\gamma_s = \gamma$ for $1 \leqslant s \leqslant 9$. Then, we have 
$$g \equiv \gamma\left(x_1^3x_2^5x_3^6x_4^6x_5^{10}+ \sum_{0\leqslant s \leqslant 9}p_s\right) = \gamma q.$$
The theorem is proved.

%=========================

\section{Appendix}\label{s5}
\setcounter{equation}{0}

\subsection{The admissible monomials of degree 20 in $P_5$}\label{s50}\

\medskip
In this subsection we list the needed admissible monomials of degree $20$ in $P_5$. The elements in
$B_5^0(20)$ can easily determined by using the results in \cite{su5} and the relation \ref{ctbs2}. We have $B_4(20) = \{w_t : 1 \leqslant t \leqslant 55\}$, where

\medskip
\centerline{\begin{tabular}{lllll}
$v_{1} =  x_1x_2x_3^{3}x_4^{15}$& $v_{2} =  x_1x_2x_3^{7}x_4^{11}$& $v_{3} =  x_1x_2x_3^{15}x_4^{3}$& $v_{4} =  x_1x_2^{3}x_3x_4^{15} $\cr  $v_{5} =  x_1x_2^{3}x_3^{3}x_4^{13}$& $v_{6} =  x_1x_2^{3}x_3^{5}x_4^{11}$& $v_{7} =  x_1x_2^{3}x_3^{7}x_4^{9}$& $v_{8} =  x_1x_2^{3}x_3^{13}x_4^{3} $\cr  $v_{9} =  x_1x_2^{3}x_3^{15}x_4$& $v_{10} =  x_1x_2^{7}x_3x_4^{11}$& $v_{11} =  x_1x_2^{7}x_3^{3}x_4^{9}$& $v_{12} =  x_1x_2^{7}x_3^{9}x_4^{3} $\cr  $v_{13} =  x_1x_2^{7}x_3^{11}x_4$& $v_{14} =  x_1x_2^{15}x_3x_4^{3}$& $v_{15} =  x_1x_2^{15}x_3^{3}x_4$& $v_{16} =  x_1^{3}x_2x_3x_4^{15} $\cr  $v_{17} =  x_1^{3}x_2x_3^{3}x_4^{13}$& $v_{18} =  x_1^{3}x_2x_3^{5}x_4^{11}$& $v_{19} =  x_1^{3}x_2x_3^{7}x_4^{9}$& $v_{20} =  x_1^{3}x_2x_3^{13}x_4^{3} $\cr  $v_{21} =  x_1^{3}x_2x_3^{15}x_4$& $v_{22} =  x_1^{3}x_2^{3}x_3x_4^{13}$& $v_{23} =  x_1^{3}x_2^{3}x_3^{5}x_4^{9}$& $v_{24} =  x_1^{3}x_2^{3}x_3^{13}x_4 $\cr  $v_{25} =  x_1^{3}x_2^{5}x_3x_4^{11}$& $v_{26} =  x_1^{3}x_2^{5}x_3^{3}x_4^{9}$& $v_{27} =  x_1^{3}x_2^{5}x_3^{9}x_4^{3}$& $v_{28} =  x_1^{3}x_2^{5}x_3^{11}x_4 $\cr  $v_{29} =  x_1^{3}x_2^{7}x_3x_4^{9}$& $v_{30} =  x_1^{3}x_2^{7}x_3^{9}x_4$& $v_{31} =  x_1^{3}x_2^{13}x_3x_4^{3}$& $v_{32} =  x_1^{3}x_2^{13}x_3^{3}x_4 $\cr  $v_{33} =  x_1^{3}x_2^{15}x_3x_4$& $v_{34} =  x_1^{7}x_2x_3x_4^{11}$& $v_{35} =  x_1^{7}x_2x_3^{3}x_4^{9}$& $v_{36} =  x_1^{7}x_2x_3^{9}x_4^{3} $\cr  $v_{37} =  x_1^{7}x_2x_3^{11}x_4$& $v_{38} =  x_1^{7}x_2^{3}x_3x_4^{9}$& $v_{39} =  x_1^{7}x_2^{3}x_3^{9}x_4$& $v_{40} =  x_1^{7}x_2^{9}x_3x_4^{3} $\cr  $v_{41} =  x_1^{7}x_2^{9}x_3^{3}x_4$& $v_{42} =  x_1^{7}x_2^{11}x_3x_4$& $v_{43} =  x_1^{15}x_2x_3x_4^{3}$& $v_{44} =  x_1^{15}x_2x_3^{3}x_4 $\cr  $v_{45} =  x_1^{15}x_2^{3}x_3x_4$& $v_{46} =  x_1^{3}x_2^{5}x_3^{5}x_4^{7}$& $v_{47} =  x_1^{3}x_2^{5}x_3^{7}x_4^{5}$& $v_{48} =  x_1^{3}x_2^{7}x_3^{5}x_4^{5} $\cr  $v_{49} =  x_1^{7}x_2^{3}x_3^{5}x_4^{5}$& $v_{50} =  x_1^{3}x_2^{3}x_3^{7}x_4^{7}$& $v_{51} =  x_1^{3}x_2^{7}x_3^{3}x_4^{7}$& $v_{52} =  x_1^{3}x_2^{7}x_3^{7}x_4^{3} $\cr  $v_{53} =  x_1^{7}x_2^{3}x_3^{3}x_4^{7}$& $v_{54} =  x_1^{7}x_2^{3}x_3^{7}x_4^{3}$& $v_{55} =  x_1^{7}x_2^{7}x_3^{3}x_4^{3}$& \cr   
\end{tabular}}

\medskip
Note that $\omega(w_t) = (4)|(2)|(1)|^2$ for $1 \leqslant t \leqslant 45$, $\omega(w_t) = (4)|(2)|(3)$ for $46 \leqslant t \leqslant 49$ and $\omega(w_t) = (4)|^{2}|(2)$ for $50 \leqslant t \leqslant 55$. Hence, we get $|B_5^0((4)|(2)|(1)|^2)| = 225$, $|B_5^0((4)|(2)|(3))| = 20$, $|B_5^0((4)|^{2}|(2))| = 30$ and $|B_5^0(20)| = 275$. So, we list only the elements in $B_5^+(20) = B_5^+((4)|(2)|(1)|^2)\bigcup B_5^+((4)|(2)|(3))\bigcup B_5^+((4)|^2|(2))$. 

\medskip
$B_5^+((4)|(2)|(1)|^2) = \{a_t = a_{t} : 1 \leqslant t \leqslant 225\}$, where

\medskip
\centerline{\begin{tabular}{lllll}
$1.\ x_1x_2x_3x_4^{2}x_5^{15}$& $2.\ x_1x_2x_3x_4^{15}x_5^{2}$& $3.\ x_1x_2x_3^{2}x_4x_5^{15}$& $4.\ x_1x_2x_3^{2}x_4^{15}x_5 $\cr  $5.\ x_1x_2x_3^{15}x_4x_5^{2}$& $6.\ x_1x_2x_3^{15}x_4^{2}x_5$& $7.\ x_1x_2^{2}x_3x_4x_5^{15}$& $8.\ x_1x_2^{2}x_3x_4^{15}x_5 $\cr  $9.\ x_1x_2^{2}x_3^{15}x_4x_5$& $10.\ x_1x_2^{15}x_3x_4x_5^{2}$& $11.\ x_1x_2^{15}x_3x_4^{2}x_5$& $12.\ x_1x_2^{15}x_3^{2}x_4x_5 $\cr  $13.\ x_1^{15}x_2x_3x_4x_5^{2}$& $14.\ x_1^{15}x_2x_3x_4^{2}x_5$& $15.\ x_1^{15}x_2x_3^{2}x_4x_5$& $16.\ x_1x_2x_3x_4^{3}x_5^{14} $\cr  $17.\ x_1x_2x_3x_4^{6}x_5^{11}$& $18.\ x_1x_2x_3x_4^{7}x_5^{10}$& $19.\ x_1x_2x_3x_4^{14}x_5^{3}$& $20.\ x_1x_2x_3^{2}x_4^{3}x_5^{13} $\cr  $21.\ x_1x_2x_3^{2}x_4^{5}x_5^{11}$& $22.\ x_1x_2x_3^{2}x_4^{7}x_5^{9}$& $23.\ x_1x_2x_3^{2}x_4^{13}x_5^{3}$& $24.\ x_1x_2x_3^{3}x_4x_5^{14} $\cr  $25.\ x_1x_2x_3^{3}x_4^{2}x_5^{13}$& $26.\ x_1x_2x_3^{3}x_4^{3}x_5^{12}$& $27.\ x_1x_2x_3^{3}x_4^{4}x_5^{11}$& $28.\ x_1x_2x_3^{3}x_4^{5}x_5^{10} $\cr  $29.\ x_1x_2x_3^{3}x_4^{6}x_5^{9}$& $30.\ x_1x_2x_3^{3}x_4^{7}x_5^{8}$& $31.\ x_1x_2x_3^{3}x_4^{12}x_5^{3}$& $32.\ x_1x_2x_3^{3}x_4^{13}x_5^{2} $\cr  $33.\ x_1x_2x_3^{3}x_4^{14}x_5$& $34.\ x_1x_2x_3^{6}x_4x_5^{11}$& $35.\ x_1x_2x_3^{6}x_4^{3}x_5^{9}$& $36.\ x_1x_2x_3^{6}x_4^{9}x_5^{3} $\cr  $37.\ x_1x_2x_3^{6}x_4^{11}x_5$& $38.\ x_1x_2x_3^{7}x_4x_5^{10}$& $39.\ x_1x_2x_3^{7}x_4^{2}x_5^{9}$& $40.\ x_1x_2x_3^{7}x_4^{3}x_5^{8} $\cr  $41.\ x_1x_2x_3^{7}x_4^{8}x_5^{3}$& $42.\ x_1x_2x_3^{7}x_4^{9}x_5^{2}$& $43.\ x_1x_2x_3^{7}x_4^{10}x_5$& $44.\ x_1x_2x_3^{14}x_4x_5^{3} $\cr  $45.\ x_1x_2x_3^{14}x_4^{3}x_5$& $46.\ x_1x_2^{2}x_3x_4^{3}x_5^{13}$& $47.\ x_1x_2^{2}x_3x_4^{5}x_5^{11}$& $48.\ x_1x_2^{2}x_3x_4^{7}x_5^{9} $\cr  $49.\ x_1x_2^{2}x_3x_4^{13}x_5^{3}$& $50.\ x_1x_2^{2}x_3^{3}x_4x_5^{13}$& $51.\ x_1x_2^{2}x_3^{3}x_4^{5}x_5^{9}$& $52.\ x_1x_2^{2}x_3^{3}x_4^{13}x_5 $\cr  $53.\ x_1x_2^{2}x_3^{5}x_4x_5^{11}$& $54.\ x_1x_2^{2}x_3^{5}x_4^{3}x_5^{9}$& $55.\ x_1x_2^{2}x_3^{5}x_4^{9}x_5^{3}$& $56.\ x_1x_2^{2}x_3^{5}x_4^{11}x_5 $\cr  $57.\ x_1x_2^{2}x_3^{7}x_4x_5^{9}$& $58.\ x_1x_2^{2}x_3^{7}x_4^{9}x_5$& $59.\ x_1x_2^{2}x_3^{13}x_4x_5^{3}$& $60.\ x_1x_2^{2}x_3^{13}x_4^{3}x_5 $\cr  $61.\ x_1x_2^{3}x_3x_4x_5^{14}$& $62.\ x_1x_2^{3}x_3x_4^{2}x_5^{13}$& $63.\ x_1x_2^{3}x_3x_4^{3}x_5^{12}$& $64.\ x_1x_2^{3}x_3x_4^{4}x_5^{11} $\cr  $65.\ x_1x_2^{3}x_3x_4^{5}x_5^{10}$& $66.\ x_1x_2^{3}x_3x_4^{6}x_5^{9}$& $67.\ x_1x_2^{3}x_3x_4^{7}x_5^{8}$& $68.\ x_1x_2^{3}x_3x_4^{12}x_5^{3} $\cr  $69.\ x_1x_2^{3}x_3x_4^{13}x_5^{2}$& $70.\ x_1x_2^{3}x_3x_4^{14}x_5$& $71.\ x_1x_2^{3}x_3^{2}x_4x_5^{13}$& $72.\ x_1x_2^{3}x_3^{2}x_4^{5}x_5^{9} $\cr  $73.\ x_1x_2^{3}x_3^{2}x_4^{13}x_5$& $74.\ x_1x_2^{3}x_3^{3}x_4x_5^{12}$& $75.\ x_1x_2^{3}x_3^{3}x_4^{4}x_5^{9}$& $76.\ x_1x_2^{3}x_3^{3}x_4^{5}x_5^{8} $\cr  $77.\ x_1x_2^{3}x_3^{3}x_4^{12}x_5$& $78.\ x_1x_2^{3}x_3^{4}x_4x_5^{11}$& $79.\ x_1x_2^{3}x_3^{4}x_4^{3}x_5^{9}$& $80.\ x_1x_2^{3}x_3^{4}x_4^{9}x_5^{3} $\cr  $81.\ x_1x_2^{3}x_3^{4}x_4^{11}x_5$& $82.\ x_1x_2^{3}x_3^{5}x_4x_5^{10}$& $83.\ x_1x_2^{3}x_3^{5}x_4^{2}x_5^{9}$& $84.\ x_1x_2^{3}x_3^{5}x_4^{3}x_5^{8} $\cr $85.\ x_1x_2^{3}x_3^{5}x_4^{8}x_5^{3}$& $86.\ x_1x_2^{3}x_3^{5}x_4^{9}x_5^{2}$& $87.\ x_1x_2^{3}x_3^{5}x_4^{10}x_5$& $88.\ x_1x_2^{3}x_3^{6}x_4x_5^{9} $\cr  
\end{tabular}} 
\centerline{\begin{tabular}{llll}
$89.\ x_1x_2^{3}x_3^{6}x_4^{9}x_5$& $90.\ x_1x_2^{3}x_3^{7}x_4x_5^{8}$& $91.\ x_1x_2^{3}x_3^{7}x_4^{8}x_5$& $92.\ x_1x_2^{3}x_3^{12}x_4x_5^{3} $\cr  $93.\ x_1x_2^{3}x_3^{12}x_4^{3}x_5$& $94.\ x_1x_2^{3}x_3^{13}x_4x_5^{2}$& $95.\ x_1x_2^{3}x_3^{13}x_4^{2}x_5$& $96.\ x_1x_2^{3}x_3^{14}x_4x_5 $\cr  $97.\ x_1x_2^{6}x_3x_4x_5^{11}$& $98.\ x_1x_2^{6}x_3x_4^{3}x_5^{9}$& $99.\ x_1x_2^{6}x_3x_4^{9}x_5^{3}$& $100.\ x_1x_2^{6}x_3x_4^{11}x_5 $\cr   $101.\ x_1x_2^{6}x_3^{3}x_4x_5^{9}$& $102.\ x_1x_2^{6}x_3^{3}x_4^{9}x_5$& $103.\ x_1x_2^{6}x_3^{9}x_4x_5^{3}$& $104.\ x_1x_2^{6}x_3^{9}x_4^{3}x_5 $\cr $105.\ x_1x_2^{6}x_3^{11}x_4x_5$& $106.\ x_1x_2^{7}x_3x_4x_5^{10}$& $107.\ x_1x_2^{7}x_3x_4^{2}x_5^{9}$& $108.\ x_1x_2^{7}x_3x_4^{3}x_5^{8} $\cr   $109.\ x_1x_2^{7}x_3x_4^{8}x_5^{3}$& $110.\ x_1x_2^{7}x_3x_4^{9}x_5^{2}$& $111.\ x_1x_2^{7}x_3x_4^{10}x_5$& $112.\ x_1x_2^{7}x_3^{2}x_4x_5^{9} $\cr  $113.\ x_1x_2^{7}x_3^{2}x_4^{9}x_5$& $114.\ x_1x_2^{7}x_3^{3}x_4x_5^{8}$& $115.\ x_1x_2^{7}x_3^{3}x_4^{8}x_5$& $116.\ x_1x_2^{7}x_3^{8}x_4x_5^{3} $\cr $117.\ x_1x_2^{7}x_3^{8}x_4^{3}x_5$& $118.\ x_1x_2^{7}x_3^{9}x_4x_5^{2}$& $119.\ x_1x_2^{7}x_3^{9}x_4^{2}x_5$& $120.\ x_1x_2^{7}x_3^{10}x_4x_5 $\cr  $121.\ x_1x_2^{14}x_3x_4x_5^{3}$& $122.\ x_1x_2^{14}x_3x_4^{3}x_5$& $123.\ x_1x_2^{14}x_3^{3}x_4x_5$& $124.\ x_1^{3}x_2x_3x_4x_5^{14} $\cr  $125.\ x_1^{3}x_2x_3x_4^{2}x_5^{13}$& $126.\ x_1^{3}x_2x_3x_4^{3}x_5^{12}$& $127.\ x_1^{3}x_2x_3x_4^{4}x_5^{11}$& $128.\ x_1^{3}x_2x_3x_4^{5}x_5^{10} $\cr  $129.\ x_1^{3}x_2x_3x_4^{6}x_5^{9}$& $130.\ x_1^{3}x_2x_3x_4^{7}x_5^{8}$& $131.\ x_1^{3}x_2x_3x_4^{12}x_5^{3}$& $132.\ x_1^{3}x_2x_3x_4^{13}x_5^{2} $\cr  $133.\ x_1^{3}x_2x_3x_4^{14}x_5$& $134.\ x_1^{3}x_2x_3^{2}x_4x_5^{13}$& $135.\ x_1^{3}x_2x_3^{2}x_4^{5}x_5^{9}$& $136.\ x_1^{3}x_2x_3^{2}x_4^{13}x_5 $\cr  $137.\ x_1^{3}x_2x_3^{3}x_4x_5^{12}$& $138.\ x_1^{3}x_2x_3^{3}x_4^{4}x_5^{9}$& $139.\ x_1^{3}x_2x_3^{3}x_4^{5}x_5^{8}$& $140.\ x_1^{3}x_2x_3^{3}x_4^{12}x_5 $\cr  $141.\ x_1^{3}x_2x_3^{4}x_4x_5^{11}$& $142.\ x_1^{3}x_2x_3^{4}x_4^{3}x_5^{9}$& $143.\ x_1^{3}x_2x_3^{4}x_4^{9}x_5^{3}$& $144.\ x_1^{3}x_2x_3^{4}x_4^{11}x_5 $\cr  $145.\ x_1^{3}x_2x_3^{5}x_4x_5^{10}$& $146.\ x_1^{3}x_2x_3^{5}x_4^{2}x_5^{9}$& $147.\ x_1^{3}x_2x_3^{5}x_4^{3}x_5^{8}$& $148.\ x_1^{3}x_2x_3^{5}x_4^{8}x_5^{3} $\cr  $149.\ x_1^{3}x_2x_3^{5}x_4^{9}x_5^{2}$& $150.\ x_1^{3}x_2x_3^{5}x_4^{10}x_5$& $151.\ x_1^{3}x_2x_3^{6}x_4x_5^{9}$& $152.\ x_1^{3}x_2x_3^{6}x_4^{9}x_5 $\cr  $153.\ x_1^{3}x_2x_3^{7}x_4x_5^{8}$& $154.\ x_1^{3}x_2x_3^{7}x_4^{8}x_5$& $155.\ x_1^{3}x_2x_3^{12}x_4x_5^{3}$& $156.\ x_1^{3}x_2x_3^{12}x_4^{3}x_5 $\cr  $157.\ x_1^{3}x_2x_3^{13}x_4x_5^{2}$& $158.\ x_1^{3}x_2x_3^{13}x_4^{2}x_5$& $159.\ x_1^{3}x_2x_3^{14}x_4x_5$& $160.\ x_1^{3}x_2^{3}x_3x_4x_5^{12} $\cr  $161.\ x_1^{3}x_2^{3}x_3x_4^{4}x_5^{9}$& $162.\ x_1^{3}x_2^{3}x_3x_4^{5}x_5^{8}$& $163.\ x_1^{3}x_2^{3}x_3x_4^{12}x_5$& $164.\ x_1^{3}x_2^{3}x_3^{4}x_4x_5^{9} $\cr  $165.\ x_1^{3}x_2^{3}x_3^{4}x_4^{9}x_5$& $166.\ x_1^{3}x_2^{3}x_3^{5}x_4x_5^{8}$& $167.\ x_1^{3}x_2^{3}x_3^{5}x_4^{8}x_5$& $168.\ x_1^{3}x_2^{3}x_3^{12}x_4x_5 $\cr  $169.\ x_1^{3}x_2^{4}x_3x_4x_5^{11}$& $170.\ x_1^{3}x_2^{4}x_3x_4^{3}x_5^{9}$& $171.\ x_1^{3}x_2^{4}x_3x_4^{11}x_5$& $172.\ x_1^{3}x_2^{4}x_3^{3}x_4x_5^{9} $\cr  $173.\ x_1^{3}x_2^{4}x_3^{3}x_4^{9}x_5$& $174.\ x_1^{3}x_2^{4}x_3^{11}x_4x_5$& $175.\ x_1^{3}x_2^{5}x_3x_4x_5^{10}$& $176.\ x_1^{3}x_2^{5}x_3x_4^{2}x_5^{9} $\cr  $177.\ x_1^{3}x_2^{5}x_3x_4^{3}x_5^{8}$& $178.\ x_1^{3}x_2^{5}x_3x_4^{8}x_5^{3}$& $179.\ x_1^{3}x_2^{5}x_3x_4^{9}x_5^{2}$& $180.\ x_1^{3}x_2^{5}x_3x_4^{10}x_5 $\cr  $181.\ x_1^{3}x_2^{5}x_3^{2}x_4x_5^{9}$& $182.\ x_1^{3}x_2^{5}x_3^{2}x_4^{9}x_5$& $183.\ x_1^{3}x_2^{5}x_3^{3}x_4x_5^{8}$& $184.\ x_1^{3}x_2^{5}x_3^{3}x_4^{8}x_5 $\cr  $185.\ x_1^{3}x_2^{5}x_3^{8}x_4x_5^{3}$& $186.\ x_1^{3}x_2^{5}x_3^{8}x_4^{3}x_5$& $187.\ x_1^{3}x_2^{5}x_3^{9}x_4x_5^{2}$& $188.\ x_1^{3}x_2^{5}x_3^{9}x_4^{2}x_5 $\cr  $189.\ x_1^{3}x_2^{5}x_3^{10}x_4x_5$& $190.\ x_1^{3}x_2^{7}x_3x_4x_5^{8}$& $191.\ x_1^{3}x_2^{7}x_3x_4^{8}x_5$& $192.\ x_1^{3}x_2^{7}x_3^{8}x_4x_5 $\cr  $193.\ x_1^{3}x_2^{13}x_3x_4x_5^{2}$& $194.\ x_1^{3}x_2^{13}x_3x_4^{2}x_5$& $195.\ x_1^{3}x_2^{13}x_3^{2}x_4x_5$& $196.\ x_1^{7}x_2x_3x_4x_5^{10} $\cr  $197.\ x_1^{7}x_2x_3x_4^{2}x_5^{9}$& $198.\ x_1^{7}x_2x_3x_4^{3}x_5^{8}$& $199.\ x_1^{7}x_2x_3x_4^{8}x_5^{3}$& $200.\ x_1^{7}x_2x_3x_4^{9}x_5^{2} $\cr  $201.\ x_1^{7}x_2x_3x_4^{10}x_5$& $202.\ x_1^{7}x_2x_3^{2}x_4x_5^{9}$& $203.\ x_1^{7}x_2x_3^{2}x_4^{9}x_5$& $204.\ x_1^{7}x_2x_3^{3}x_4x_5^{8} $\cr  $205.\ x_1^{7}x_2x_3^{3}x_4^{8}x_5$& $206.\ x_1^{7}x_2x_3^{8}x_4x_5^{3}$& $207.\ x_1^{7}x_2x_3^{8}x_4^{3}x_5$& $208.\ x_1^{7}x_2x_3^{9}x_4x_5^{2} $\cr  $209.\ x_1^{7}x_2x_3^{9}x_4^{2}x_5$& $210.\ x_1^{7}x_2x_3^{10}x_4x_5$& $211.\ x_1^{7}x_2^{3}x_3x_4x_5^{8}$& $212.\ x_1^{7}x_2^{3}x_3x_4^{8}x_5 $\cr  $213.\ x_1^{7}x_2^{3}x_3^{8}x_4x_5$& $214.\ x_1^{7}x_2^{9}x_3x_4x_5^{2}$& $215.\ x_1^{7}x_2^{9}x_3x_4^{2}x_5$& $216.\ x_1^{3}x_2^{4}x_3x_4^{9}x_5^{3} $\cr  $217.\ x_1^{3}x_2^{4}x_3^{9}x_4x_5^{3}$& $218.\ x_1^{3}x_2^{4}x_3^{9}x_4^{3}x_5$& $219.\ x_1^{3}x_2^{12}x_3x_4x_5^{3}$& $220.\ x_1^{3}x_2^{12}x_3x_4^{3}x_5 $\cr  $221.\ x_1^{3}x_2^{12}x_3^{3}x_4x_5$& $222.\ x_1^{7}x_2^{8}x_3x_4x_5^{3}$& $223.\ x_1^{7}x_2^{8}x_3x_4^{3}x_5$& $224.\ x_1^{7}x_2^{8}x_3^{3}x_4x_5 $\cr  $225.\ x_1^{7}x_2^{9}x_3^{2}x_4x_5$& 
\end{tabular}} 
\centerline{\begin{tabular}{llll}
\end{tabular}}

\medskip 
We have $B_5^+(4,2,3) = \{b_t: 1 \leqslant t \leqslant 50\},$ where

\medskip
\centerline{\begin{tabular}{llll}
$1.\ x_1x_2^{2}x_3^{5}x_4^{5}x_5^{7}$& $2.\ x_1x_2^{2}x_3^{5}x_4^{7}x_5^{5}$& $3.\ x_1x_2^{2}x_3^{7}x_4^{5}x_5^{5}$& $4.\ x_1x_2^{7}x_3^{2}x_4^{5}x_5^{5} $\cr  $5.\ x_1^{7}x_2x_3^{2}x_4^{5}x_5^{5}$& $6.\ x_1x_2^{3}x_3^{4}x_4^{5}x_5^{7}$& $7.\ x_1x_2^{3}x_3^{4}x_4^{7}x_5^{5}$& $8.\ x_1x_2^{3}x_3^{5}x_4^{4}x_5^{7} $\cr  $9.\ x_1x_2^{3}x_3^{5}x_4^{7}x_5^{4}$& $10.\ x_1x_2^{3}x_3^{7}x_4^{4}x_5^{5}$& $11.\ x_1x_2^{3}x_3^{7}x_4^{5}x_5^{4}$& $12.\ x_1x_2^{7}x_3^{3}x_4^{4}x_5^{5} $\cr  $13.\ x_1x_2^{7}x_3^{3}x_4^{5}x_5^{4}$& $14.\ x_1^{3}x_2x_3^{4}x_4^{5}x_5^{7}$& $15.\ x_1^{3}x_2x_3^{4}x_4^{7}x_5^{5}$& $16.\ x_1^{3}x_2x_3^{5}x_4^{4}x_5^{7} $\cr  $17.\ x_1^{3}x_2x_3^{5}x_4^{7}x_5^{4}$& $18.\ x_1^{3}x_2x_3^{7}x_4^{4}x_5^{5}$& $19.\ x_1^{3}x_2x_3^{7}x_4^{5}x_5^{4}$& $20.\ x_1^{3}x_2^{5}x_3x_4^{4}x_5^{7} $\cr  $21.\ x_1^{3}x_2^{5}x_3x_4^{7}x_5^{4}$& $22.\ x_1^{3}x_2^{5}x_3^{7}x_4x_5^{4}$& $23.\ x_1^{3}x_2^{7}x_3x_4^{4}x_5^{5}$& $24.\ x_1^{3}x_2^{7}x_3x_4^{5}x_5^{4} $\cr  $25.\ x_1^{3}x_2^{7}x_3^{5}x_4x_5^{4}$& $26.\ x_1^{7}x_2x_3^{3}x_4^{4}x_5^{5}$& $27.\ x_1^{7}x_2x_3^{3}x_4^{5}x_5^{4}$& $28.\ x_1^{7}x_2^{3}x_3x_4^{4}x_5^{5} $\cr  $29.\ x_1^{7}x_2^{3}x_3x_4^{5}x_5^{4}$& $30.\ x_1^{7}x_2^{3}x_3^{5}x_4x_5^{4}$& $31.\ x_1x_2^{3}x_3^{5}x_4^{5}x_5^{6}$& $32.\ x_1x_2^{3}x_3^{5}x_4^{6}x_5^{5} $\cr  $33.\ x_1x_2^{3}x_3^{6}x_4^{5}x_5^{5}$& $34.\ x_1x_2^{6}x_3^{3}x_4^{5}x_5^{5}$& $35.\ x_1^{3}x_2x_3^{5}x_4^{5}x_5^{6}$& $36.\ x_1^{3}x_2x_3^{5}x_4^{6}x_5^{5} $\cr  $37.\ x_1^{3}x_2x_3^{6}x_4^{5}x_5^{5}$& $38.\ x_1^{3}x_2^{5}x_3x_4^{5}x_5^{6}$& $39.\ x_1^{3}x_2^{5}x_3x_4^{6}x_5^{5}$& $40.\ x_1^{3}x_2^{5}x_3^{5}x_4x_5^{6} $\cr  $41.\ x_1^{3}x_2^{5}x_3^{2}x_4^{5}x_5^{5}$& $42.\ x_1^{3}x_2^{5}x_3^{5}x_4^{5}x_5^{2}$& $43.\ x_1^{3}x_2^{3}x_3^{4}x_4^{5}x_5^{5}$& $44.\ x_1^{3}x_2^{3}x_3^{5}x_4^{4}x_5^{5} $\cr  
\end{tabular}} 
\centerline{\begin{tabular}{llll}
$45.\ x_1^{3}x_2^{3}x_3^{5}x_4^{5}x_5^{4}$& $46.\ x_1^{3}x_2^{4}x_3^{3}x_4^{5}x_5^{5}$& $47.\ x_1^{3}x_2^{5}x_3^{3}x_4^{4}x_5^{5}$& $48.\ x_1^{3}x_2^{5}x_3^{3}x_4^{5}x_5^{4} $\cr  $49.\ x_1^{3}x_2^{5}x_3^{5}x_4^{3}x_5^{4}$& $50.\ x_1^{3}x_2^{5}x_3^{5}x_4^{4}x_5^{3}$& &\cr 
\end{tabular}}

\medskip
Note that these monomials have also been explicitly determined in \cite{p24}.

\medskip
$B_5^+(4,4,2) = \{c_t: 1 \leqslant t \leqslant 91\},$ where

\medskip
\centerline{\begin{tabular}{llll}
$1.\ x_1x_2^{2}x_3^{3}x_4^{7}x_5^{7}$& $2.\ x_1x_2^{2}x_3^{7}x_4^{3}x_5^{7}$& $3.\ x_1x_2^{2}x_3^{7}x_4^{7}x_5^{3}$& $4.\ x_1x_2^{3}x_3^{2}x_4^{7}x_5^{7} $\cr  $5.\ x_1x_2^{3}x_3^{7}x_4^{2}x_5^{7}$& $6.\ x_1x_2^{3}x_3^{7}x_4^{7}x_5^{2}$& $7.\ x_1x_2^{7}x_3^{2}x_4^{3}x_5^{7}$& $8.\ x_1x_2^{7}x_3^{2}x_4^{7}x_5^{3} $\cr  $9.\ x_1x_2^{7}x_3^{3}x_4^{2}x_5^{7}$& $10.\ x_1x_2^{7}x_3^{3}x_4^{7}x_5^{2}$& $11.\ x_1x_2^{7}x_3^{7}x_4^{2}x_5^{3}$& $12.\ x_1x_2^{7}x_3^{7}x_4^{3}x_5^{2} $\cr  $13.\ x_1^{3}x_2x_3^{2}x_4^{7}x_5^{7}$& $14.\ x_1^{3}x_2x_3^{7}x_4^{2}x_5^{7}$& $15.\ x_1^{3}x_2x_3^{7}x_4^{7}x_5^{2}$& $16.\ x_1^{3}x_2^{7}x_3x_4^{2}x_5^{7} $\cr  $17.\ x_1^{3}x_2^{7}x_3x_4^{7}x_5^{2}$& $18.\ x_1^{3}x_2^{7}x_3^{7}x_4x_5^{2}$& $19.\ x_1^{7}x_2x_3^{2}x_4^{3}x_5^{7}$& $20.\ x_1^{7}x_2x_3^{2}x_4^{7}x_5^{3} $\cr  $21.\ x_1^{7}x_2x_3^{3}x_4^{2}x_5^{7}$& $22.\ x_1^{7}x_2x_3^{3}x_4^{7}x_5^{2}$& $23.\ x_1^{7}x_2x_3^{7}x_4^{2}x_5^{3}$& $24.\ x_1^{7}x_2x_3^{7}x_4^{3}x_5^{2} $\cr  $25.\ x_1^{7}x_2^{3}x_3x_4^{2}x_5^{7}$& $26.\ x_1^{7}x_2^{3}x_3x_4^{7}x_5^{2}$& $27.\ x_1^{7}x_2^{3}x_3^{7}x_4x_5^{2}$& $28.\ x_1^{7}x_2^{7}x_3x_4^{2}x_5^{3} $\cr  $29.\ x_1^{7}x_2^{7}x_3x_4^{3}x_5^{2}$& $30.\ x_1^{7}x_2^{7}x_3^{3}x_4x_5^{2}$& $31.\ x_1x_2^{3}x_3^{3}x_4^{6}x_5^{7}$& $32.\ x_1x_2^{3}x_3^{3}x_4^{7}x_5^{6} $\cr  $33.\ x_1x_2^{3}x_3^{6}x_4^{3}x_5^{7}$& $34.\ x_1x_2^{3}x_3^{6}x_4^{7}x_5^{3}$& $35.\ x_1x_2^{3}x_3^{7}x_4^{3}x_5^{6}$& $36.\ x_1x_2^{3}x_3^{7}x_4^{6}x_5^{3} $\cr  $37.\ x_1x_2^{6}x_3^{3}x_4^{3}x_5^{7}$& $38.\ x_1x_2^{6}x_3^{3}x_4^{7}x_5^{3}$& $39.\ x_1x_2^{6}x_3^{7}x_4^{3}x_5^{3}$& $40.\ x_1x_2^{7}x_3^{3}x_4^{3}x_5^{6} $\cr  $41.\ x_1x_2^{7}x_3^{3}x_4^{6}x_5^{3}$& $42.\ x_1x_2^{7}x_3^{6}x_4^{3}x_5^{3}$& $43.\ x_1^{3}x_2x_3^{3}x_4^{6}x_5^{7}$& $44.\ x_1^{3}x_2x_3^{3}x_4^{7}x_5^{6} $\cr  $45.\ x_1^{3}x_2x_3^{6}x_4^{3}x_5^{7}$& $46.\ x_1^{3}x_2x_3^{6}x_4^{7}x_5^{3}$& $47.\ x_1^{3}x_2x_3^{7}x_4^{3}x_5^{6}$& $48.\ x_1^{3}x_2x_3^{7}x_4^{6}x_5^{3} $\cr  $49.\ x_1^{3}x_2^{3}x_3x_4^{6}x_5^{7}$& $50.\ x_1^{3}x_2^{3}x_3x_4^{7}x_5^{6}$& $51.\ x_1^{3}x_2^{3}x_3^{7}x_4x_5^{6}$& $52.\ x_1^{3}x_2^{7}x_3x_4^{3}x_5^{6} $\cr  $53.\ x_1^{3}x_2^{7}x_3x_4^{6}x_5^{3}$& $54.\ x_1^{3}x_2^{7}x_3^{3}x_4x_5^{6}$& $55.\ x_1^{7}x_2x_3^{3}x_4^{3}x_5^{6}$& $56.\ x_1^{7}x_2x_3^{3}x_4^{6}x_5^{3} $\cr  $57.\ x_1^{7}x_2x_3^{6}x_4^{3}x_5^{3}$& $58.\ x_1^{7}x_2^{3}x_3x_4^{3}x_5^{6}$& $59.\ x_1^{7}x_2^{3}x_3x_4^{6}x_5^{3}$& $60.\ x_1^{7}x_2^{3}x_3^{3}x_4x_5^{6} $\cr  $61.\ x_1^{3}x_2^{3}x_3^{5}x_4^{2}x_5^{7}$& $62.\ x_1^{3}x_2^{3}x_3^{5}x_4^{7}x_5^{2}$& $63.\ x_1^{3}x_2^{3}x_3^{7}x_4^{5}x_5^{2}$& $64.\ x_1^{3}x_2^{5}x_3^{2}x_4^{3}x_5^{7} $\cr  $65.\ x_1^{3}x_2^{5}x_3^{2}x_4^{7}x_5^{3}$& $66.\ x_1^{3}x_2^{5}x_3^{3}x_4^{2}x_5^{7}$& $67.\ x_1^{3}x_2^{5}x_3^{3}x_4^{7}x_5^{2}$& $68.\ x_1^{3}x_2^{5}x_3^{7}x_4^{2}x_5^{3} $\cr  $69.\ x_1^{3}x_2^{5}x_3^{7}x_4^{3}x_5^{2}$& $70.\ x_1^{3}x_2^{7}x_3^{3}x_4^{5}x_5^{2}$& $71.\ x_1^{3}x_2^{7}x_3^{5}x_4^{2}x_5^{3}$& $72.\ x_1^{3}x_2^{7}x_3^{5}x_4^{3}x_5^{2} $\cr  $73.\ x_1^{7}x_2^{3}x_3^{3}x_4^{5}x_5^{2}$& $74.\ x_1^{7}x_2^{3}x_3^{5}x_4^{2}x_5^{3}$& $75.\ x_1^{7}x_2^{3}x_3^{5}x_4^{3}x_5^{2}$& $76.\ x_1^{3}x_2^{3}x_3^{3}x_4^{4}x_5^{7} $\cr  $77.\ x_1^{3}x_2^{3}x_3^{3}x_4^{7}x_5^{4}$& $78.\ x_1^{3}x_2^{3}x_3^{4}x_4^{3}x_5^{7}$& $79.\ x_1^{3}x_2^{3}x_3^{4}x_4^{7}x_5^{3}$& $80.\ x_1^{3}x_2^{3}x_3^{7}x_4^{3}x_5^{4} $\cr  $81.\ x_1^{3}x_2^{3}x_3^{7}x_4^{4}x_5^{3}$& $82.\ x_1^{3}x_2^{7}x_3^{3}x_4^{3}x_5^{4}$& $83.\ x_1^{3}x_2^{7}x_3^{3}x_4^{4}x_5^{3}$& $84.\ x_1^{7}x_2^{3}x_3^{3}x_4^{3}x_5^{4} $\cr  $85.\ x_1^{7}x_2^{3}x_3^{3}x_4^{4}x_5^{3}$& $86.\ x_1^{3}x_2^{3}x_3^{3}x_4^{5}x_5^{6}$& $87.\ x_1^{3}x_2^{3}x_3^{5}x_4^{3}x_5^{6}$& $88.\ x_1^{3}x_2^{3}x_3^{5}x_4^{6}x_5^{3} $\cr  $89.\ x_1^{3}x_2^{5}x_3^{3}x_4^{3}x_5^{6}$& $90.\ x_1^{3}x_2^{5}x_3^{3}x_4^{6}x_5^{3}$& $91.\ x_1^{3}x_2^{5}x_3^{6}x_4^{3}x_5^{3}$& \cr  
\end{tabular}} 

\subsection{A basis of the space $(\widetilde{SF}_5)_{20}$}\label{s51}\

\smallskip
We have $(\widetilde{SF}_5)_{20} = \langle [p_u]: 1 \leqslant u \leqslant 11\rangle$, where
\begin{align*}
g_1 &= x_1x_2x_3x_4^{3}x_5^{14} + x_1x_2x_3x_4^{14}x_5^{3} + x_1x_2x_3^{6}x_4^{3}x_5^{9} + x_1x_2x_3^{6}x_4^{9}x_5^{3}\\ &\quad + x_1x_2^{6}x_3x_4^{3}x_5^{9} + x_1x_2^{6}x_3x_4^{9}x_5^{3} + x_1^{3}x_2x_3x_4^{3}x_5^{12} + x_1^{3}x_2x_3x_4^{5}x_5^{10}\\ &\quad + x_1^{3}x_2x_3x_4^{6}x_5^{9} + x_1^{3}x_2x_3x_4^{12}x_5^{3} + x_1^{3}x_2x_3^{4}x_4^{3}x_5^{9} + x_1^{3}x_2x_3^{4}x_4^{9}x_5^{3}\\ &\quad + x_1^{3}x_2^{4}x_3x_4^{3}x_5^{9} + x_1^{3}x_2^{4}x_3x_4^{9}x_5^{3},\\
g_2 &= x_1x_2x_3^{3}x_4x_5^{14} + x_1x_2x_3^{3}x_4^{12}x_5^{3} + x_1x_2x_3^{14}x_4x_5^{3} + x_1x_2^{6}x_3^{3}x_4x_5^{9}\\ &\quad + x_1x_2^{6}x_3^{9}x_4x_5^{3} + x_1^{3}x_2x_3^{3}x_4x_5^{12} + x_1^{3}x_2x_3^{3}x_4^{4}x_5^{9} + x_1^{3}x_2x_3^{5}x_4x_5^{10}\\ &\quad + x_1^{3}x_2x_3^{5}x_4^{8}x_5^{3} + x_1^{3}x_2x_3^{6}x_4x_5^{9} + x_1^{3}x_2x_3^{12}x_4x_5^{3} + x_1^{3}x_2^{4}x_3^{3}x_4x_5^{9}\\ &\quad + x_1^{3}x_2^{4}x_3^{9}x_4x_5^{3},\\
g_3 & = x_1x_2x_3^{3}x_4^{3}x_5^{12} + x_1x_2x_3^{3}x_4^{14}x_5 + x_1x_2x_3^{14}x_4^{3}x_5 + x_1x_2^{6}x_3^{3}x_4^{9}x_5\\ &\quad + x_1x_2^{6}x_3^{9}x_4^{3}x_5 + x_1^{3}x_2x_3^{3}x_4^{5}x_5^{8} + x_1^{3}x_2x_3^{3}x_4^{12}x_5 + x_1^{3}x_2x_3^{5}x_4^{3}x_5^{8}\\ &\quad + x_1^{3}x_2x_3^{5}x_4^{10}x_5 + x_1^{3}x_2x_3^{6}x_4^{9}x_5 + x_1^{3}x_2x_3^{12}x_4^{3}x_5 + x_1^{3}x_2^{4}x_3^{3}x_4^{9}x_5\\ &\quad + x_1^{3}x_2^{4}x_3^{9}x_4^{3}x_5,\\
g_4 &= x_1x_2^{3}x_3x_4x_5^{14} + x_1x_2^{3}x_3x_4^{12}x_5^{3} + x_1x_2^{3}x_3^{12}x_4x_5^{3} + x_1x_2^{14}x_3x_4x_5^{3}\\ &\quad + x_1^{3}x_2^{3}x_3x_4x_5^{12} + x_1^{3}x_2^{3}x_3x_4^{4}x_5^{9} + x_1^{3}x_2^{3}x_3^{4}x_4x_5^{9} + x_1^{3}x_2^{5}x_3x_4^{2}x_5^{9}\\ &\quad + x_1^{3}x_2^{5}x_3x_4^{8}x_5^{3} + x_1^{3}x_2^{5}x_3^{2}x_4x_5^{9} + x_1^{3}x_2^{5}x_3^{8}x_4x_5^{3} + x_1^{3}x_2^{12}x_3x_4x_5^{3},\\
g_5 &= x_1x_2^{3}x_3x_4^{3}x_5^{12} + x_1x_2^{3}x_3x_4^{14}x_5 + x_1x_2^{3}x_3^{12}x_4^{3}x_5 + x_1x_2^{14}x_3x_4^{3}x_5\\ &\quad + x_1^{3}x_2^{3}x_3x_4^{5}x_5^{8} + x_1^{3}x_2^{3}x_3x_4^{12}x_5 + x_1^{3}x_2^{3}x_3^{4}x_4^{9}x_5 + x_1^{3}x_2^{5}x_3x_4^{3}x_5^{8}\\ &\quad + x_1^{3}x_2^{5}x_3x_4^{9}x_5^{2} + x_1^{3}x_2^{5}x_3^{2}x_4^{9}x_5 + x_1^{3}x_2^{5}x_3^{8}x_4^{3}x_5 + x_1^{3}x_2^{12}x_3x_4^{3}x_5,\\
g_6 &= x_1x_2^{3}x_3^{3}x_4x_5^{12} + x_1x_2^{3}x_3^{3}x_4^{12}x_5 + x_1x_2^{3}x_3^{14}x_4x_5 + x_1x_2^{14}x_3^{3}x_4x_5\\ &\quad + x_1^{3}x_2^{3}x_3^{5}x_4x_5^{8} + x_1^{3}x_2^{3}x_3^{5}x_4^{8}x_5 + x_1^{3}x_2^{3}x_3^{12}x_4x_5 + x_1^{3}x_2^{5}x_3^{3}x_4x_5^{8}\\ &\quad + x_1^{3}x_2^{5}x_3^{3}x_4^{8}x_5 + x_1^{3}x_2^{5}x_3^{9}x_4x_5^{2} + x_1^{3}x_2^{5}x_3^{9}x_4^{2}x_5 + x_1^{3}x_2^{12}x_3^{3}x_4x_5,\\
g_7 & = x_1^{3}x_2x_3x_4x_5^{14} + x_1^{3}x_2x_3x_4^{12}x_5^{3} + x_1^{3}x_2x_3^{12}x_4x_5^{3} + x_1^{3}x_2^{12}x_3x_4x_5^{3}\\ &\quad + x_1^{7}x_2x_3x_4x_5^{10} + x_1^{7}x_2x_3x_4^{8}x_5^{3} + x_1^{7}x_2x_3^{8}x_4x_5^{3} + x_1^{7}x_2^{8}x_3x_4x_5^{3},\\
g_8 &= x_1^{3}x_2x_3x_4^{3}x_5^{12} + x_1^{3}x_2x_3x_4^{14}x_5 + x_1^{3}x_2x_3^{12}x_4^{3}x_5 + x_1^{3}x_2^{12}x_3x_4^{3}x_5\\ &\quad + x_1^{7}x_2x_3x_4^{3}x_5^{8} + x_1^{7}x_2x_3x_4^{10}x_5 + x_1^{7}x_2x_3^{8}x_4^{3}x_5 + x_1^{7}x_2^{8}x_3x_4^{3}x_5,\\
g_9 &= x_1^{3}x_2x_3^{3}x_4x_5^{12} + x_1^{3}x_2x_3^{3}x_4^{12}x_5 + x_1^{3}x_2x_3^{14}x_4x_5 + x_1^{3}x_2^{12}x_3^{3}x_4x_5\\ &\quad + x_1^{7}x_2x_3^{3}x_4x_5^{8} + x_1^{7}x_2x_3^{3}x_4^{8}x_5 + x_1^{7}x_2x_3^{10}x_4x_5 + x_1^{7}x_2^{8}x_3^{3}x_4x_5,\\
g_{10} &= x_1^{3}x_2^{3}x_3x_4x_5^{12} + x_1^{3}x_2^{3}x_3x_4^{12}x_5 + x_1^{3}x_2^{3}x_3^{12}x_4x_5 + x_1^{3}x_2^{13}x_3x_4x_5^{2}\\ &\quad + x_1^{3}x_2^{13}x_3x_4^{2}x_5 + x_1^{3}x_2^{13}x_3^{2}x_4x_5 + x_1^{7}x_2^{3}x_3x_4x_5^{8} + x_1^{7}x_2^{3}x_3x_4^{8}x_5\\ &\quad + x_1^{7}x_2^{3}x_3^{8}x_4x_5 + x_1^{7}x_2^{9}x_3x_4x_5^{2} + x_1^{7}x_2^{9}x_3x_4^{2}x_5 + x_1^{7}x_2^{9}x_3^{2}x_4x_5,\\
g_{11} &= x_1x_2^{3}x_3^{5}x_4^{5}x_5^{6} + x_1x_2^{3}x_3^{5}x_4^{6}x_5^{5} + x_1x_2^{3}x_3^{6}x_4^{5}x_5^{5} + x_1x_2^{6}x_3^{3}x_4^{5}x_5^{5}\\ &\quad + x_1^{3}x_2x_3^{5}x_4^{5}x_5^{6} + x_1^{3}x_2x_3^{5}x_4^{6}x_5^{5} + x_1^{3}x_2^{5}x_3x_4^{5}x_5^{6} + x_1^{3}x_2^{5}x_3x_4^{6}x_5^{5}\\ &\quad + x_1^{3}x_2^{3}x_3^{4}x_4^{5}x_5^{5} + x_1^{3}x_2^{3}x_3^{5}x_4^{4}x_5^{5} + x_1^{3}x_2^{3}x_3^{5}x_4^{5}x_5^{4} + x_1^{3}x_2^{4}x_3^{3}x_4^{5}x_5^{5}\\ &\quad + x_1^{3}x_2^{5}x_3^{3}x_4^{4}x_5^{5} + x_1^{3}x_2^{5}x_3^{3}x_4^{5}x_5^{4} + x_1^{3}x_2^{5}x_3^{5}x_4^{3}x_5^{4} + x_1^{3}x_2^{5}x_3^{5}x_4^{4}x_5^{3}.
\end{align*}

\subsection{$\Sigma_5$-invariants of degree 20}\label{s52}\

\medskip
$QP_5((4)|(2)|(1)|^2)^{\Sigma_5} = \langle [h_t] : 1 \leqslant t \leqslant 7 \rangle$, where
\begin{align*}
h_1 &= x_2x_3x_4^{15}x_5^{3} + x_2x_3x_4^{3}x_5^{15} + x_2x_3^{15}x_4x_5^{3} + x_2x_3^{15}x_4^{3}x_5 + x_2x_3^{3}x_4x_5^{15}\\ &\quad + x_2x_3^{3}x_4^{15}x_5 + x_2^{15}x_3x_4x_5^{3} + x_2^{15}x_3x_4^{3}x_5 + x_2^{15}x_3^{3}x_4x_5 + x_2^{3}x_3x_4x_5^{15}\\ &\quad + x_2^{3}x_3x_4^{15}x_5 + x_2^{3}x_3^{15}x_4x_5 + x_1x_3x_4^{15}x_5^{3} + x_1x_3x_4^{3}x_5^{15} + x_1x_3^{15}x_4x_5^{3}\\ &\quad + x_1x_3^{15}x_4^{3}x_5 + x_1x_3^{3}x_4x_5^{15} + x_1x_3^{3}x_4^{15}x_5 + x_1x_2x_4^{15}x_5^{3} + x_1x_2x_4^{3}x_5^{15}\\ &\quad + x_1x_2x_3^{15}x_5^{3} + x_1x_2x_3^{15}x_4^{3} + x_1x_2x_3^{3}x_5^{15} + x_1x_2x_3^{3}x_4^{15} + x_1x_2^{15}x_4x_5^{3}\\ &\quad + x_1x_2^{15}x_4^{3}x_5 + x_1x_2^{15}x_3x_5^{3} + x_1x_2^{15}x_3x_4^{3} + x_1x_2^{15}x_3^{3}x_5 + x_1x_2^{15}x_3^{3}x_4\\ &\quad + x_1x_2^{3}x_4x_5^{15} + x_1x_2^{3}x_4^{15}x_5 + x_1x_2^{3}x_3x_5^{15} + x_1x_2^{3}x_3x_4^{15} + x_1x_2^{3}x_3^{15}x_5\\ &\quad + x_1x_2^{3}x_3^{15}x_4 + x_1^{15}x_3x_4x_5^{3} + x_1^{15}x_3x_4^{3}x_5 + x_1^{15}x_3^{3}x_4x_5 + x_1^{15}x_2x_4x_5^{3}\\ &\quad + x_1^{15}x_2x_4^{3}x_5 + x_1^{15}x_2x_3x_5^{3} + x_1^{15}x_2x_3x_4^{3} + x_1^{15}x_2x_3^{3}x_5 + x_1^{15}x_2x_3^{3}x_4\\ &\quad + x_1^{15}x_2^{3}x_4x_5 + x_1^{15}x_2^{3}x_3x_5 + x_1^{15}x_2^{3}x_3x_4 + x_1^{3}x_3x_4x_5^{15} + x_1^{3}x_3x_4^{15}x_5\\ &\quad + x_1^{3}x_3^{15}x_4x_5 + x_1^{3}x_2x_4x_5^{15} + x_1^{3}x_2x_4^{15}x_5 + x_1^{3}x_2x_3x_5^{15} + x_1^{3}x_2x_3x_4^{15}\\ &\quad + x_1^{3}x_2x_3^{15}x_5 + x_1^{3}x_2x_3^{15}x_4 + x_1^{3}x_2^{15}x_4x_5 + x_1^{3}x_2^{15}x_3x_5 + x_1^{3}x_2^{15}x_3x_4,\\
h_2 &= x_2x_3x_4^{7}x_5^{11} + x_2x_3^{7}x_4x_5^{11} + x_2x_3^{7}x_4^{11}x_5 + x_2^{7}x_3x_4x_5^{11} + x_2^{7}x_3x_4^{11}x_5\\ &\quad + x_2^{7}x_3^{11}x_4x_5 + x_1x_3x_4^{7}x_5^{11} + x_1x_3^{7}x_4x_5^{11} + x_1x_3^{7}x_4^{11}x_5 + x_1x_2x_4^{7}x_5^{11}\\ &\quad + x_1x_2x_3^{7}x_5^{11} + x_1x_2x_3^{7}x_4^{11} + x_1x_2^{7}x_4x_5^{11} + x_1x_2^{7}x_4^{11}x_5 + x_1x_2^{7}x_3x_5^{11}\\ &\quad + x_1x_2^{7}x_3^{11}x_5 + x_1x_2^{7}x_3x_4^{11} + x_1x_2^{7}x_3^{11}x_4 + x_1^{7}x_3x_4x_5^{11} + x_1^{7}x_3x_4^{11}x_5\\ &\quad + x_1^{7}x_3^{11}x_4x_5 + x_1^{7}x_2x_4x_5^{11} + x_1^{7}x_2x_4^{11}x_5 + x_1^{7}x_2x_3x_5^{11} + x_1^{7}x_2^{11}x_4x_5\\ &\quad + x_1^{7}x_2x_3^{11}x_5 + x_1^{7}x_2^{11}x_3x_5 + x_1^{7}x_2x_3x_4^{11} + x_1^{7}x_2x_3^{11}x_4 + x_1^{7}x_2^{11}x_3x_4,\\
h_3 &= x_2x_3^{3}x_4^{13}x_5^{3} + x_2x_3^{3}x_4^{3}x_5^{13} + x_2^{3}x_3x_4^{13}x_5^{3} + x_2^{3}x_3x_4^{3}x_5^{13} + x_2^{3}x_3^{13}x_4x_5^{3}\\ &\quad + x_2^{3}x_3^{13}x_4^{3}x_5 + x_2^{3}x_3^{3}x_4x_5^{13} + x_2^{3}x_3^{3}x_4^{13}x_5 + x_1x_3^{3}x_4^{13}x_5^{3} + x_1x_3^{3}x_4^{3}x_5^{13}\\ &\quad + x_1x_2^{3}x_4^{13}x_5^{3} + x_1x_2^{3}x_4^{3}x_5^{13} + x_1x_2^{3}x_3^{13}x_5^{3} + x_1x_2^{3}x_3^{13}x_4^{3} + x_1x_2^{3}x_3^{3}x_5^{13}\\ &\quad + x_1x_2^{3}x_3^{3}x_4^{13} + x_1^{3}x_3x_4^{13}x_5^{3} + x_1^{3}x_3x_4^{3}x_5^{13} + x_1^{3}x_3^{13}x_4x_5^{3} + x_1^{3}x_3^{13}x_4^{3}x_5\\ &\quad + x_1^{3}x_3^{3}x_4x_5^{13} + x_1^{3}x_3^{3}x_4^{13}x_5 + x_1^{3}x_2x_4^{13}x_5^{3} + x_1^{3}x_2x_4^{3}x_5^{13} + x_1^{3}x_2x_3^{13}x_5^{3}\\ &\quad + x_1^{3}x_2x_3^{13}x_4^{3} + x_1^{3}x_2x_3^{3}x_5^{13} + x_1^{3}x_2^{13}x_4x_5^{3} + x_1^{3}x_2^{13}x_4^{3}x_5 + x_1^{3}x_2^{13}x_3x_5^{3}\\ &\quad + x_1^{3}x_2x_3^{3}x_4^{13} + x_1^{3}x_2^{13}x_3x_4^{3} + x_1^{3}x_2^{13}x_3^{3}x_5 + x_1^{3}x_2^{13}x_3^{3}x_4 + x_1^{3}x_2^{3}x_4x_5^{13}\\ &\quad + x_1^{3}x_2^{3}x_4^{13}x_5 + x_1^{3}x_2^{3}x_3x_5^{13} + x_1^{3}x_2^{3}x_3x_4^{13} + x_1^{3}x_2^{3}x_3^{13}x_5 + x_1^{3}x_2^{3}x_3^{13}x_4\\ &\quad + x_2x_3^{7}x_4^{3}x_5^{9} + x_2x_3^{7}x_4^{9}x_5^{3} + x_2^{7}x_3x_4^{3}x_5^{9} + x_2^{7}x_3x_4^{9}x_5^{3} + x_2^{7}x_3^{3}x_4x_5^{9}\\ &\quad + x_2^{7}x_3^{3}x_4^{9}x_5 + x_2^{7}x_3^{9}x_4x_5^{3} + x_2^{7}x_3^{9}x_4^{3}x_5 + x_1x_3^{7}x_4^{3}x_5^{9} + x_1x_3^{7}x_4^{9}x_5^{3}\\ &\quad + x_1x_2^{7}x_4^{3}x_5^{9} + x_1x_2^{7}x_4^{9}x_5^{3} + x_1x_2^{7}x_3^{3}x_5^{9} + x_1x_2^{7}x_3^{3}x_4^{9} + x_1x_2^{7}x_3^{9}x_5^{3}\\ &\quad + x_1x_2^{7}x_3^{9}x_4^{3} + x_1^{7}x_3x_4^{3}x_5^{9} + x_1^{7}x_3x_4^{9}x_5^{3} + x_1^{7}x_3^{3}x_4x_5^{9} + x_1^{7}x_3^{3}x_4^{9}x_5\\ &\quad + x_1^{7}x_3^{9}x_4x_5^{3} + x_1^{7}x_3^{9}x_4^{3}x_5 + x_1^{7}x_2x_4^{3}x_5^{9} + x_1^{7}x_2x_4^{9}x_5^{3} + x_1^{7}x_2x_3^{3}x_5^{9}\\ &\quad + x_1^{7}x_2x_3^{3}x_4^{9} + x_1^{7}x_2x_3^{9}x_5^{3} + x_1^{7}x_2x_3^{9}x_4^{3} + x_1^{7}x_2^{3}x_4x_5^{9} + x_1^{7}x_2^{3}x_4^{9}x_5\\ &\quad + x_1^{7}x_2^{3}x_3x_5^{9} + x_1^{7}x_2^{3}x_3x_4^{9} + x_1^{7}x_2^{3}x_3^{9}x_5 + x_1^{7}x_2^{3}x_3^{9}x_4 + x_1^{7}x_2^{9}x_4x_5^{3}\\ &\quad + x_1^{7}x_2^{9}x_4^{3}x_5 + x_1^{7}x_2^{9}x_3x_5^{3} + x_1^{7}x_2^{9}x_3x_4^{3} + x_1^{7}x_2^{9}x_3^{3}x_5 + x_1^{7}x_2^{9}x_3^{3}x_4,\\
h_4 &= x_2^{3}x_3^{3}x_4^{5}x_5^{9} + x_2^{3}x_3^{5}x_4^{3}x_5^{9} + x_2^{3}x_3^{5}x_4^{9}x_5^{3} + x_1^{3}x_3^{3}x_4^{5}x_5^{9} + x_1^{3}x_3^{5}x_4^{3}x_5^{9}\\ &\quad + x_1^{3}x_3^{5}x_4^{9}x_5^{3} + x_1^{3}x_2^{3}x_4^{5}x_5^{9} + x_1^{3}x_2^{3}x_3^{5}x_5^{9} + x_1^{3}x_2^{3}x_3^{5}x_4^{9} + x_1^{3}x_2^{5}x_4^{3}x_5^{9}\\ &\quad + x_1^{3}x_2^{5}x_4^{9}x_5^{3} + x_1^{3}x_2^{5}x_3^{3}x_5^{9} + x_1^{3}x_2^{5}x_3^{3}x_4^{9} + x_1^{3}x_2^{5}x_3^{9}x_5^{3} + x_1^{3}x_2^{5}x_3^{9}x_4^{3},\\
h_5 &= x_1x_2x_3x_4^{3}x_5^{14} + x_1x_2x_3x_4^{6}x_5^{11} + x_1x_2x_3x_4^{7}x_5^{10} + x_1x_2x_3x_4^{14}x_5^{3}\\ &\quad + x_1x_2x_3^{3}x_4x_5^{14} + x_1x_2x_3^{3}x_4^{14}x_5 + x_1x_2x_3^{6}x_4x_5^{11} + x_1x_2x_3^{6}x_4^{11}x_5\\ &\quad + x_1x_2x_3^{7}x_4x_5^{10} + x_1x_2x_3^{7}x_4^{10}x_5 + x_1x_2x_3^{14}x_4x_5^{3} + x_1x_2x_3^{14}x_4^{3}x_5\\ &\quad + x_1x_2^{3}x_3x_4x_5^{14} + x_1x_2^{3}x_3x_4^{14}x_5 + x_1x_2^{3}x_3^{14}x_4x_5 + x_1x_2^{6}x_3x_4x_5^{11}\\ &\quad + x_1x_2^{6}x_3x_4^{11}x_5 + x_1x_2^{6}x_3^{11}x_4x_5 + x_1x_2^{7}x_3x_4x_5^{10} + x_1x_2^{7}x_3x_4^{10}x_5\\ &\quad + x_1x_2^{7}x_3^{10}x_4x_5 + x_1x_2^{14}x_3x_4x_5^{3} + x_1x_2^{14}x_3x_4^{3}x_5 + x_1x_2^{14}x_3^{3}x_4x_5\\ &\quad + x_1^{3}x_2x_3x_4^{4}x_5^{11} + x_1^{3}x_2x_3x_4^{7}x_5^{8} + x_1^{3}x_2x_3^{4}x_4x_5^{11} + x_1^{3}x_2x_3^{4}x_4^{11}x_5\\ &\quad + x_1^{3}x_2x_3^{7}x_4x_5^{8} + x_1^{3}x_2x_3^{7}x_4^{8}x_5 + x_1^{3}x_2^{4}x_3x_4x_5^{11} + x_1^{3}x_2^{4}x_3x_4^{11}x_5\\ &\quad + x_1^{3}x_2^{4}x_3^{11}x_4x_5 + x_1^{3}x_2^{7}x_3x_4x_5^{8} + x_1^{3}x_2^{7}x_3x_4^{8}x_5 + x_1^{3}x_2^{7}x_3^{8}x_4x_5\\ &\quad + x_1^{7}x_2x_3x_4^{3}x_5^{8} + x_1^{7}x_2x_3x_4^{8}x_5^{3} + x_1^{7}x_2x_3^{3}x_4x_5^{8} + x_1^{7}x_2x_3^{3}x_4^{8}x_5\\ &\quad + x_1^{7}x_2x_3^{8}x_4x_5^{3} + x_1^{7}x_2x_3^{8}x_4^{3}x_5 + x_1^{7}x_2^{3}x_3x_4x_5^{8} + x_1^{7}x_2^{3}x_3x_4^{8}x_5\\ &\quad + x_1^{7}x_2^{3}x_3^{8}x_4x_5 + x_1^{7}x_2^{8}x_3x_4x_5^{3} + x_1^{7}x_2^{8}x_3x_4^{3}x_5 + x_1^{7}x_2^{8}x_3^{3}x_4x_5,\\
h_6 &= x_1x_2x_3x_4^{6}x_5^{11} + x_1x_2x_3x_4^{7}x_5^{10} + x_1x_2x_3^{6}x_4x_5^{11} + x_1x_2x_3^{6}x_4^{11}x_5\\ &\quad + x_1x_2x_3^{7}x_4x_5^{10} + x_1x_2x_3^{7}x_4^{10}x_5 + x_1x_2^{6}x_3x_4x_5^{11} + x_1x_2^{6}x_3x_4^{11}x_5\\ &\quad + x_1x_2^{6}x_3^{11}x_4x_5 + x_1x_2^{7}x_3x_4x_5^{10} + x_1x_2^{7}x_3x_4^{10}x_5 + x_1x_2^{7}x_3^{10}x_4x_5\\ &\quad + x_1^{3}x_2x_3x_4x_5^{14} + x_1^{3}x_2x_3x_4^{4}x_5^{11} + x_1^{3}x_2x_3x_4^{7}x_5^{8} + x_1^{3}x_2x_3x_4^{14}x_5\\ &\quad + x_1^{3}x_2x_3^{4}x_4x_5^{11} + x_1^{3}x_2x_3^{4}x_4^{11}x_5 + x_1^{3}x_2x_3^{7}x_4x_5^{8} + x_1^{3}x_2x_3^{7}x_4^{8}x_5\\ &\quad + x_1^{3}x_2x_3^{14}x_4x_5 + x_1^{3}x_2^{4}x_3x_4x_5^{11} + x_1^{3}x_2^{4}x_3x_4^{11}x_5 + x_1^{3}x_2^{4}x_3^{11}x_4x_5\\ &\quad + x_1^{3}x_2^{7}x_3x_4x_5^{8} + x_1^{3}x_2^{7}x_3x_4^{8}x_5 + x_1^{3}x_2^{7}x_3^{8}x_4x_5 + x_1^{3}x_2^{13}x_3x_4x_5^{2}\\ &\quad + x_1^{3}x_2^{13}x_3x_4^{2}x_5 + x_1^{3}x_2^{13}x_3^{2}x_4x_5 + x_1^{7}x_2x_3x_4x_5^{10} + x_1^{7}x_2x_3x_4^{10}x_5\\ &\quad + x_1^{7}x_2x_3^{10}x_4x_5 + x_1^{7}x_2^{9}x_3x_4x_5^{2} + x_1^{7}x_2^{9}x_3x_4^{2}x_5 + x_1^{7}x_2^{9}x_3^{2}x_4x_5,\\
h_7 &= x_1x_2x_3^{3}x_4^{3}x_5^{12} + x_1x_2x_3^{3}x_4^{12}x_5^{3} + x_1x_2x_3^{6}x_4^{3}x_5^{9} + x_1x_2x_3^{6}x_4^{9}x_5^{3}\\ &\quad + x_1x_2^{3}x_3x_4^{3}x_5^{12} + x_1x_2^{3}x_3x_4^{12}x_5^{3} + x_1x_2^{3}x_3^{3}x_4x_5^{12} + x_1x_2^{3}x_3^{3}x_4^{12}x_5\\ &\quad + x_1x_2^{3}x_3^{12}x_4x_5^{3} + x_1x_2^{3}x_3^{12}x_4^{3}x_5 + x_1x_2^{6}x_3x_4^{3}x_5^{9} + x_1x_2^{6}x_3x_4^{9}x_5^{3}\\ &\quad + x_1x_2^{6}x_3^{3}x_4x_5^{9} + x_1x_2^{6}x_3^{3}x_4^{9}x_5 + x_1x_2^{6}x_3^{9}x_4x_5^{3} + x_1x_2^{6}x_3^{9}x_4^{3}x_5\\ &\quad + x_1^{3}x_2x_3x_4^{5}x_5^{10} + x_1^{3}x_2x_3x_4^{6}x_5^{9} + x_1^{3}x_2x_3^{3}x_4^{4}x_5^{9} + x_1^{3}x_2x_3^{3}x_4^{5}x_5^{8}\\ &\quad + x_1^{3}x_2x_3^{4}x_4^{3}x_5^{9} + x_1^{3}x_2x_3^{4}x_4^{9}x_5^{3} + x_1^{3}x_2x_3^{5}x_4x_5^{10} + x_1^{3}x_2x_3^{5}x_4^{3}x_5^{8}\\ &\quad + x_1^{3}x_2x_3^{5}x_4^{8}x_5^{3} + x_1^{3}x_2x_3^{5}x_4^{10}x_5 + x_1^{3}x_2x_3^{6}x_4x_5^{9} + x_1^{3}x_2x_3^{6}x_4^{9}x_5\\ &\quad + x_1^{3}x_2^{3}x_3x_4^{4}x_5^{9} + x_1^{3}x_2^{3}x_3x_4^{5}x_5^{8} + x_1^{3}x_2^{3}x_3^{4}x_4x_5^{9} + x_1^{3}x_2^{3}x_3^{4}x_4^{9}x_5\\ &\quad + x_1^{3}x_2^{3}x_3^{5}x_4x_5^{8} + x_1^{3}x_2^{3}x_3^{5}x_4^{8}x_5 + x_1^{3}x_2^{4}x_3x_4^{3}x_5^{9} + x_1^{3}x_2^{4}x_3x_4^{9}x_5^{3}\\ &\quad + x_1^{3}x_2^{4}x_3^{3}x_4x_5^{9} + x_1^{3}x_2^{4}x_3^{3}x_4^{9}x_5 + x_1^{3}x_2^{4}x_3^{9}x_4x_5^{3} + x_1^{3}x_2^{4}x_3^{9}x_4^{3}x_5\\ &\quad + x_1^{3}x_2^{5}x_3x_4^{2}x_5^{9} + x_1^{3}x_2^{5}x_3x_4^{3}x_5^{8} + x_1^{3}x_2^{5}x_3x_4^{8}x_5^{3} + x_1^{3}x_2^{5}x_3x_4^{9}x_5^{2}\\ &\quad + x_1^{3}x_2^{5}x_3^{2}x_4x_5^{9} + x_1^{3}x_2^{5}x_3^{2}x_4^{9}x_5 + x_1^{3}x_2^{5}x_3^{3}x_4x_5^{8} + x_1^{3}x_2^{5}x_3^{3}x_4^{8}x_5\\ &\quad + x_1^{3}x_2^{5}x_3^{8}x_4x_5^{3} + x_1^{3}x_2^{5}x_3^{8}x_4^{3}x_5 + x_1^{3}x_2^{5}x_3^{9}x_4x_5^{2} + x_1^{3}x_2^{5}x_3^{9}x_4^{2}x_5.
\end{align*}

$QP_5((4)|(2)|(3))^{\Sigma_5} = \langle [h_t]_{\overline{\omega}} : 8 \leqslant t \leqslant 10 \rangle$, where $\overline{\omega} = (4)|(2)|(3)$ and
\begin{align*}
h_8 &= x_2^{3}x_3^{5}x_4^{5}x_5^{7} + x_2^{3}x_3^{5}x_4^{7}x_5^{5} + x_2^{3}x_3^{7}x_4^{5}x_5^{5} + x_2^{7}x_3^{3}x_4^{5}x_5^{5} + x_1^{3}x_3^{5}x_4^{5}x_5^{7}\\ &\quad + x_1^{3}x_3^{5}x_4^{7}x_5^{5} + x_1^{3}x_3^{7}x_4^{5}x_5^{5} + x_1^{3}x_2^{5}x_4^{5}x_5^{7} + x_1^{3}x_2^{5}x_4^{7}x_5^{5} + x_1^{3}x_2^{5}x_3^{5}x_5^{7}\\ &\quad + x_1^{3}x_2^{5}x_3^{5}x_4^{7} + x_1^{3}x_2^{5}x_3^{7}x_5^{5} + x_1^{3}x_2^{5}x_3^{7}x_4^{5} + x_1^{3}x_2^{7}x_4^{5}x_5^{5} + x_1^{3}x_2^{7}x_3^{5}x_5^{5}\\ &\quad + x_1^{3}x_2^{7}x_3^{5}x_4^{5} + x_1^{7}x_3^{3}x_4^{5}x_5^{5} + x_1^{7}x_2^{3}x_4^{5}x_5^{5} + x_1^{7}x_2^{3}x_3^{5}x_5^{5} + x_1^{7}x_2^{3}x_3^{5}x_4^{5},\\
h_9 &= x_1x_2^{2}x_3^{5}x_4^{5}x_5^{7} + x_1x_2^{2}x_3^{5}x_4^{7}x_5^{5} + x_1x_2^{2}x_3^{7}x_4^{5}x_5^{5} + x_1x_2^{7}x_3^{2}x_4^{5}x_5^{5}\\ &\quad + x_1^{7}x_2x_3^{2}x_4^{5}x_5^{5} + x_1x_2^{3}x_3^{4}x_4^{5}x_5^{7} + x_1x_2^{3}x_3^{4}x_4^{7}x_5^{5} + x_1x_2^{3}x_3^{5}x_4^{4}x_5^{7}\\ &\quad + x_1x_2^{3}x_3^{5}x_4^{7}x_5^{4} + x_1x_2^{3}x_3^{7}x_4^{4}x_5^{5} + x_1x_2^{3}x_3^{7}x_4^{5}x_5^{4} + x_1x_2^{7}x_3^{3}x_4^{4}x_5^{5}\\ &\quad + x_1x_2^{7}x_3^{3}x_4^{5}x_5^{4} + x_1^{3}x_2x_3^{4}x_4^{5}x_5^{7} + x_1^{3}x_2x_3^{4}x_4^{7}x_5^{5} + x_1^{3}x_2x_3^{5}x_4^{4}x_5^{7}\\ &\quad + x_1^{3}x_2x_3^{5}x_4^{7}x_5^{4} + x_1^{3}x_2x_3^{7}x_4^{4}x_5^{5} + x_1^{3}x_2x_3^{7}x_4^{5}x_5^{4} + x_1^{3}x_2^{5}x_3x_4^{4}x_5^{7}\\ &\quad + x_1^{3}x_2^{5}x_3x_4^{7}x_5^{4} + x_1^{3}x_2^{5}x_3^{7}x_4x_5^{4} + x_1^{3}x_2^{7}x_3x_4^{4}x_5^{5} + x_1^{3}x_2^{7}x_3x_4^{5}x_5^{4}\\ &\quad + x_1^{3}x_2^{7}x_3^{5}x_4x_5^{4} + x_1^{7}x_2x_3^{3}x_4^{4}x_5^{5} + x_1^{7}x_2x_3^{3}x_4^{5}x_5^{4} + x_1^{7}x_2^{3}x_3x_4^{4}x_5^{5}\\ &\quad + x_1^{7}x_2^{3}x_3x_4^{5}x_5^{4} + x_1^{7}x_2^{3}x_3^{5}x_4x_5^{4},\\
h_{10} &= x_1x_2^{3}x_3^{5}x_4^{5}x_5^{6} + x_1x_2^{3}x_3^{5}x_4^{6}x_5^{5} + x_1x_2^{3}x_3^{6}x_4^{5}x_5^{5} + x_1x_2^{6}x_3^{3}x_4^{5}x_5^{5}\\ &\quad + x_1^{3}x_2x_3^{5}x_4^{5}x_5^{6} + x_1^{3}x_2x_3^{5}x_4^{6}x_5^{5} + x_1^{3}x_2^{5}x_3x_4^{5}x_5^{6} + x_1^{3}x_2^{5}x_3x_4^{6}x_5^{5}\\ &\quad + x_1^{3}x_2^{3}x_3^{4}x_4^{5}x_5^{5} + x_1^{3}x_2^{3}x_3^{5}x_4^{4}x_5^{5} + x_1^{3}x_2^{3}x_3^{5}x_4^{5}x_5^{4} + x_1^{3}x_2^{4}x_3^{3}x_4^{5}x_5^{5}\\ &\quad + x_1^{3}x_2^{5}x_3^{3}x_4^{4}x_5^{5} + x_1^{3}x_2^{5}x_3^{3}x_4^{5}x_5^{4} + x_1^{3}x_2^{5}x_3^{5}x_4^{3}x_5^{4} + x_1^{3}x_2^{5}x_3^{5}x_4^{4}x_5^{3}.
\end{align*}

$QP_5((4)|^2|(2))^{\Sigma_5} = \langle [h_t]_{\omega^*} : 11 \leqslant t \leqslant 13 \rangle$, where $\omega^* = (4)|^2|(2)$ and 
\begin{align*}
h_{11} &= x_2^{3}x_3^{3}x_4^{7}x_5^{7} + x_2^{3}x_3^{7}x_4^{3}x_5^{7} + x_2^{3}x_3^{7}x_4^{7}x_5^{3} + x_2^{7}x_3^{3}x_4^{3}x_5^{7} + x_2^{7}x_3^{3}x_4^{7}x_5^{3}\\ &\quad + x_2^{7}x_3^{7}x_4^{3}x_5^{3} + x_1^{3}x_3^{3}x_4^{7}x_5^{7} + x_1^{3}x_3^{7}x_4^{3}x_5^{7} + x_1^{3}x_3^{7}x_4^{7}x_5^{3} + x_1^{3}x_2^{3}x_4^{7}x_5^{7}\\ &\quad + x_1^{3}x_2^{3}x_3^{7}x_5^{7} + x_1^{3}x_2^{3}x_3^{7}x_4^{7} + x_1^{3}x_2^{7}x_4^{3}x_5^{7} + x_1^{3}x_2^{7}x_4^{7}x_5^{3} + x_1^{3}x_2^{7}x_3^{3}x_5^{7}\\ &\quad + x_1^{3}x_2^{7}x_3^{3}x_4^{7} + x_1^{3}x_2^{7}x_3^{7}x_5^{3} + x_1^{3}x_2^{7}x_3^{7}x_4^{3} + x_1^{7}x_3^{3}x_4^{3}x_5^{7} + x_1^{7}x_3^{3}x_4^{7}x_5^{3}\\ &\quad + x_1^{7}x_3^{7}x_4^{3}x_5^{3} + x_1^{7}x_2^{3}x_4^{3}x_5^{7} + x_1^{7}x_2^{3}x_4^{7}x_5^{3} + x_1^{7}x_2^{3}x_3^{3}x_5^{7} + x_1^{7}x_2^{3}x_3^{3}x_4^{7}\\ &\quad + x_1^{7}x_2^{3}x_3^{7}x_5^{3} + x_1^{7}x_2^{3}x_3^{7}x_4^{3} + x_1^{7}x_2^{7}x_4^{3}x_5^{3} + x_1^{7}x_2^{7}x_3^{3}x_5^{3} + x_1^{7}x_2^{7}x_3^{3}x_4^{3},\\
h_{12} &= x_1x_2^{2}x_3^{3}x_4^{7}x_5^{7} + x_1x_2^{2}x_3^{7}x_4^{3}x_5^{7} + x_1x_2^{2}x_3^{7}x_4^{7}x_5^{3} + x_1x_2^{3}x_3^{2}x_4^{7}x_5^{7}\\ &\quad  + x_1x_2^{3}x_3^{7}x_4^{2}x_5^{7} + x_1x_2^{3}x_3^{7}x_4^{7}x_5^{2} + x_1x_2^{7}x_3^{2}x_4^{3}x_5^{7} + x_1x_2^{7}x_3^{2}x_4^{7}x_5^{3}\\ &\quad  + x_1x_2^{7}x_3^{3}x_4^{2}x_5^{7} + x_1x_2^{7}x_3^{3}x_4^{7}x_5^{2} + x_1x_2^{7}x_3^{7}x_4^{2}x_5^{3} + x_1x_2^{7}x_3^{7}x_4^{3}x_5^{2}\\ &\quad  + x_1^{3}x_2x_3^{2}x_4^{7}x_5^{7} + x_1^{3}x_2x_3^{7}x_4^{2}x_5^{7} + x_1^{3}x_2x_3^{7}x_4^{7}x_5^{2} + x_1^{3}x_2^{7}x_3x_4^{2}x_5^{7}\\ &\quad  + x_1^{3}x_2^{7}x_3x_4^{7}x_5^{2} + x_1^{3}x_2^{7}x_3^{7}x_4x_5^{2} + x_1^{7}x_2x_3^{2}x_4^{3}x_5^{7} + x_1^{7}x_2x_3^{2}x_4^{7}x_5^{3}\\ &\quad  + x_1^{7}x_2x_3^{3}x_4^{2}x_5^{7} + x_1^{7}x_2x_3^{3}x_4^{7}x_5^{2} + x_1^{7}x_2x_3^{7}x_4^{2}x_5^{3} + x_1^{7}x_2x_3^{7}x_4^{3}x_5^{2}\\ &\quad  + x_1^{7}x_2^{3}x_3x_4^{2}x_5^{7} + x_1^{7}x_2^{3}x_3x_4^{7}x_5^{2} + x_1^{7}x_2^{3}x_3^{7}x_4x_5^{2} + x_1^{7}x_2^{7}x_3x_4^{2}x_5^{3}\\ &\quad  + x_1^{7}x_2^{7}x_3x_4^{3}x_5^{2} + x_1^{7}x_2^{7}x_3^{3}x_4x_5^{2},\\
h_{13} &= x_1^{3}x_2^{3}x_3^{3}x_4^{5}x_5^{6} + x_1^{3}x_2^{3}x_3^{5}x_4^{3}x_5^{6} + x_1^{3}x_2^{3}x_3^{5}x_4^{6}x_5^{3} + x_1^{3}x_2^{5}x_3^{3}x_4^{3}x_5^{6}\\ &\quad + x_1^{3}x_2^{5}x_3^{3}x_4^{6}x_5^{3} + x_1^{3}x_2^{5}x_3^{6}x_4^{3}x_5^{3}.
\end{align*}

\subsection{$\Sigma_5$-invariants of $QP_5((2)|^4)$}\label{s53}\

\medskip
$QP_5((4)|(2)|(1)|^2)^{\Sigma_5} = \langle [p_t] : 1 \leqslant t \leqslant 9 \rangle$, where
\begin{align*}
p_1 &= x_4^{15}x_5^{15} + x_3^{15}x_5^{15} + x_3^{15}x_4^{15} + x_2^{15}x_5^{15} + x_2^{15}x_4^{15}\\ 
&\quad + x_2^{15}x_3^{15} + x_1^{15}x_5^{15} + x_1^{15}x_4^{15} + x_1^{15}x_3^{15} + x_1^{15}x_2^{15},\\
p_2 &= x_3x_4^{14}x_5^{15} + x_3x_4^{15}x_5^{14} + x_3^{15}x_4x_5^{14} + x_2x_4^{14}x_5^{15} + x_2x_4^{15}x_5^{14} + x_2x_3^{14}x_5^{15}\\ &\quad + x_2x_3^{14}x_4^{15} + x_2x_3^{15}x_5^{14} + x_2x_3^{15}x_4^{14} + x_2^{15}x_4x_5^{14} + x_2^{15}x_3x_5^{14} + x_2^{15}x_3x_4^{14}\\ &\quad + x_1x_4^{14}x_5^{15} + x_1x_4^{15}x_5^{14} + x_1x_3^{14}x_5^{15} + x_1x_3^{14}x_4^{15} + x_1x_3^{15}x_5^{14} + x_1x_3^{15}x_4^{14}\\ &\quad + x_1x_2^{14}x_5^{15} + x_1x_2^{14}x_4^{15} + x_1x_2^{14}x_3^{15} + x_1x_2^{15}x_5^{14} + x_1x_2^{15}x_4^{14} + x_1x_2^{15}x_3^{14}\\ &\quad + x_1^{15}x_4x_5^{14} + x_1^{15}x_3x_5^{14} + x_1^{15}x_3x_4^{14} + x_1^{15}x_2x_5^{14} + x_1^{15}x_2x_4^{14} + x_1^{15}x_2x_3^{14},\\
p_3 &= x_3^{3}x_4^{13}x_5^{14} + x_2^{3}x_4^{13}x_5^{14} + x_2^{3}x_3^{13}x_5^{14} + x_2^{3}x_3^{13}x_4^{14} + x_1^{3}x_4^{13}x_5^{14}\\ &\quad + x_1^{3}x_3^{13}x_5^{14} + x_1^{3}x_3^{13}x_4^{14} + x_1^{3}x_2^{13}x_5^{14} + x_1^{3}x_2^{13}x_4^{14} + x_1^{3}x_2^{13}x_3^{14},\\
p_4 &= x_2x_3x_4^{14}x_5^{14} + x_2x_3^{14}x_4x_5^{14} + x_2^{3}x_3^{5}x_4^{10}x_5^{12} + x_1x_3x_4^{14}x_5^{14} + x_1x_3^{14}x_4x_5^{14}\\ &\quad + x_1x_2x_4^{14}x_5^{14} + x_1x_2x_3^{14}x_5^{14} + x_1x_2x_3^{14}x_4^{14} + x_1x_2^{14}x_4x_5^{14} + x_1x_2^{14}x_3x_5^{14}\\ &\quad + x_1x_2^{14}x_3x_4^{14} + x_1^{3}x_3^{5}x_4^{10}x_5^{12} + x_1^{3}x_2^{5}x_4^{10}x_5^{12} + x_1^{3}x_2^{5}x_3^{10}x_5^{12} + x_1^{3}x_2^{5}x_3^{10}x_4^{12},\\
p_5 & = x_2x_3^{2}x_4^{13}x_5^{14} + x_2x_3^{3}x_4^{12}x_5^{14} + x_2x_3^{3}x_4^{14}x_5^{12} + x_2^{3}x_3x_4^{12}x_5^{14} + x_2^{3}x_3x_4^{14}x_5^{12}\\ &\quad + x_2^{3}x_3^{13}x_4^{2}x_5^{12} + x_1x_3^{2}x_4^{13}x_5^{14} + x_1x_3^{3}x_4^{12}x_5^{14} + x_1x_3^{3}x_4^{14}x_5^{12} + x_1x_2^{2}x_4^{13}x_5^{14}\\ &\quad + x_1x_2^{2}x_3^{13}x_5^{14} + x_1x_2^{2}x_3^{13}x_4^{14} + x_1x_2^{3}x_4^{12}x_5^{14} + x_1x_2^{3}x_4^{14}x_5^{12} + x_1x_2^{3}x_3^{12}x_5^{14}\\ &\quad + x_1x_2^{3}x_3^{12}x_4^{14} + x_1x_2^{3}x_3^{14}x_5^{12} + x_1x_2^{3}x_3^{14}x_4^{12} + x_1^{3}x_3x_4^{12}x_5^{14} + x_1^{3}x_3x_4^{14}x_5^{12}\\ &\quad + x_1^{3}x_3^{13}x_4^{2}x_5^{12} + x_1^{3}x_2x_4^{12}x_5^{14} + x_1^{3}x_2x_4^{14}x_5^{12} + x_1^{3}x_2x_3^{12}x_5^{14} + x_1^{3}x_2x_3^{12}x_4^{14}\\ &\quad + x_1^{3}x_2x_3^{14}x_5^{12} + x_1^{3}x_2x_3^{14}x_4^{12} + x_1^{3}x_2^{13}x_4^{2}x_5^{12} + x_1^{3}x_2^{13}x_3^{2}x_5^{12} + x_1^{3}x_2^{13}x_3^{2}x_4^{12},\\
p_6 &= x_2x_3^{2}x_4^{12}x_5^{15} + x_2x_3^{2}x_4^{15}x_5^{12} + x_2x_3^{15}x_4^{2}x_5^{12} + x_2^{15}x_3x_4^{2}x_5^{12} + x_1x_3^{2}x_4^{12}x_5^{15}\\ &\quad + x_1x_3^{2}x_4^{15}x_5^{12} + x_1x_3^{15}x_4^{2}x_5^{12} + x_1x_2^{2}x_4^{12}x_5^{15} + x_1x_2^{2}x_4^{15}x_5^{12} + x_1x_2^{2}x_3^{12}x_5^{15}\\ &\quad + x_1x_2^{2}x_3^{12}x_4^{15} + x_1x_2^{2}x_3^{15}x_5^{12} + x_1x_2^{2}x_3^{15}x_4^{12} + x_1x_2^{15}x_4^{2}x_5^{12} + x_1x_2^{15}x_3^{2}x_5^{12}\\ &\quad + x_1x_2^{15}x_3^{2}x_4^{12} + x_1^{15}x_3x_4^{2}x_5^{12} + x_1^{15}x_2x_4^{2}x_5^{12} + x_1^{15}x_2x_3^{2}x_5^{12} + x_1^{15}x_2x_3^{2}x_4^{12},\\
p_7 &= x_1x_2^{2}x_3^{4}x_4^{8}x_5^{15} + x_1x_2^{2}x_3^{4}x_4^{15}x_5^{8} + x_1x_2^{2}x_3^{15}x_4^{4}x_5^{8} + x_1x_2^{15}x_3^{2}x_4^{4}x_5^{8}\\ &\quad + x_1^{15}x_2x_3^{2}x_4^{4}x_5^{8},\\
p_8 &= x_1x_2x_3^{2}x_4^{14}x_5^{12} + x_1x_2x_3^{6}x_4^{10}x_5^{12} + x_1x_2x_3^{14}x_4^{2}x_5^{12} + x_1x_2^{2}x_3x_4^{12}x_5^{14}\\ &\quad + x_1x_2^{2}x_3^{5}x_4^{10}x_5^{12} + x_1x_2^{2}x_3^{12}x_4x_5^{14} + x_1x_2^{3}x_3^{2}x_4^{12}x_5^{12} + x_1x_2^{3}x_3^{4}x_4^{10}x_5^{12}\\ &\quad + x_1x_2^{3}x_3^{6}x_4^{8}x_5^{12} + x_1x_2^{3}x_3^{6}x_4^{12}x_5^{8} + x_1x_2^{14}x_3x_4^{2}x_5^{12} + x_1^{3}x_2x_3^{4}x_4^{10}x_5^{12}\\ &\quad + x_1^{3}x_2x_3^{6}x_4^{8}x_5^{12} + x_1^{3}x_2x_3^{6}x_4^{12}x_5^{8} + x_1^{3}x_2^{5}x_3^{2}x_4^{8}x_5^{12} + x_1^{3}x_2^{5}x_3^{2}x_4^{12}x_5^{8}\\ &\quad + x_1^{3}x_2^{5}x_3^{10}x_4^{4}x_5^{8},\\
p_9 &= x_1x_2^{2}x_3x_4^{12}x_5^{14} + x_1x_2^{2}x_3x_4^{14}x_5^{12} + x_1x_2^{2}x_3^{4}x_4^{9}x_5^{14} + x_1x_2^{2}x_3^{5}x_4^{8}x_5^{14}\\ &\quad + x_1x_2^{2}x_3^{5}x_4^{14}x_5^{8} + x_1x_2^{2}x_3^{12}x_4x_5^{14} + x_1x_2^{3}x_3^{4}x_4^{8}x_5^{14} + x_1x_2^{3}x_3^{4}x_4^{14}x_5^{8}\\ &\quad + x_1x_2^{3}x_3^{14}x_4^{4}x_5^{8} + x_1x_2^{14}x_3x_4^{2}x_5^{12} + x_1^{3}x_2x_3^{4}x_4^{8}x_5^{14} + x_1^{3}x_2x_3^{4}x_4^{14}x_5^{8}\\ &\quad + x_1^{3}x_2x_3^{14}x_4^{4}x_5^{8} + x_1^{3}x_2^{13}x_3^{2}x_4^{4}x_5^{8},
\end{align*}

\subsection{$\Sigma_5$-invariants of $QP_5((4)|(3)|^2|(1))$}\label{s54}\

\medskip
$QP_5((4)|((4)|(3)|^2|(1))^{\Sigma_5} = \langle [q_t]_{(4)|(3)|^2|(1)} : 1 \leqslant t \leqslant 7 \rangle$, where
\begin{align*}
q_1 &= x_2x_3^{7}x_4^{7}x_5^{15} + x_2x_3^{7}x_4^{15}x_5^{7} + x_2x_3^{15}x_4^{7}x_5^{7} + x_2^{7}x_3x_4^{7}x_5^{15} + x_2^{7}x_3x_4^{15}x_5^{7}\\ 
&\quad + x_2^{7}x_3^{7}x_4x_5^{15} + x_2^{7}x_3^{7}x_4^{15}x_5 + x_2^{7}x_3^{15}x_4x_5^{7} + x_2^{7}x_3^{15}x_4^{7}x_5 + x_2^{15}x_3x_4^{7}x_5^{7}\\ 
&\quad + x_2^{15}x_3^{7}x_4x_5^{7} + x_2^{15}x_3^{7}x_4^{7}x_5 + x_1x_3^{7}x_4^{7}x_5^{15} + x_1x_3^{7}x_4^{15}x_5^{7} + x_1x_3^{15}x_4^{7}x_5^{7}\\ 
&\quad + x_1x_2^{7}x_4^{7}x_5^{15} + x_1x_2^{7}x_4^{15}x_5^{7} + x_1x_2^{7}x_3^{7}x_5^{15} + x_1x_2^{7}x_3^{7}x_4^{15} + x_1x_2^{7}x_3^{15}x_5^{7}\\ 
&\quad + x_1x_2^{7}x_3^{15}x_4^{7} + x_1x_2^{15}x_4^{7}x_5^{7} + x_1x_2^{15}x_3^{7}x_5^{7} + x_1x_2^{15}x_3^{7}x_4^{7} + x_1^{7}x_3x_4^{7}x_5^{15}\\ 
&\quad + x_1^{7}x_3x_4^{15}x_5^{7} + x_1^{7}x_3^{7}x_4x_5^{15} + x_1^{7}x_3^{7}x_4^{15}x_5 + x_1^{7}x_3^{15}x_4x_5^{7} + x_1^{7}x_3^{15}x_4^{7}x_5\\ 
&\quad + x_1^{7}x_2x_4^{7}x_5^{15} + x_1^{7}x_2x_4^{15}x_5^{7} + x_1^{7}x_2x_3^{7}x_5^{15} + x_1^{7}x_2x_3^{7}x_4^{15} + x_1^{7}x_2x_3^{15}x_5^{7}\\ 
&\quad + x_1^{7}x_2x_3^{15}x_4^{7} + x_1^{7}x_2^{7}x_4x_5^{15} + x_1^{7}x_2^{7}x_4^{15}x_5 + x_1^{7}x_2^{7}x_3x_5^{15} + x_1^{7}x_2^{7}x_3x_4^{15}\\ 
&\quad + x_1^{7}x_2^{7}x_3^{15}x_5 + x_1^{7}x_2^{7}x_3^{15}x_4 + x_1^{7}x_2^{15}x_4x_5^{7} + x_1^{7}x_2^{15}x_4^{7}x_5 + x_1^{7}x_2^{15}x_3x_5^{7}\\ 
&\quad + x_1^{7}x_2^{15}x_3x_4^{7} + x_1^{7}x_2^{15}x_3^{7}x_5 + x_1^{7}x_2^{15}x_3^{7}x_4 + x_1^{15}x_3x_4^{7}x_5^{7} + x_1^{15}x_3^{7}x_4x_5^{7}\\ 
&\quad + x_1^{15}x_3^{7}x_4^{7}x_5 + x_1^{15}x_2x_4^{7}x_5^{7} + x_1^{15}x_2x_3^{7}x_5^{7} + x_1^{15}x_2x_3^{7}x_4^{7} + x_1^{15}x_2^{7}x_4x_5^{7}\\ 
&\quad + x_1^{15}x_2^{7}x_4^{7}x_5 + x_1^{15}x_2^{7}x_3x_5^{7} + x_1^{15}x_2^{7}x_3x_4^{7} + x_1^{15}x_2^{7}x_3^{7}x_5 + x_1^{15}x_2^{7}x_3^{7}x_4,\\
q_2 &= x_2^{3}x_3^{5}x_4^{7}x_5^{15} + x_2^{3}x_3^{5}x_4^{15}x_5^{7} + x_2^{3}x_3^{7}x_4^{5}x_5^{15} + x_2^{3}x_3^{7}x_4^{15}x_5^{5} + x_2^{3}x_3^{15}x_4^{5}x_5^{7}\\ 
&\quad + x_2^{3}x_3^{15}x_4^{7}x_5^{5} + x_2^{7}x_3^{3}x_4^{5}x_5^{15} + x_2^{7}x_3^{3}x_4^{15}x_5^{5} + x_2^{7}x_3^{15}x_4^{3}x_5^{5} + x_2^{15}x_3^{3}x_4^{5}x_5^{7}\\ 
&\quad + x_2^{15}x_3^{3}x_4^{7}x_5^{5} + x_2^{15}x_3^{7}x_4^{3}x_5^{5} + x_1^{3}x_3^{5}x_4^{7}x_5^{15} + x_1^{3}x_3^{5}x_4^{15}x_5^{7} + x_1^{3}x_3^{7}x_4^{5}x_5^{15}\\ 
&\quad + x_1^{3}x_3^{7}x_4^{15}x_5^{5} + x_1^{3}x_3^{15}x_4^{5}x_5^{7} + x_1^{3}x_3^{15}x_4^{7}x_5^{5} + x_1^{3}x_2^{5}x_4^{7}x_5^{15} + x_1^{3}x_2^{5}x_4^{15}x_5^{7}\\ 
&\quad + x_1^{3}x_2^{5}x_3^{7}x_5^{15} + x_1^{3}x_2^{5}x_3^{7}x_4^{15} + x_1^{3}x_2^{5}x_3^{15}x_5^{7} + x_1^{3}x_2^{5}x_3^{15}x_4^{7} + x_1^{3}x_2^{7}x_4^{5}x_5^{15}\\ 
&\quad + x_1^{3}x_2^{7}x_4^{15}x_5^{5} + x_1^{3}x_2^{7}x_3^{5}x_5^{15} + x_1^{3}x_2^{7}x_3^{5}x_4^{15} + x_1^{3}x_2^{7}x_3^{15}x_5^{5} + x_1^{3}x_2^{7}x_3^{15}x_4^{5}\\ 
&\quad + x_1^{3}x_2^{15}x_4^{5}x_5^{7} + x_1^{3}x_2^{15}x_4^{7}x_5^{5} + x_1^{3}x_2^{15}x_3^{5}x_5^{7} + x_1^{3}x_2^{15}x_3^{5}x_4^{7} + x_1^{3}x_2^{15}x_3^{7}x_5^{5}\\ 
&\quad + x_1^{3}x_2^{15}x_3^{7}x_4^{5} + x_1^{7}x_3^{3}x_4^{5}x_5^{15} + x_1^{7}x_3^{3}x_4^{15}x_5^{5} + x_1^{7}x_3^{15}x_4^{3}x_5^{5} + x_1^{7}x_2^{3}x_4^{5}x_5^{15}\\ 
&\quad + x_1^{7}x_2^{3}x_4^{15}x_5^{5} + x_1^{7}x_2^{3}x_3^{5}x_5^{15} + x_1^{7}x_2^{3}x_3^{5}x_4^{15} + x_1^{7}x_2^{3}x_3^{15}x_5^{5} + x_1^{7}x_2^{3}x_3^{15}x_4^{5}\\ 
&\quad + x_1^{7}x_2^{15}x_4^{3}x_5^{5} + x_1^{7}x_2^{15}x_3^{3}x_5^{5} + x_1^{7}x_2^{15}x_3^{3}x_4^{5} + x_1^{15}x_3^{3}x_4^{5}x_5^{7} + x_1^{15}x_3^{3}x_4^{7}x_5^{5}\\ 
&\quad + x_1^{15}x_3^{7}x_4^{3}x_5^{5} + x_1^{15}x_2^{3}x_4^{5}x_5^{7} + x_1^{15}x_2^{3}x_4^{7}x_5^{5} + x_1^{15}x_2^{3}x_3^{5}x_5^{7} + x_1^{15}x_2^{3}x_3^{5}x_4^{7}\\ 
&\quad + x_1^{15}x_2^{3}x_3^{7}x_5^{5} + x_1^{15}x_2^{3}x_3^{7}x_4^{5} + x_1^{15}x_2^{7}x_4^{3}x_5^{5} + x_1^{15}x_2^{7}x_3^{3}x_5^{5} + x_1^{15}x_2^{7}x_3^{3}x_4^{5},\\
q_3 &= x_1x_2x_3^{6}x_4^{7}x_5^{15} + x_1x_2x_3^{6}x_4^{15}x_5^{7} + x_1x_2x_3^{7}x_4^{6}x_5^{15} + x_1x_2x_3^{7}x_4^{15}x_5^{6}\\ 
&\quad + x_1x_2x_3^{15}x_4^{6}x_5^{7} + x_1x_2x_3^{15}x_4^{7}x_5^{6} + x_1x_2^{6}x_3x_4^{7}x_5^{15} + x_1x_2^{6}x_3x_4^{15}x_5^{7}\\ 
&\quad + x_1x_2^{6}x_3^{7}x_4x_5^{15} + x_1x_2^{6}x_3^{7}x_4^{15}x_5 + x_1x_2^{6}x_3^{15}x_4x_5^{7} + x_1x_2^{6}x_3^{15}x_4^{7}x_5\\ 
&\quad + x_1x_2^{7}x_3x_4^{6}x_5^{15} + x_1x_2^{7}x_3x_4^{15}x_5^{6} + x_1x_2^{7}x_3^{6}x_4x_5^{15} + x_1x_2^{7}x_3^{6}x_4^{15}x_5\\ 
&\quad + x_1x_2^{7}x_3^{15}x_4x_5^{6} + x_1x_2^{7}x_3^{15}x_4^{6}x_5 + x_1x_2^{15}x_3x_4^{6}x_5^{7} + x_1x_2^{15}x_3x_4^{7}x_5^{6}\\ 
&\quad + x_1x_2^{15}x_3^{6}x_4x_5^{7} + x_1x_2^{15}x_3^{6}x_4^{7}x_5 + x_1x_2^{15}x_3^{7}x_4x_5^{6} + x_1x_2^{15}x_3^{7}x_4^{6}x_5\\ 
&\quad + x_1^{7}x_2x_3x_4^{6}x_5^{15} + x_1^{7}x_2x_3x_4^{15}x_5^{6} + x_1^{7}x_2x_3^{6}x_4x_5^{15} + x_1^{7}x_2x_3^{6}x_4^{15}x_5\\ 
&\quad + x_1^{7}x_2x_3^{15}x_4x_5^{6} + x_1^{7}x_2x_3^{15}x_4^{6}x_5 + x_1^{7}x_2^{15}x_3x_4x_5^{6} + x_1^{7}x_2^{15}x_3x_4^{6}x_5\\ 
&\quad + x_1^{15}x_2x_3x_4^{6}x_5^{7} + x_1^{15}x_2x_3x_4^{7}x_5^{6} + x_1^{15}x_2x_3^{6}x_4x_5^{7} + x_1^{15}x_2x_3^{6}x_4^{7}x_5\\ 
&\quad + x_1^{15}x_2x_3^{7}x_4x_5^{6} + x_1^{15}x_2x_3^{7}x_4^{6}x_5 + x_1^{15}x_2^{7}x_3x_4x_5^{6} + x_1^{15}x_2^{7}x_3x_4^{6}x_5\\ 
&\quad + x_1^{3}x_2x_3^{4}x_4^{7}x_5^{15} + x_1^{3}x_2x_3^{4}x_4^{15}x_5^{7} + x_1^{3}x_2x_3^{7}x_4^{4}x_5^{15} + x_1^{3}x_2x_3^{7}x_4^{15}x_5^{4}\\ 
&\quad + x_1^{3}x_2x_3^{15}x_4^{4}x_5^{7} + x_1^{3}x_2x_3^{15}x_4^{7}x_5^{4} + x_1^{3}x_2^{4}x_3x_4^{7}x_5^{15} + x_1^{3}x_2^{4}x_3x_4^{15}x_5^{7}\\ 
&\quad + x_1^{3}x_2^{4}x_3^{7}x_4x_5^{15} + x_1^{3}x_2^{4}x_3^{7}x_4^{15}x_5 + x_1^{3}x_2^{4}x_3^{15}x_4x_5^{7} + x_1^{3}x_2^{4}x_3^{15}x_4^{7}x_5\\ 
&\quad + x_1^{3}x_2^{7}x_3x_4^{4}x_5^{15} + x_1^{3}x_2^{7}x_3x_4^{15}x_5^{4} + x_1^{3}x_2^{7}x_3^{4}x_4x_5^{15} + x_1^{3}x_2^{7}x_3^{4}x_4^{15}x_5\\ 
&\quad + x_1^{3}x_2^{7}x_3^{15}x_4x_5^{4} + x_1^{3}x_2^{7}x_3^{15}x_4^{4}x_5 + x_1^{3}x_2^{15}x_3x_4^{4}x_5^{7} + x_1^{3}x_2^{15}x_3x_4^{7}x_5^{4}\\ 
&\quad + x_1^{3}x_2^{15}x_3^{4}x_4x_5^{7} + x_1^{3}x_2^{15}x_3^{4}x_4^{7}x_5 + x_1^{3}x_2^{15}x_3^{7}x_4x_5^{4} + x_1^{3}x_2^{15}x_3^{7}x_4^{4}x_5\\ 
&\quad + x_1^{7}x_2^{3}x_3x_4^{4}x_5^{15} + x_1^{7}x_2^{3}x_3x_4^{15}x_5^{4} + x_1^{7}x_2^{3}x_3^{4}x_4x_5^{15} + x_1^{7}x_2^{3}x_3^{4}x_4^{15}x_5\\ 
&\quad + x_1^{7}x_2^{3}x_3^{15}x_4x_5^{4} + x_1^{7}x_2^{3}x_3^{15}x_4^{4}x_5 + x_1^{7}x_2^{15}x_3^{3}x_4x_5^{4} + x_1^{7}x_2^{15}x_3^{3}x_4^{4}x_5\\ 
&\quad + x_1^{15}x_2^{3}x_3x_4^{4}x_5^{7} + x_1^{15}x_2^{3}x_3x_4^{7}x_5^{4} + x_1^{15}x_2^{3}x_3^{4}x_4x_5^{7} + x_1^{15}x_2^{3}x_3^{4}x_4^{7}x_5\\ 
&\quad + x_1^{15}x_2^{3}x_3^{7}x_4x_5^{4} + x_1^{15}x_2^{3}x_3^{7}x_4^{4}x_5 + x_1^{15}x_2^{7}x_3^{3}x_4x_5^{4} + x_1^{15}x_2^{7}x_3^{3}x_4^{4}x_5,\\
q_4 &= x_1x_2^{3}x_3^{5}x_4^{6}x_5^{15} + x_1x_2^{3}x_3^{5}x_4^{15}x_5^{6} + x_1x_2^{3}x_3^{6}x_4^{5}x_5^{15} + x_1x_2^{3}x_3^{6}x_4^{15}x_5^{5}\\ 
&\quad + x_1x_2^{3}x_3^{15}x_4^{5}x_5^{6} + x_1x_2^{3}x_3^{15}x_4^{6}x_5^{5} + x_1x_2^{6}x_3^{3}x_4^{5}x_5^{15} + x_1x_2^{6}x_3^{3}x_4^{15}x_5^{5}\\ 
&\quad + x_1x_2^{6}x_3^{15}x_4^{3}x_5^{5} + x_1x_2^{15}x_3^{3}x_4^{5}x_5^{6} + x_1x_2^{15}x_3^{3}x_4^{6}x_5^{5} + x_1x_2^{15}x_3^{6}x_4^{3}x_5^{5}\\ 
&\quad + x_1^{3}x_2x_3^{5}x_4^{6}x_5^{15} + x_1^{3}x_2x_3^{5}x_4^{15}x_5^{6} + x_1^{3}x_2x_3^{15}x_4^{5}x_5^{6} + x_1^{3}x_2^{5}x_3x_4^{6}x_5^{15}\\ 
&\quad + x_1^{3}x_2^{5}x_3x_4^{15}x_5^{6} + x_1^{3}x_2^{5}x_3^{6}x_4x_5^{15} + x_1^{3}x_2^{5}x_3^{6}x_4^{15}x_5 + x_1^{3}x_2^{5}x_3^{15}x_4x_5^{6}\\ 
&\quad + x_1^{3}x_2^{5}x_3^{15}x_4^{6}x_5 + x_1^{3}x_2^{15}x_3x_4^{5}x_5^{6} + x_1^{3}x_2^{15}x_3^{5}x_4x_5^{6} + x_1^{3}x_2^{15}x_3^{5}x_4^{6}x_5\\ 
&\quad + x_1^{15}x_2x_3^{3}x_4^{5}x_5^{6} + x_1^{15}x_2x_3^{3}x_4^{6}x_5^{5} + x_1^{15}x_2x_3^{6}x_4^{3}x_5^{5} + x_1^{15}x_2^{3}x_3x_4^{5}x_5^{6}\\ 
&\quad + x_1^{15}x_2^{3}x_3^{5}x_4x_5^{6} + x_1^{15}x_2^{3}x_3^{5}x_4^{6}x_5 + x_1^{3}x_2^{3}x_3^{4}x_4^{5}x_5^{15} + x_1^{3}x_2^{3}x_3^{4}x_4^{15}x_5^{5}\\ 
&\quad + x_1^{3}x_2^{3}x_3^{5}x_4^{4}x_5^{15} + x_1^{3}x_2^{3}x_3^{5}x_4^{15}x_5^{4} + x_1^{3}x_2^{3}x_3^{15}x_4^{4}x_5^{5} + x_1^{3}x_2^{3}x_3^{15}x_4^{5}x_5^{4}\\ 
&\quad + x_1^{3}x_2^{4}x_3^{3}x_4^{5}x_5^{15} + x_1^{3}x_2^{4}x_3^{3}x_4^{15}x_5^{5} + x_1^{3}x_2^{4}x_3^{15}x_4^{3}x_5^{5} + x_1^{3}x_2^{15}x_3^{3}x_4^{4}x_5^{5}\\ 
&\quad + x_1^{3}x_2^{15}x_3^{3}x_4^{5}x_5^{4} + x_1^{3}x_2^{15}x_3^{4}x_4^{3}x_5^{5} + x_1^{15}x_2^{3}x_3^{3}x_4^{4}x_5^{5} + x_1^{15}x_2^{3}x_3^{3}x_4^{5}x_5^{4}\\ 
&\quad + x_1^{15}x_2^{3}x_3^{4}x_4^{3}x_5^{5},\\
q_5 &= x_1x_2^{3}x_3^{5}x_4^{7}x_5^{14} + x_1x_2^{3}x_3^{5}x_4^{14}x_5^{7} + x_1x_2^{3}x_3^{6}x_4^{7}x_5^{13} + x_1x_2^{3}x_3^{6}x_4^{13}x_5^{7}\\ 
&\quad + x_1x_2^{6}x_3^{3}x_4^{7}x_5^{13} + x_1x_2^{6}x_3^{3}x_4^{13}x_5^{7} + x_1x_2^{6}x_3^{7}x_4^{7}x_5^{9} + x_1x_2^{6}x_3^{7}x_4^{9}x_5^{7}\\ 
&\quad + x_1x_2^{7}x_3^{6}x_4^{7}x_5^{9} + x_1x_2^{7}x_3^{6}x_4^{9}x_5^{7} + x_1x_2^{7}x_3^{7}x_4^{7}x_5^{8} + x_1x_2^{7}x_3^{7}x_4^{8}x_5^{7}\\ 
&\quad + x_1^{3}x_2x_3^{5}x_4^{7}x_5^{14} + x_1^{3}x_2x_3^{5}x_4^{14}x_5^{7} + x_1^{3}x_2x_3^{7}x_4^{7}x_5^{12} + x_1^{3}x_2x_3^{7}x_4^{12}x_5^{7}\\ 
&\quad + x_1^{3}x_2^{3}x_3^{4}x_4^{7}x_5^{13} + x_1^{3}x_2^{3}x_3^{4}x_4^{13}x_5^{7} + x_1^{3}x_2^{3}x_3^{5}x_4^{7}x_5^{12} + x_1^{3}x_2^{3}x_3^{5}x_4^{12}x_5^{7}\\ 
&\quad + x_1^{3}x_2^{3}x_3^{7}x_4^{4}x_5^{13} + x_1^{3}x_2^{3}x_3^{7}x_4^{13}x_5^{4} + x_1^{3}x_2^{3}x_3^{13}x_4^{4}x_5^{7} + x_1^{3}x_2^{3}x_3^{13}x_4^{7}x_5^{4}\\ 
&\quad + x_1^{3}x_2^{4}x_3^{3}x_4^{7}x_5^{13} + x_1^{3}x_2^{4}x_3^{3}x_4^{13}x_5^{7} + x_1^{3}x_2^{4}x_3^{7}x_4^{7}x_5^{9} + x_1^{3}x_2^{4}x_3^{7}x_4^{9}x_5^{7}\\ 
&\quad + x_1^{3}x_2^{5}x_3x_4^{7}x_5^{14} + x_1^{3}x_2^{5}x_3x_4^{14}x_5^{7} + x_1^{3}x_2^{5}x_3^{6}x_4^{7}x_5^{9} + x_1^{3}x_2^{5}x_3^{6}x_4^{9}x_5^{7}\\ 
&\quad + x_1^{3}x_2^{5}x_3^{7}x_4x_5^{14} + x_1^{3}x_2^{5}x_3^{7}x_4^{7}x_5^{8} + x_1^{3}x_2^{5}x_3^{7}x_4^{8}x_5^{7} + x_1^{3}x_2^{5}x_3^{7}x_4^{14}x_5\\ 
&\quad + x_1^{3}x_2^{5}x_3^{14}x_4x_5^{7} + x_1^{3}x_2^{5}x_3^{14}x_4^{7}x_5 + x_1^{3}x_2^{7}x_3x_4^{7}x_5^{12} + x_1^{3}x_2^{7}x_3x_4^{12}x_5^{7}\\ 
&\quad + x_1^{3}x_2^{7}x_3^{4}x_4^{7}x_5^{9} + x_1^{3}x_2^{7}x_3^{4}x_4^{9}x_5^{7} + x_1^{3}x_2^{7}x_3^{7}x_4x_5^{12} + x_1^{3}x_2^{7}x_3^{7}x_4^{4}x_5^{9}\\ 
&\quad + x_1^{3}x_2^{7}x_3^{7}x_4^{9}x_5^{4} + x_1^{3}x_2^{7}x_3^{7}x_4^{12}x_5 + x_1^{3}x_2^{7}x_3^{9}x_4^{4}x_5^{7} + x_1^{3}x_2^{7}x_3^{9}x_4^{7}x_5^{4}\\ 
&\quad + x_1^{3}x_2^{7}x_3^{12}x_4x_5^{7} + x_1^{3}x_2^{7}x_3^{12}x_4^{7}x_5 + x_1^{7}x_2x_3^{6}x_4^{7}x_5^{9} + x_1^{7}x_2x_3^{6}x_4^{9}x_5^{7}\\ 
&\quad + x_1^{7}x_2x_3^{7}x_4^{7}x_5^{8} + x_1^{7}x_2x_3^{7}x_4^{8}x_5^{7} + x_1^{7}x_2^{3}x_3x_4^{7}x_5^{12} + x_1^{7}x_2^{3}x_3x_4^{12}x_5^{7}\\ 
&\quad + x_1^{7}x_2^{3}x_3^{4}x_4^{7}x_5^{9} + x_1^{7}x_2^{3}x_3^{4}x_4^{9}x_5^{7} + x_1^{7}x_2^{3}x_3^{7}x_4x_5^{12} + x_1^{7}x_2^{3}x_3^{7}x_4^{4}x_5^{9}\\ 
&\quad + x_1^{7}x_2^{3}x_3^{7}x_4^{9}x_5^{4} + x_1^{7}x_2^{3}x_3^{7}x_4^{12}x_5 + x_1^{7}x_2^{3}x_3^{9}x_4^{4}x_5^{7} + x_1^{7}x_2^{3}x_3^{9}x_4^{7}x_5^{4}\\ 
&\quad + x_1^{7}x_2^{3}x_3^{12}x_4x_5^{7} + x_1^{7}x_2^{3}x_3^{12}x_4^{7}x_5 + x_1^{7}x_2^{7}x_3x_4^{7}x_5^{8} + x_1^{7}x_2^{7}x_3x_4^{8}x_5^{7}\\ 
&\quad + x_1^{7}x_2^{7}x_3^{7}x_4x_5^{8} + x_1^{7}x_2^{7}x_3^{7}x_4^{8}x_5 + x_1^{7}x_2^{7}x_3^{8}x_4x_5^{7} + x_1^{7}x_2^{7}x_3^{8}x_4^{7}x_5,\\
q_6 & = x_1x_2^{3}x_3^{6}x_4^{7}x_5^{13} + x_1x_2^{3}x_3^{6}x_4^{13}x_5^{7} + x_1x_2^{3}x_3^{7}x_4^{7}x_5^{12} + x_1x_2^{3}x_3^{7}x_4^{12}x_5^{7}\\ 
&\quad + x_1x_2^{6}x_3^{3}x_4^{7}x_5^{13} + x_1x_2^{6}x_3^{3}x_4^{13}x_5^{7} + x_1x_2^{6}x_3^{7}x_4^{3}x_5^{13} + x_1x_2^{6}x_3^{7}x_4^{11}x_5^{5}\\ 
&\quad + x_1x_2^{6}x_3^{11}x_4^{5}x_5^{7} + x_1x_2^{6}x_3^{11}x_4^{7}x_5^{5} + x_1x_2^{7}x_3^{3}x_4^{7}x_5^{12} + x_1x_2^{7}x_3^{3}x_4^{12}x_5^{7}\\ 
&\quad + x_1x_2^{7}x_3^{6}x_4^{7}x_5^{9} + x_1x_2^{7}x_3^{6}x_4^{9}x_5^{7} + x_1x_2^{7}x_3^{7}x_4^{3}x_5^{12} + x_1x_2^{7}x_3^{7}x_4^{7}x_5^{8}\\ 
&\quad + x_1x_2^{7}x_3^{7}x_4^{8}x_5^{7} + x_1x_2^{7}x_3^{7}x_4^{10}x_5^{5} + x_1x_2^{7}x_3^{10}x_4^{5}x_5^{7} + x_1x_2^{7}x_3^{10}x_4^{7}x_5^{5}\\ 
&\quad + x_1^{3}x_2x_3^{7}x_4^{6}x_5^{13} + x_1^{3}x_2x_3^{7}x_4^{7}x_5^{12} + x_1^{3}x_2x_3^{7}x_4^{12}x_5^{7} + x_1^{3}x_2x_3^{7}x_4^{14}x_5^{5}\\ 
&\quad + x_1^{3}x_2x_3^{14}x_4^{5}x_5^{7} + x_1^{3}x_2x_3^{14}x_4^{7}x_5^{5} + x_1^{3}x_2^{3}x_3^{4}x_4^{7}x_5^{13} + x_1^{3}x_2^{3}x_3^{4}x_4^{13}x_5^{7}\\ 
&\quad + x_1^{3}x_2^{3}x_3^{7}x_4^{4}x_5^{13} + x_1^{3}x_2^{3}x_3^{7}x_4^{13}x_5^{4} + x_1^{3}x_2^{3}x_3^{13}x_4^{4}x_5^{7} + x_1^{3}x_2^{3}x_3^{13}x_4^{7}x_5^{4}\\ 
&\quad + x_1^{3}x_2^{4}x_3^{3}x_4^{7}x_5^{13} + x_1^{3}x_2^{4}x_3^{3}x_4^{13}x_5^{7} + x_1^{3}x_2^{4}x_3^{7}x_4^{3}x_5^{13} + x_1^{3}x_2^{4}x_3^{7}x_4^{11}x_5^{5}\\ 
&\quad + x_1^{3}x_2^{4}x_3^{11}x_4^{5}x_5^{7} + x_1^{3}x_2^{4}x_3^{11}x_4^{7}x_5^{5} + x_1^{3}x_2^{7}x_3x_4^{5}x_5^{14} + x_1^{3}x_2^{7}x_3x_4^{6}x_5^{13}\\ 
&\quad + x_1^{3}x_2^{7}x_3x_4^{13}x_5^{6} + x_1^{3}x_2^{7}x_3x_4^{14}x_5^{5} + x_1^{3}x_2^{7}x_3^{3}x_4^{4}x_5^{13} + x_1^{3}x_2^{7}x_3^{3}x_4^{13}x_5^{4}\\ 
&\quad + x_1^{3}x_2^{7}x_3^{4}x_4^{7}x_5^{9} + x_1^{3}x_2^{7}x_3^{4}x_4^{9}x_5^{7} + x_1^{3}x_2^{7}x_3^{7}x_4^{4}x_5^{9} + x_1^{3}x_2^{7}x_3^{7}x_4^{8}x_5^{5}\\ 
&\quad + x_1^{3}x_2^{7}x_3^{8}x_4^{5}x_5^{7} + x_1^{3}x_2^{7}x_3^{8}x_4^{7}x_5^{5} + x_1^{3}x_2^{7}x_3^{13}x_4^{2}x_5^{5} + x_1^{3}x_2^{7}x_3^{13}x_4^{3}x_5^{4}\\ 
&\quad + x_1^{3}x_2^{13}x_3^{2}x_4^{5}x_5^{7} + x_1^{3}x_2^{13}x_3^{2}x_4^{7}x_5^{5} + x_1^{3}x_2^{13}x_3^{3}x_4^{4}x_5^{7} + x_1^{3}x_2^{13}x_3^{3}x_4^{7}x_5^{4}\\ 
&\quad + x_1^{3}x_2^{13}x_3^{7}x_4^{2}x_5^{5} + x_1^{3}x_2^{13}x_3^{7}x_4^{3}x_5^{4} + x_1^{7}x_2x_3^{3}x_4^{7}x_5^{12} + x_1^{7}x_2x_3^{3}x_4^{12}x_5^{7}\\ 
&\quad + x_1^{7}x_2x_3^{6}x_4^{7}x_5^{9} + x_1^{7}x_2x_3^{6}x_4^{9}x_5^{7} + x_1^{7}x_2x_3^{7}x_4^{3}x_5^{12} + x_1^{7}x_2x_3^{7}x_4^{7}x_5^{8}\\ 
&\quad + x_1^{7}x_2x_3^{7}x_4^{8}x_5^{7} + x_1^{7}x_2x_3^{7}x_4^{10}x_5^{5} + x_1^{7}x_2x_3^{10}x_4^{5}x_5^{7} + x_1^{7}x_2x_3^{10}x_4^{7}x_5^{5}\\ 
&\quad + x_1^{7}x_2^{3}x_3^{4}x_4^{7}x_5^{9} + x_1^{7}x_2^{3}x_3^{4}x_4^{9}x_5^{7} + x_1^{7}x_2^{3}x_3^{7}x_4^{4}x_5^{9} + x_1^{7}x_2^{3}x_3^{7}x_4^{9}x_5^{4}\\ 
&\quad + x_1^{7}x_2^{3}x_3^{9}x_4^{4}x_5^{7} + x_1^{7}x_2^{3}x_3^{9}x_4^{7}x_5^{4} + x_1^{7}x_2^{7}x_3x_4^{3}x_5^{12} + x_1^{7}x_2^{7}x_3x_4^{6}x_5^{9}\\ 
&\quad + x_1^{7}x_2^{7}x_3x_4^{9}x_5^{6} + x_1^{7}x_2^{7}x_3x_4^{10}x_5^{5} + x_1^{7}x_2^{7}x_3^{3}x_4^{4}x_5^{9} + x_1^{7}x_2^{7}x_3^{3}x_4^{9}x_5^{4}\\ 
&\quad + x_1^{7}x_2^{7}x_3^{9}x_4^{2}x_5^{5} + x_1^{7}x_2^{7}x_3^{9}x_4^{3}x_5^{4} + x_1^{7}x_2^{9}x_3^{2}x_4^{5}x_5^{7} + x_1^{7}x_2^{9}x_3^{2}x_4^{7}x_5^{5}\\ 
&\quad + x_1^{7}x_2^{9}x_3^{3}x_4^{4}x_5^{7} + x_1^{7}x_2^{9}x_3^{3}x_4^{7}x_5^{4} + x_1^{7}x_2^{9}x_3^{7}x_4^{2}x_5^{5} + x_1^{7}x_2^{9}x_3^{7}x_4^{3}x_5^{4},\\
q_7 &= x_1^{3}x_2^{3}x_3^{5}x_4^{6}x_5^{13} + x_1^{3}x_2^{3}x_3^{5}x_4^{14}x_5^{5} + x_1^{3}x_2^{3}x_3^{13}x_4^{6}x_5^{5} + x_1^{3}x_2^{5}x_3^{3}x_4^{5}x_5^{14}\\ 
&\quad + x_1^{3}x_2^{5}x_3^{3}x_4^{6}x_5^{13} + x_1^{3}x_2^{5}x_3^{3}x_4^{13}x_5^{6} + x_1^{3}x_2^{5}x_3^{3}x_4^{14}x_5^{5} + x_1^{3}x_2^{5}x_3^{6}x_4^{3}x_5^{13}\\ 
&\quad + x_1^{3}x_2^{5}x_3^{6}x_4^{11}x_5^{5} + x_1^{3}x_2^{5}x_3^{11}x_4^{5}x_5^{6} + x_1^{3}x_2^{5}x_3^{11}x_4^{6}x_5^{5} + x_1^{3}x_2^{5}x_3^{14}x_4^{3}x_5^{5}\\ 
&\quad + x_1^{3}x_2^{7}x_3^{3}x_4^{5}x_5^{12} + x_1^{3}x_2^{7}x_3^{3}x_4^{12}x_5^{5} + x_1^{3}x_2^{7}x_3^{5}x_4^{6}x_5^{9} + x_1^{3}x_2^{7}x_3^{5}x_4^{9}x_5^{6}\\ 
&\quad + x_1^{3}x_2^{7}x_3^{9}x_4^{5}x_5^{6} + x_1^{3}x_2^{7}x_3^{12}x_4^{3}x_5^{5} + x_1^{7}x_2^{3}x_3^{3}x_4^{5}x_5^{12} + x_1^{7}x_2^{3}x_3^{3}x_4^{12}x_5^{5}\\ 
&\quad + x_1^{7}x_2^{3}x_3^{5}x_4^{6}x_5^{9} + x_1^{7}x_2^{3}x_3^{5}x_4^{9}x_5^{6} + x_1^{7}x_2^{3}x_3^{9}x_4^{5}x_5^{6} + x_1^{7}x_2^{3}x_3^{12}x_4^{3}x_5^{5}\\ 
&\quad + x_1^{7}x_2^{7}x_3^{3}x_4^{5}x_5^{8} + x_1^{7}x_2^{7}x_3^{3}x_4^{8}x_5^{5} + x_1^{7}x_2^{7}x_3^{8}x_4^{3}x_5^{5}.
\end{align*}

\end{document}